\documentclass[12pt]{amsart}

\usepackage[margin=1.15in]{geometry}
\usepackage{amscd,amssymb, amsmath,amsfonts, wasysym, mathrsfs, enumerate, mathtools,hhline,color}
\usepackage{graphicx}
\usepackage[all, cmtip]{xy}

\usepackage{url}

\definecolor{hot}{RGB}{65,105,225}

\usepackage[pagebackref=true,colorlinks=true, linkcolor=hot ,  citecolor=hot, urlcolor=hot]{hyperref}

\theoremstyle{plain}
\newtheorem{theorem}{Theorem}[section]
\newtheorem{proposition}[theorem]{Proposition}

\newtheorem{corollary}[theorem]{Corollary}
\newtheorem{conjecture}[theorem]{Conjecture}
\newtheorem{lemma}[theorem]{Lemma}

\theoremstyle{definition}
\newtheorem{definition}[theorem]{Definition}
\theoremstyle{question}
\newtheorem{question}[theorem]{Question}
\theoremstyle{remark}
\newtheorem{remark}[theorem]{Remark}

\theoremstyle{assumption}
\newtheorem{assumption}[theorem]{Assumption}
\newtheorem{example}[theorem]{Example}
\newtheorem*{ex*}{Example}

\numberwithin{equation}{section}

\theoremstyle{bigthm}
\newtheorem{bigthm}{Theorem}
\newtheorem{bigcor}[bigthm]{Corollary}

\title[Self-Covering]{Self-Covering, Finiteness, Commutativity, And Fibering Over Tori}

\begin{document}

\author{Lizhen Qin}
\address{School of Mathematics, Nanjing University, 22 Hankou Road, Nanjing, Jiangsu 210093, China}
\email{qinlz@nju.edu.cn}

\author{Yang Su}
\address{HLM, Academy of Mathematics and Systems Science, Chinese Academy of Sciences, Beijing 100190, China}
\address{School of Mathematical Sciences, University of Chinese Academy of Sciences, Beijing 100049, China}
\email{suyang@math.ac.cn}


\maketitle

\begin{abstract}
A topological space is called self-covering if it is a nontrivial cover of itself. We prove that, under mild assumptions, a closed self-covering manifold with an abelian fundamental group fibers over a torus in various senses. As a corollary, if its dimension is above $5$ and its fundamental group is free abelian, then it is a fiber bundle over a circle. We also construct non-fibering examples when these assumptions are not fulfilled. In particular, one class of examples illustrates that the structure of self-covering manifolds is more complicated when the fundamental groups are nonabelian, and the corresponding fibering problem encounters significant difficulties.
\end{abstract}

\section{Introduction}\label{sec_introduction}
A topological space is called self-covering if it is homeomorphic, or more generally homotopy equivalent, to a nontrivial covering space of itself. Self-covering spaces, especially self-covering manifolds, are well-studied in the literature (see e.g. \cite{BDT}, \cite{Dekimpe}, \cite{Dere17}, \cite{VanLimbeek} and \cite{Wang_Wu}).

The simplest self-covering closed manifolds are tori $\mathbb{T}^{n} := (S^{1})^{n}$ with a self-covering map $A_{\mathbb{T}} \colon  \mathbb{T}^{n} \rightarrow \mathbb{T}^{n}$ given by a matrix $A \in \mathrm{End} (\mathbb{Z}^{n})$ with $\det (A) \ne 0$ and $\bigcap_{k=1}^{\infty} \mathrm{im} A^{k} =0$. Typical examples of self-covering manifolds are constructed by taking the product of $\mathbb{T}^n$ with a closed manifold $N$, with a self-covering map $h = A_{\mathbb{T}} \times \mathrm{id}_N$. These examples indicate that, to elucidate the structure of self-covering manifolds with abelian fundamental groups, it is natural to fiber such a manifold either over a torus or with a torus fiber. In this framework, the self-covering property of the manifold is expected to arise directly from a self-covering of the torus.

In \cite{Qin_Su_Wang}, Qin-Su-Wang studied when a closed self-covering manifold with an abelian fundamental group is a fiber bundle over a circle. As an extension of the previous work, this paper tries to fiber such a manifold over a torus. We work in the category of smooth (DIFF), piecewise-linear (PL) and topological (TOP) manifolds simultaneously. Throughout this paper, the abbreviation CAT stands for any of these categories. All topological spaces are assumed path-connected unless otherwise stated. Motivated by \cite{Qin_Su_Wang}, we ask the following question.

\begin{question}\label{que_abelian}
Suppose $M$ is a closed $\mathrm{CAT}$ manifold with abelian $\pi_{1} (M)$, $q \colon  M' \rightarrow M$ is a nontrivial covering, and $h \colon  M \rightarrow M'$ is a homotopy equivalence.

Are there a $\mathrm{CAT}$ (fiber) bundle projection $p \colon  M \rightarrow \mathbb{T}^{n}$ and a linear endomorphism $A_{\mathbb{T}}$ of $\mathbb{T}^{n}$ such that the following diagram is a homotopy pullback?
\begin{equation}\label{que_abelian_1}
\xymatrix{
  M \ar[d]_{p} \ar[r]^{q \circ h} & M \ar[d]^{p} \\
  \mathbb{T}^{n} \ar[r]^{A_{\mathbb{T}}} & \mathbb{T}^{n}   }
\end{equation}
Here $A_{\mathbb{T}}$ is required to be defined by a matrix $A \in \mathrm{End} (\mathbb{Z}^{n})$ such that $\det (A) \neq 0$ and $\bigcap_{k=1}^{\infty} \mathrm{im} A^{k} =0$.
\end{question}

Recall that by a homotopy pullback we mean \eqref{que_abelian_1} commutes up to homotopy and satisfies a certain universal property (see \cite[p.~205]{Arkowitz}). If the answer to Question \ref{que_abelian} was affirmative, self-covering phenomenon would be well understood. For closed manifolds with abelian fundamental groups, all self-covering maps essentially would be base changes of fiber bundles resulted from linear endomorphisms of tori. However, it is known that the answer is no in general. Theorems C and D in \cite{Qin_Su_Wang} provide non-fibering examples using obstructions from algebraic $K$-theory (Wall's finiteness obstruction and Whitehead torsion). In Theorem \ref{thm_nonfibering} below we provide another type of non-fibering examples using tangential data such as the Pontrjagin classes and the Kirby-Siebenmann invariant. Therefore, one usually cannot find such a bundle projection $p \colon  M \rightarrow \mathbb{T}^{n}$.

Nevertheless, the answer to Question \ref{que_abelian} is yes in a weaker sense. By Theorem \ref{thm_Poincare} below, one can always get a homotopy fibration $p \colon  M \rightarrow \mathbb{T}^{n}$ whose homotopy fiber is a finitely dominated Poincar\'{e} Duality Space, and $p$ fits \eqref{que_abelian_1} well. (Recall that a Poincar\'{e} Duality Space is an imitation of a closed manifold in the homotopy sense.) An interesting consequence is a characterization of topological tori in terms of self-covering properties (Corollary \ref{cor_torus}).

To get a positive solution to Question \ref{que_abelian} in its original sense, we shall try, under additional assumptions, to fiber $M$ in various senses, from the weak to the strong, between a homotopy fibration and a fiber bundle. These results are mainly formulated in Subsection \ref{subsec_high}. Some nontrivial obstructions will appear when we improve a weak solution step by step. They will vanish if certain extra assumptions are fulfilled. Eventually, a desired fiber bundle structure will be achieved.

It's certainly interesting to obtain some positive solutions to Question \ref{que_abelian} in the case that $\pi_{1} (M)$ is nonabelian. To make broad sense, the $\mathbb{T}^{n}$ in Question \ref{que_abelian} should be now replaced with a general closed aspherical manifold. This motivates us to propose Question \ref{que_nonabelian} below. Theorem \ref{thm_cw_manifold} and \ref{thm_cw_4-manifold} show that self-covering manifolds are closely related to self-covering complexes. This allows us to construct self-covering manifolds out of known examples of self-covering complexes with interesting fundamental groups (Theorem \ref{thm_BS_manifold}). These examples indicate that the case of nonabelian fundamental groups is much more complicated than the abelian case. The answer to Question \ref{que_nonabelian} is severely negative even in the weaker sense.

\subsection{CW Complexes}
Note that the covering space $\overline M$ of $M$ with $\pi_{1} (\overline{M}) = \ker p_{\#}$ is homotopy equivalent to the homotopy fiber of $p$ in \eqref{que_abelian}. Therefore a necessary condition for the map $p$ to be a bundle projection is that $\overline{M}$ is homotopy equivalent to a closed manifold. Thus, to obtain a positive solution to Question \ref{que_abelian}, we have to show the homotopy finiteness of certain infinite cover of $M$. This is a problem in homotopy theory. So we may work in the wider category of CW complexes.

Recall that a CW complex is of finite type if there are only finitely many cells in each dimension, whereas the dimension of the complex may be infinite; a space is finitely dominated if it is dominated by a finite CW complex.

Suppose $X$ is a CW complex, $q \colon  X' \rightarrow X$ a nontrivial covering, and $h \colon  X \rightarrow X'$ is a homotopy equivalence. Since $q_{\#} \colon  \pi_{1} (X') \rightarrow \pi_{1} (X)$ is a subgroup inclusion, for brevity, we abuse the notation $h_{\#}$ to denote the composition
\[
(qh)_{\#} \colon  \pi_{1} (X) \rightarrow \pi_{1} (X') \hookrightarrow \pi_{1} (X).
\]
Then $h_{\#} \colon  \pi_{1} (X) \rightarrow \pi_{1} (X)$ is a monomorphism whose image is a proper subgroup of $\pi_{1} (X)$. Let $\mathrm{im} h_{\#}^{k}$ denote the image of $h_{\#}^{k} \colon  \pi_{1} (X) \rightarrow \pi_{1} (X)$.

We obtain the following generalization of \cite[Theorem~A]{Qin_Su_Wang} on homotopy finiteness. The theorem serves the geometric foundation of this paper.
\begin{bigthm}\label{thm_homotopy_finite}
Let $X$, $X'$ and $h$ be as the above. Suppose further $\pi_{1} (X)$ is a finitely generated abelian group. Let $G := \bigcap_{k=1}^{\infty} \mathrm{im} h_{\#}^{k}$. Let $\overline{X}$ be the cover of $X$ with $\pi_{1} (\overline{X}) =G$. Then the following conclusion holds:
\begin{enumerate}
\item $\pi_{1} (X) / G \cong \mathbb{Z}^{n}$ for some $n>0$, and $X'$ is a finite cover of $X$.

\item If $X$ is homotopy equivalent to a CW complex of finite type (resp.~is finitely dominated), then so is $\overline{X}$.
\end{enumerate}
\end{bigthm}

\begin{remark}
As pointed in \cite[Remark~1.4]{Qin_Su_Wang}, if $X$ is homotopy equivalent to a finite CW complex, then unnecessarily so is $\overline{X}$ in Theorem \ref{thm_homotopy_finite}.
\end{remark}

\begin{remark}
By Theorem \ref{thm_BS_manifold}, if we drop the assumption that $\pi_{1} (X)$ is abelian in Theorem \ref{thm_homotopy_finite}, then its conclusion will be significantly wrong, see also Remark \ref{rmk_BS_nonfinite}.
\end{remark}

To get Theorem \ref{thm_homotopy_finite}, we shall actually prove a more general Theorem \ref{thm_homotopy_skeleton} on homotopy finiteness which follows from an even broader Theorem \ref{thm_homology_finite} on homological finiteness. As a byproduct, we have the following result in group theory.
\begin{bigcor}\label{cor_group}
Suppose $\pi$ is a finitely generated (unnecessarily abelian) group, and $\varphi \colon  \pi \rightarrow \pi$ is a monomorphism with $[\pi : \mathrm{im} \varphi] >1$. Suppose $G:= \bigcap_{k=1}^{\infty} \mathrm{im} \varphi^{k}$ is normal in $\pi$ and $\pi / G$ is abelian. Then $\pi / G\cong \mathbb{Z}^{n}$ for some $n>0$, $[\pi : \mathrm{im} \varphi] < \infty$, and the abelianization $G / [G,G]$ is finitely generated.
\end{bigcor}

\subsection{Poincar\'{e} Duality Spaces}
A Poincar\'{e} duality space, or Poincar\'e space for brevity, is a topological space satisfying Poincar\'{e} duality in the strictest sense, i.e.~there are Poincar\' e duality isomorphisms between its homology and cohomology groups for all local coefficient systems. Particularly, a closed TOP $m$-manifold is a Poincar\'e space with formal dimension $m$. Such a manifold is homotopy equivalent to a finite CW complex, let alone finitely dominated. Poincar\'e spaces are foundational to the surgery theory of manifolds. For the precise definition and basic properties of Poincar\'e spaces, see e.g.~\cite{Wall67} and \cite{Klein_Qin_Su}.

The following theorem gives a positive answer to Question \ref{que_abelian} in the homotopy sense which is weaker than the original sense.
\begin{bigthm}\label{thm_Poincare}
Suppose $X$ is a finitely dominated Poincar\'e space with formal dimension $d$, $\pi_{1} (X)$ is abelian, and $q \colon  X' \rightarrow X$ is a nontrivial covering. Suppose $h \colon  X \rightarrow X'$ is a homotopy equivalence. Let $G := \bigcap_{k=1}^{\infty} \mathrm{im} h_{\#}^{k}$. Let $\overline{X}$ be the cover of $X$ with $\pi_{1} (\overline{X}) =G$. Let $n := \mathrm{rank} (\pi_{1} (X) / G)$. Then the following statements hold:
\begin{enumerate}
\item $n>0$, $\pi_{1} (X) / G \cong \mathbb{Z}^{n}$, and $X'$ is a finite cover of $X$.

\item There exists a continuous map $p \colon  X \rightarrow \mathbb{T}^{n}$ which induces an epimorphism of fundamental groups and the homotopy fiber of $p$ is homotopy equivalent to $\overline{X}$.

\item $\overline{X}$ is also a finitely dominated Poincar\'{e} space with formal dimension $d-n$. The following diagram is a homotopy pullback.
\begin{equation}\label{thm_Poincare_1}
\xymatrix{
  X \ar[d]_{p} \ar[r]^{q \circ h} & X \ar[d]^{p} \\
  \mathbb{T}^{n} \ar[r]^{A_{\mathbb{T}}} & \mathbb{T}^{n}   }
\end{equation}
Here $A_{\mathbb{T}}$ is a linear endomorphism of $\mathbb{T}^{n}$ induced by a matrix $A \in \mathrm{End} (\mathbb{Z}^{n})$ such that $\det (A) \neq 0$ and $\bigcap_{k=1}^{\infty} \mathrm{im} A^{k} =0$.
\end{enumerate}
\end{bigthm}

By the classification of low dimensional Poincar\'{e} spaces, we obtain Corollaries \ref{cor_Poincare-0}, \ref{cor_Poincare-1}, \ref{cor_Poincare-2}, \ref{cor_Poincare-3} and \ref{cor_Poincare-4} in the case of $d-n \le 4$. In particular, the following corollary is a characterization of topological tori in terms of self-covering properties.

\begin{bigcor}\label{cor_torus}
Suppose $M$ is a closed TOP $m$-manifold such that $\pi_{1} (M)$ is abelian. Suppose $M'$ is a nontrivial cover of $M$ and $h \colon  M \rightarrow M'$ is a homotopy equivalence. Let $G := \bigcap_{k=1}^{\infty} \mathrm{im} h_{\#}^{k}$ and $n := \mathrm{rank} (\pi_{1} (M) / G)$. Suppose further $M$ satisfies one of the following:
\begin{enumerate}
\item $m -n \leq 1$;

\item $m-n = 2$ and $|G| > 2$;

\item $m-n = 3$, $G$ is not cyclic, and $G \ne \mathbb{Z} \oplus \mathbb{Z}/2$.
\end{enumerate}
Then $M$ is homeomorphic to $\mathbb{T}^{m}$.
\end{bigcor}

\begin{remark}
Even if assume further $M$ is a self-covering DIFF (resp.~PL) manifold in Corollary \ref{cor_torus}, it is not true that $M$ is diffeomorphic (resp.~PL isomorphic) to $\mathbb{T}^{m}$. Actually, by \cite{Farrell_Gogolev} (see also \cite{Farrell_Jones78}), there exists a DIFF closed manifold $M$ such that: (i) $M$ carries a smooth expanding map, which means $M$ is self-covering in a very strong sense; (ii) $M$ is homeomorphic to $\mathbb{T}^{m}$; (iii) but $M$ is even not PL isomorphic to $\mathbb{T}^{m}$.
\end{remark}

\subsection{Low Dimensional Manifolds}
The categories $\mathrm{TOP} = \mathrm{PL} = \mathrm{DIFF}$ for dimensions from $1$ to $3$.

The classification of closed self-covering manifolds with dimensions $1$ or $2$ is trivial. They are $S^{1}$, $\mathbb{T}^{2}$ and a Klein bottle.

For dimension $3$, Wang-Wu had classified closed self-covering (in the sense of homeomorphism) manifolds up to finite covers in \cite[Theorem~8.6]{Wang_Wu}. More precisely, a closed $3$-manifold is homeomorphic to a nontrivial cover of itself if and only if it is finitely covered by a $\mathbb{T}^{2}$-bundle or a trivial surface bundle over $S^{1}$. In fact, Wang-Wu worked on the family of geometric manifolds (see \cite[p.~203]{Wang_Wu} for definition) and nonorientable manifolds with geometric double covers. By the solution to Thurston's geometrization conjecture, that family contains all closed $3$-manifolds.

For dimension $4$, the categories $\mathrm{TOP} \ne \mathrm{PL} = \mathrm{DIFF}$. We have a classification of closed self-covering $\mathrm{TOP}$ $4$-manifolds with abelian fundamental groups in Theorem \ref{thm_4-manifold} below. The manifolds $S^{1} \tilde{\times} S^{3}$ and $S^{1} \tilde{\times} \mathbb{R}P^{3}$ are defined in Lemma \ref{lem_4-manifold}, and the manifolds $E_{1}$, $E_{2}$ and $PE_{1}$ are defined in Example \ref{exa_S2_T2}.
\begin{bigthm}\label{thm_4-manifold}
Suppose $M$ is a closed $\mathrm{TOP}$ $4$-manifold with abelian $\pi_{1} (M)$, and $M$ is homotopy equivalent to a nontrivial cover of itself. Then $M$ satisfies one of the following:
\begin{enumerate}
\item $\pi_{1} (M) = \mathbb{Z}$, and $M$ is homeomorphic to $S^{1} \times S^{3}$ or $S^{1} \tilde{\times} S^{3}$.

\item $\pi_{1} (M) = \mathbb{Z} \oplus \mathbb{Z} /2$, and $M$ is homotopy equivalent to $S^{1} \times \mathbb{R}P^{3}$ or $S^{1} \tilde{\times} \mathbb{R}P^{3}$. A $2$-cover of $M$ is homeomorphic to $S^{1} \times S^{3}$ or $S^{1} \tilde{\times} S^{3}$, respectively.

\item $\pi_{1} (M) = \mathbb{Z} \oplus \mathbb{Z} /k$ for some $k>2$, and $M$ is homotopy equivalent to $S^{1} \times L$, where $L$ is a $3$-dimensional lens space. A $k$-cover of $M$ is homeomorphic to $S^{1} \times S^{3}$.

\item $\pi_{1} (M) = \mathbb{Z}^{2}$, and $M$ is homeomorphic to $\mathbb{T}^{2} \times S^{2}$, $S^{1} \times (S^{1} \tilde{\times} S^{2})$, $E_{1}$ or $E_{2}$.

\item $\pi_{1} (M) = \mathbb{Z}^{2} \oplus \mathbb{Z}/2$, and $M$ is homotopy equivalent to $\mathbb{T}^{2} \times \mathbb{R} P^{2}$ or $PE_{1}$. A $2$-cover of $M$ is homeomorphic to $\mathbb{T}^{2} \times S^{2}$ or $E_{1}$, respectively.

\item $\pi_{1} (M) = \mathbb{Z}^{4}$, and $M$ is homeomorphic to $\mathbb{T}^{4}$.
\end{enumerate}

Furthermore, the above listed concrete manifolds ($S^{1} \times S^{3}$ etc.) are closed $\mathrm{DIFF}$ $4$-manifolds which are diffeomorphic to certain nontrivial covers of themselves. They are also of distinct homotopy types. (See also Lemma \ref{lem_S1_lens}.)
\end{bigthm}

For more properties of the above concrete manifolds, see Section \ref{sec_low_dim}. In particular, they are smooth bundles over certain tori.
\begin{remark}
As indicated by Theorem \ref{thm_BS_manifold}, self-covering closed $4$-manifolds with noncommutative fundamental groups can be very complicated.
\end{remark}

\subsection{High Dimensional Manifolds}\label{subsec_high}
Before focusing on Question \ref{que_abelian}, we consider an easier problem: \textit{fiber a closed self-covering manifold with abelian fundamental group over $S^{1}$.} The following Theorem \ref{thm_bundle_s1} generalizes the main Theorem B in \cite{Qin_Su_Wang}.

\begin{bigthm}\label{thm_bundle_s1}
Suppose $M$ is a closed $\mathrm{CAT}$ $m$-manifold with $m \ge 6$. Suppose $M'$ is a nontrivial cover of $M$, and $h \colon  M \rightarrow M'$ is a homotopy equivalence. Suppose $\pi_{1} (M)$ is free abelian. Let $G: = \bigcap_{k=1}^{\infty} \mathrm{im} h_{\#}^{k}$. Then the following statements hold:
\begin{enumerate}
\item $\pi_{1} (M) /G \cong \mathbb{Z}^{n}$ for some $n >0$.

\item Let $\theta \colon  \pi_{1} (M) \rightarrow \mathbb{Z}$ be an epimorphism with $G \le \ker \theta$. Then there exists a $\mathrm{CAT}$ bundle projection $p \colon  M \rightarrow S^{1}$ with $p_{\#} = \theta$.
\end{enumerate}
\end{bigthm}

\begin{remark}
Assuming further $n=1$ in Theorem \ref{thm_bundle_s1}, this theorem becomes the main part of \cite[Theorem~B]{Qin_Su_Wang}.
\end{remark}

\begin{remark}
If $m \le 5$, $\pi_{1} (M) = \mathbb{Z}$, and $\mathrm{CAT} = \mathrm{TOP}$, then $G=0$ and the conclusion of Theorem \ref{thm_bundle_s1} still holds. See the (1) in Theorem \ref{thm_bundle_T2}.
\end{remark}

\begin{remark}
By \cite[Theorems~C~\&~D]{Qin_Su_Wang}, the conclusion of Theorem \ref{thm_bundle_s1} is no longer true if we allow $\pi_{1} (M)$ to contain torsion elements.
\end{remark}

\begin{bigcor}
Suppose $M$ is a closed $\mathrm{CAT}$ $m$-manifold such that $m \ge 6$ and $\pi_{1} (M)$ is free abelian. If $M$ is homotopy equivalent to one of its nontrival cover, then it is a $\mathrm{CAT}$ bundle over $S^{1}$ with a connected fiber.
\end{bigcor}

In the following, we try to get some positive solutions to Question \ref{que_abelian} under mild extra assumptions. It's easy to see $G: = \ker p_{\#} = \bigcap_{k=1}^{\infty} \mathrm{im} h_{\#}^{k}$ and $n = \mathrm{rank} (\pi_{1} (M) / G)$ in this question. Theorem \ref{thm_Poincare} tells us that there always exists a homotopy pullback \eqref{que_abelian_1} such that the homotopy fiber of $p$ is a finitely dominated Poincar\'{e} space. The remaining task is to improve $p$ in its homotopy class.

The case of $n=1$ had been studied in \cite{Qin_Su_Wang}. More precisely, assuming $n=1$, if either (i) $\pi_{1} (M)$ is free abelian and $m \ge 6$, or (ii) $\pi_{1} (M) = \mathbb{Z}$ and $\mathrm{CAT} = \mathrm{TOP}$, then the answer to Question \ref{que_abelian} is yes.

Now we deal with the case of $n>1$. Fibering over a torus is much more difficult than fibering over a circle. To obtain good results, we work in the category $\mathrm{TOP}$ only. We shall improve the $p$ in Theorem \ref{thm_Poincare} step by step which leads to fibering results in four senses: approximated fibrations, block bundles, stable bundles and bundles.

The concepts of approximated fibrations and block bundles will be recalled in Definitions \ref{def_approximate_fibration} and \ref{def_block_bundle}, respectively. A bundle is clearly an approximated fibration and a block bundle. By \cite[Lemma~3.3.1]{Quinn79}, every block bundle projection between manifolds is homotopic to a map which is both a block bundle projection and an approximated fibration projection. So, roughly speaking, one has the implications:
\[
\text{(fiber) bundle} \Rightarrow \text{block bundle} \Rightarrow \text{approximated fibration}.
\]
The theory of block bundles is mature and of great importance, especially in $\mathrm{PL}$, to geometric topology (see e.g. \cite{Rourke_Sanderson68I}, \cite{Rourke_Sanderson68II}, \cite{Rourke_Sanderson68III}, \cite{Rourke_Sanderson71}, \cite{BLR1975}, and \cite{Ebert_Randal}). The theory of approximated fibration is also well developed and important in $\mathrm{TOP}$ (see e.g. \cite{Coram_Duvall1}, \cite{Coram_Duvall2}, \cite{Ferry}, \cite{HTW90} and \cite{FLS2018}). Though these two notions are weaker than bundles, they also possess nice properties. For example, by \cite[Corollary~12.14~\&~Theorem~12.15]{HTW90}, if $p \colon  E \rightarrow B$ is an approximated fibration projection between manifolds, then so is $p|_{p^{-1} (U)} \colon  p^{-1} (U) \rightarrow U$ for every open subset $U$ of $B$, and $p^{-1} (U)$ has the homotopy type of the homotopy fiber of $p$ provided that $U$ is contractible. These facts are very desirable from a homotopy view point.

The following Theorem \ref{thm_approximate_fibration} provides a positive solution to Question \ref{que_abelian} in the sense of approximate fibrations. Recall that a finite Poincar\'{e} complex is both a finite CW complex and a Poincar\'{e} space.
\begin{bigthm}\label{thm_approximate_fibration}
Suppose the $X$ in Theorem \ref{thm_Poincare} is a closed $\mathrm{TOP}$ $m$-manifold $M$ such that $m \ne 4$ and $\pi_{1} (M)$ is free abelian. Then $p$ can be further arranged as an approximate fibration projection, and the homotopy fiber of $p$ is homotopy equivalent to a finite Poincar\'{e} complex.
\end{bigthm}

The following Theorem \ref{thm_block_bundle} gives a positive solution to Question \ref{que_abelian} in the sense of block bundles.
\begin{bigthm}\label{thm_block_bundle}
Suppose the $X$ in Theorem \ref{thm_Poincare} is a closed $\mathrm{TOP}$ $m$-manifold $M$ such that $m-n \ge 5$ and $\pi_{1} (M)$ is free abelian. Then $p$ can be further arranged as a $\mathrm{TOP}$ block bundle projection. Here $\mathbb{T}^{n}$ carries a suitable triangulation compatible with its standard $\mathrm{PL}$ structure.
\end{bigthm}

\begin{remark}
Again, by \cite[Theorems~C~\&~D]{Qin_Su_Wang}, if we allow $\pi_{1} (M)$ to contain torsion elements, then the conclusion of Theorem \ref{thm_block_bundle} is no longer true even if $n=1$. In fact, it's easy to see every block bundle projection over $S^{1}$ is homotopic to a bundle projection.
\end{remark}

The following Theorem \ref{thm_stable_bundle} gives a positive solution to Question \ref{que_abelian} in the sense of stable bundles.
\begin{bigthm}\label{thm_stable_bundle}
Under the assumption of Theorem \ref{thm_block_bundle}, let $s = \frac{1}{2} n (n-1)$. Then the composition map $M \times \mathbb{T}^{s} \rightarrow M \overset{p}{\rightarrow} \mathbb{T}^{n}$ is homotopic to a $\mathrm{TOP}$ bundle projection.
\end{bigthm}

\begin{remark}
In Theorem \ref{thm_stable_bundle}, if $n=1$, then $s=0$ and $\mathbb{T}^{s} := (S^{1})^{s}$ is considered as a point, the conclusion follows from Theorem \ref{thm_bundle_s1} as well.
\end{remark}

The following Theorems \ref{thm_bundle_T2} and \ref{thm_bundle} give positive solutions to Question \ref{que_abelian} in the sense of bundles.
\begin{bigthm}\label{thm_bundle_T2}
Suppose the $X$ in Theorem \ref{thm_Poincare} is a closed $\mathrm{TOP}$ $m$-manifold $M$ such that $\pi_{1} (M)$ is abelian and $G=0$. Assume further one of the following:
\begin{enumerate}
\item $n=1$;

\item $n=2$ and $m \ge 7$.
\end{enumerate}
Then $p$ can be further arranged as a $\mathrm{TOP}$ bundle projection with a $1$-connected fiber.
\end{bigthm}

\begin{bigthm}\label{thm_bundle}
Suppose the $X$ in Theorem \ref{thm_Poincare} is a closed $\mathrm{TOP}$ $m$-manifold $M$ such that $\pi_{1} (M)$ is abelian and $G=0$. Suppose further $\pi_{i} (M) =0$ for $2 \le i \le r$, where $m-n \ge r+4$ and $\min \{ 2r-1, r+4 \} \ge n$. Then $p$ can be further arranged as a $\mathrm{TOP}$ bundle projection with an $r$-connected fiber.
\end{bigthm}

\begin{remark}
In Theorem \ref{thm_bundle}, when $n \le 2$, its statement is obviously weaker than that of Theorem \ref{thm_bundle_T2}.
\end{remark}

If one drops the assumptions on the dimension and the high connectivity of the homotopy fiber, the conclusion of Theorem \ref{thm_bundle} will be no long true even if $G=0$. The following Theorem \ref{thm_nonfibering} provides two families of counterexamples. They are closed manifolds of homotopy types $\mathbb{T}^{n} \times S^{i}$, where $i=2$ or $3$. Clearly, each manifold has a homotopy equivalence to one of its nontrivial cover such that the corresponding $G =0$.
\begin{bigthm}\label{thm_nonfibering}
For each $n \ge 4$, and $i=2,3$, there exists a closed $\mathrm{TOP}$ manifold $M$ satisfying the following properties:
\begin{enumerate}
\item $M$ is homotopy equivalent to $\mathbb{T}^{n} \times S^{i}$.

\item $M$ is not a $\mathrm{TOP}$ bundle over $\mathbb{T}^{n}$.
\end{enumerate}

Furthermore, if $n \ge 8$, then $M$ can be additionally arranged as a closed $\mathrm{DIFF}$ manifold.
\end{bigthm}

\subsection{Noncommutative Fundamental Groups}
To study closed self-covering manifolds with general, i.e. unnecessarily commutative, fundamentals groups, we propose the following question which is a generalization of Question \ref{que_abelian}. Recall that an aspherical manifold is a manifold with a contractible universal cover.
\begin{question}\label{que_nonabelian}
Suppose $M$ is a closed $\mathrm{CAT}$ manifold, $q \colon  M' \rightarrow M$ is a nontrivial covering, and $h \colon  M \rightarrow M'$ is a homotopy equivalence preserving base points.

Are there a closed aspherical $\mathrm{CAT}$ manifold $B$, a $\mathrm{CAT}$ bundle projection $p \colon  M \rightarrow B$ and a $\mathrm{CAT}$ covering map $\psi$ such that the following diagram is a homotopy pullback?
\begin{equation}\label{que_nonabelian_1}
\xymatrix{
  M \ar[d]_{p} \ar[r]^{q \circ h} & M \ar[d]^{p} \\
  B \ar[r]^{\psi} & B   }
\end{equation}
Here $\psi$ is required to preserve base points and to satisfy the condition $\bigcap_{k=1}^{\infty} \mathrm{im} \psi_{\#}^{k} = 1$.
\end{question}

By ``a $\mathrm{CAT}$ covering map", we mean $\psi$ is a composition $B \rightarrow B' \rightarrow B$, where $B'$ is a cover of $B$, and $B \rightarrow B'$ is a $\mathrm{CAT}$ isomorphism. Since the answer to Question \ref{que_abelian} is no in general, the answer to Question \ref{que_nonabelian} has to be negative. We still wish to get some positive answers by imposing mild extra assumptions. First of all, we wonder if the answer to Question \ref{que_nonabelian} is affirmative in a weaker sense. If there was a $p$ fitting \eqref{que_nonabelian_1} well, then the homotopy fiber of $p$ would be homotopy equivalent to the cover $\overline{M}$ of $M$ such that $\pi_{1} (\overline{M}) = \ker p_{\#}$. Since \eqref{que_nonabelian_1} was a homotopy pullback, we would have $h_{\#} \ker p_{\#} = \ker p_{\#}$. By $\bigcap_{k=1}^{\infty} \mathrm{im} \psi_{\#}^{k} = 1$, we infer $\ker p_{\#} = \bigcap_{k=1}^{\infty} \mathrm{im} h_{\#}^{k}$. Again, here we abuse the notation $h_{\#}$ to denote the composition
\[
(qh)_{\#} \colon  \pi_{1} (M) \rightarrow \pi_{1} (M') \hookrightarrow \pi_{1} (M).
\]
This yields an obvious necessary condition:

\begin{question}\label{que_cover}
Let $G := \bigcap_{k=1}^{\infty} \mathrm{im} h_{\#}^{k}$. Let $\overline{M}$ be the cover of $M$ with $\pi_{1} (\overline{M}) = G$. Is $G$ normal in $\pi_{1} (M)$? And is $\overline{M}$ finitely dominated?
\end{question}

If the answer to Question \ref{que_nonabelian} was yes in the homotopy sense, similar to Theorem \ref{thm_Poincare}, the answer to Question \ref{que_cover} would be positive. In particular, $G$ would be a finitely presented normal subgroup of $\pi_{1} (M)$. However, the following Theorem \ref{thm_BS_manifold} shows this is not true even if $h$ is a diffeomorphism. In this theorem, $M$ is oriented, the orientation of $M'$ is induced from that of $M$ via the covering map.

\begin{bigthm}\label{thm_BS_manifold}
Let $BS(2,3) = \langle a,t \mid t a^{2} t^{-1} = a^{3} \rangle$ be a Baumslag-Solitar group. For each $m \geq 5$, there exist a compact smooth hypersurface $M$ (without boundary) in $\mathbb{R}^{m+1}$ and a ($C^{\infty}$) diffeomorphism $h \colon  M \rightarrow M'$ satisfying the following properties:
\begin{enumerate}
\item $\pi_{1} (M) = BS(2,3)$, and $M'$ is a $5$-cover of $M$.

\item $h$ preserves base points and can be arranged to preserve or reverse orientations as one wishes.

\item $G := \bigcap_{k=1}^{\infty} \mathrm{im} h_{\#}^{k}$ is not a normal subgroup of $\pi_{1} (M)$.

\item $G$ is a free group with infinite rank, hence it is not finitely generated.
\end{enumerate}

Furthermore, there exist a compact smooth hypersurface $M$ in $\mathbb{R}^{5}$ and a simple homotopy equivalence $h \colon  M \rightarrow M'$ satisfying the above properties.
\end{bigthm}

\begin{remark}\label{rmk_BS_general}
Suppose $q_{1} < q_{2}$ are two positive prime numbers, and $s>1$ is an integer prime to $q_{1} q_{2}$. Similarly, there is a desired triple $(M, M', h)$ as the one in Theorem \ref{thm_BS_manifold} except that $\pi_{1} (M) = BS(q_{1}, q_{2})$ and $M'$ is an $s$-cover of $M$. See Remark \ref{rmk_BS_proof}.
\end{remark}

\begin{remark}\label{rmk_BS_nonfinite}
According to the (4) in Theorem \ref{thm_BS_manifold}, the covering manifold $\overline{M}$ of $M$ with $\pi_{1} (\overline{M}) =G$ is far from being finitely dominated. Hence, the conclusion of Theorem \ref{thm_homotopy_finite} would fail dramatically when the fundamental group is nonabelian.
\end{remark}

By Theorem \ref{thm_BS_manifold}, to find any positive answer to Question \ref{que_nonabelian}, one has to at least impose the condition that $\bigcap_{k=1}^{\infty} \mathrm{im} h_{\#}^{k}$ is a finitely presented normal subgroup of $\pi_{1} (M)$. We are not able to make further progress on Question \ref{que_nonabelian} at this stage. Nevertheless, we shall raise a program together with a few questions on it in Section \ref{sec_question}.

Theorem \ref{thm_BS_manifold} follows from Theorems \ref{thm_cw_manifold} and \ref{thm_cw_4-manifold} which are general constructions of self-covering manifolds out of self-covering complexes. They reflect a diversity of the self-covering phenomena in the category of closed manifolds. In these theorems, $\pi_{1} (M)$ is identified with $\pi_{1} (X)$ by a fixed isomorphism, the subgroups $\pi_{1} (M')$ and $\pi_{1} (X')$ are then identified.

\begin{bigthm}\label{thm_cw_manifold}
Suppose $X$ is an $r$-dimensional finitely dominated CW complex, $X'$ is a $k$-cover of $X$ for some $k>1$. Suppose $\hat{h} \colon  X \rightarrow X'$ is a homotopy equivalence preserving base points. Then for each integer $m \ge 2r+12$, there exist a compact smooth hypersurface $M$ in $\mathbb{R}^{m+1}$ and a ($C^{\infty}$) diffeomorphism $h \colon  M \rightarrow M'$ satisfying the following properties:
\begin{enumerate}
\item $\pi_{1} (M) = \pi_{1} (X)$, and $M'$ is a $k$-cover of $M$ with $\pi_{1} (M') = \pi_{1} (X')$.

\item $h$ preserves base points and $h_{\#} \colon  \pi_{1} (M) \rightarrow \pi_{1} (M')$ equals $\hat{h}_{\#} \colon  \pi_{1} (X) \rightarrow \pi_{1} (X')$. Additionally, $h$ can be arranged to preserve or reverse orientations as one wishes.
\end{enumerate}

Furthermore, if $X$ is the underlying space of a finite simplicial complex and $\hat{h}$ is a simple homotopy equivalence, then the restriction on the above $m$ can be weakened as $m \ge \max \{ 2r, r+2, 5 \}$.
\end{bigthm}

\begin{bigthm}\label{thm_cw_4-manifold}
Suppose $X$ is the underlying space of a $2$-dimensional finite simplicial complex, $X'$ is a $k$-cover of $X$ for some $k>1$. Suppose $\hat{h} \colon  X \rightarrow X'$ is a homotopy (resp.~simple homotopy) equivalence preserving base points. Then there exist a compact smooth hypersurface $M$ in $\mathbb{R}^{5}$ and a homotopy (resp.~simple homotopy) equivalence $h \colon  M \rightarrow M'$ such that $(M,M',h)$ satisfies the (1) and (2) in Theorem \ref{thm_cw_manifold}.

Furthermore, if $\hat{h}$ is a simple homotopy equivalence and $\pi_{1} (X)$ is good in the sense of \cite{Freedman_Teichner}, then $h$ is a homeomorphism.
\end{bigthm}

A canonical application of Theorems \ref{thm_cw_manifold} and \ref{thm_cw_4-manifold} is to set $X$ as an Eilenberg-MacLane space $K(\pi, 1)$.

\subsection{Algebraic Foundation}
We convert Theorem \ref{thm_homotopy_finite} to the following algebraic Theorem \ref{thm_finite_module} on finite generation of modules. This theorem may be of independent interest.
\begin{assumption}\label{asm_matrix}
Let $A=(a_{ij})_{n \times n} \in \mathrm{End} (\mathbb{Z}^{n})$, where elements in $\mathbb{Z}^{n}$ are considered as column vectors and the action of $A$ on $\mathbb{Z}^{n}$ is by left multiplication. Suppose $\det (A) \neq 0$ and $\bigcap_{k=1}^{\infty} \mathrm{im} A^{k} =0$.
\end{assumption}

Let $R$ be a commutative Noetherian ring and $S= R[t_{1}^{\pm 1}, \dots, t_{n}^{\pm 1}]$, where $t_{1}, \dots, t_{n}$ are indeterminates.

\begin{assumption}\label{asm_ring_hom}
Suppose $\phi \colon  S \rightarrow S$ is a ring monomorphism such that $\phi (R) =R$. Suppose further, $\forall j$, $\phi(t_{j}) = g_{j} \prod_{i=1}^{n} t_{i}^{a_{ij}}$, where $(a_{ij})_{n \times n} = A$ and $g_{j}$ are units of $R$.
\end{assumption}

\begin{bigthm}\label{thm_finite_module}
Under the Assumptions \ref{asm_matrix} and \ref{asm_ring_hom}, assume further $\mathfrak{M}$ is a finite $S$-module. Suppose there exists a group isomorphism $\eta \colon  \mathfrak{M} \rightarrow \mathfrak{M}$ such that, $\forall s \in S$, $\forall x \in \mathfrak{M}$, $\eta (sx) = \phi (s) \eta (x)$. Then $\mathfrak{M}$ is a finite $R$-module.
\end{bigthm}

Here we say $\eta$ is a \textit{$\phi$-twisted} homomorphism of $S$-modules.

\subsection{Outline}
The proof of Theorem \ref{thm_finite_module} is complicated and occupies Sections \ref{sec_algebra_finite}, \ref{sec_eigenvalue} and \ref{sec_algebra_tori}. The argument in Section \ref{sec_algebra_finite} is exclusively based on commutative algebra. Section \ref{sec_eigenvalue} only needs linear algebra. However, Section \ref{sec_algebra_tori} mainly involves abstract algebraic geometry, including the language of schemes. In Section \ref{sec_CW}, we prove Theorem \ref{thm_homotopy_finite} and Corollary \ref{cor_group}, and their generalizations, Theorems \ref{thm_homology_finite} and \ref{thm_homotopy_skeleton}. Theorem \ref{thm_Poincare} and Corollary \ref{cor_torus} are proved in Section \ref{sec_Poincare}. We prove in Section \ref{sec_low_dim}  Theorem \ref{thm_4-manifold} on $4$-manifolds. Section \ref{sec_fibering} includes the proofs of the fibering Theorems \ref{thm_bundle_s1}, \ref{thm_approximate_fibration}, \ref{thm_block_bundle}, \ref{thm_stable_bundle}, \ref{thm_bundle_T2} and \ref{thm_bundle}, while Section \ref{sec_nonfibering} contains the proof of the non-fibering Theorem \ref{thm_nonfibering}. In Section \ref{sec_noncommutative}, we prove Theorems \ref{thm_BS_manifold}, \ref{thm_cw_manifold} and \ref{thm_cw_4-manifold}. Finally, we propose some questions in Section \ref{sec_question}.

\medskip

\noindent \textbf{Acknowledgement.}
We are indebted to Botong Wang who helped us with the proof of Theorem \ref{thm_finite_module}. Particularly, the proof of Proposition \ref{prop_degree} is completely due to him. We thank Xiaoming Du, Hailiang Hu, Wei Wang, Xiaolei Wu and Shengkui Ye for various discussions. We also thank Shicheng Wang who explained \cite{Wang_Wu} to us. The first author is partially supported by NSFC 11871272. The second author is partially supported by NSFC 12471069.

\section{Algebraic Finiteness}\label{sec_algebra_finite}
In this section, we shall prove Theorem \ref{thm_finite_module}. Our argument is based on Propositions \ref{prop_eigenvalue} and \ref{prop_degree} which will be proved later.

We introduce a technique of iteration which will be frequently used. For each integer $m>0$, we have $\eta^{m} \colon  \mathfrak{M} \rightarrow \mathfrak{M}$ is a group isomorphism such that, $\forall x \in \mathfrak{M}$, $\forall s \in S$, $\eta^{m} (sx) = \phi^{m} (s) \eta^{m} (x)$. Obviously, $\phi^{m} \colon  S \rightarrow S$ is also a ring monomorphism with $\phi^{m} (R) =R$. Furthermore, it's easy to check that, $\forall j$, $\phi^{m} (t_{j}) = g_{m,j} \prod_{i=1}^{n} t_{i}^{b_{ij}}$, where $(b_{ij})_{n \times n} = A^{m}$ and $g_{m,j}$ are units of $R$. Clearly, $\bigcap_{k=1}^{\infty} \mathrm{im} A^{mk} = \bigcap_{k=1}^{\infty} \mathrm{im} A^{k} =0$. Therefore, $(\eta^{m}, \phi^{m}, A^{m})$ essentially satisfies the same assumption as $(\eta, \phi, A)$ does. To prove Theorem \ref{thm_finite_module}, we may consider $(\eta^{m}, \phi^{m}, A^{m})$ for some suitable $m>0$.

Let's assume $\mathfrak{M} \neq 0$. Otherwise, the conclusion of Theorem \ref{thm_finite_module} is already true. Recall that a prime ideal $Q$ of $S$ is an associated prime ideal of $\mathfrak{M}$ if there exists $x \in \mathfrak{M}$ such that the annihilator $\mathrm{Ann}_{S} (x) = Q$. Let $\mathrm{Ass}_{S} (\mathfrak{M})$ be the set of prime ideals of $S$ associated to $\mathfrak{M}$. Since $\mathfrak{M}$ is a finite $S$-module, the set $\mathrm{Ass}_{S} (\mathfrak{M})$ is nonempty and finite (cf. \cite[(7.B)~\&~(7.G)]{Matsumura1}). By \cite[(9.A)]{Matsumura1},
\begin{equation}\label{eqn_R_S_associated}
\mathrm{Ass}_{R} (\mathfrak{M}) = \{ Q \cap R \mid Q \in \mathrm{Ass}_{S} (\mathfrak{M}) \}.
\end{equation}
Thus $\mathrm{Ass}_{R} (\mathfrak{M})$ is also nonempty and finite.

\begin{lemma}\label{lem_iteration}
There exists an integer $m>0$ such that $\phi^{-m} (Q) = Q$ and $\phi^{-m} (P) = P$ for all $Q \in \mathrm{Ass}_{S} (\mathfrak{M})$ and $P \in \mathrm{Ass}_{R} (\mathfrak{M})$.
\end{lemma}
\begin{proof}
Since $\phi \colon  S \rightarrow \phi (S)$ is a ring isomorphism, $\eta \colon  (\mathfrak{M}, S) \rightarrow (\mathfrak{M}, \phi (S))$ is a $\phi$-twisted isomorphism of modules. Thus $\phi^{-1}$ induces a bijection $\phi^{-1} \colon  \mathrm{Ass}_{\phi (S)} (\mathfrak{M}) \rightarrow \mathrm{Ass}_{S} (\mathfrak{M})$. On the other hand, $\phi (S) \subseteq S$, by \cite[(9.A)]{Matsumura1},
\[
\mathrm{Ass}_{\phi (S)} (\mathfrak{M}) = \{ Q \cap \phi (S) \mid Q \in \mathrm{Ass}_{S} (\mathfrak{M}) \}.
\]
Since $\mathrm{Ass}_{S} (\mathfrak{M})$ is finite, then $\phi^{-1}$ induces a bijection $\phi^{-1} \colon  \mathrm{Ass}_{S} (\mathfrak{M}) \rightarrow \mathrm{Ass}_{S} (\mathfrak{M})$, and $\phi^{-m}$ is the identity on $\mathrm{Ass}_{S} (\mathfrak{M})$ for some $m>0$. In other words, $\phi^{-m} (Q) = Q$ for all $Q \in \mathrm{Ass}_{S} (\mathfrak{M})$. By \eqref{eqn_R_S_associated}, $\phi^{-m} (P) = P$ for all $P \in \mathrm{Ass}_{R} (\mathfrak{M})$.
\end{proof}

We firstly prove a special and important case of Theorem \ref{thm_finite_module}.
\begin{proposition}\label{prop_field_module}
In Theorem \ref{thm_finite_module}, assuming further $R=K$ for some field $K$, then the conclusion holds, i.e. $\dim_{K} \mathfrak{M}$ is finite.
\end{proposition}
\begin{proof}
Let's assume $\mathfrak{M} \neq 0$. By Proposition \ref{prop_ideal_dim_zero} below, it suffices to show $Q$ is maximal for all $Q \in \mathrm{Ass}_{S} (\mathfrak{M})$. Fix a $Q \in \mathrm{Ass}_{S} (\mathfrak{M})$. By Lemma \ref{lem_iteration}, replacing $\phi$ with its power if necessary, we may assume $\phi^{-1} (Q) =Q$. Let $\kappa (Q) := \mathrm{Quot} (S/Q)$ be the residue field of $S$ at $Q$. Then $\phi$ induces a field embedding $\phi'' \colon  \kappa (Q) \rightarrow \kappa (Q)$.

We prove $Q$ is maximal by contradiction. Suppose not, taking the $\phi'$ in Proposition \ref{prop_degree} as $\phi$ here, we would have $[\kappa (Q): \mathrm{im} \phi''] >1$. Note that $\kappa (Q) \otimes_{S} \mathfrak{M}$ is finite dimensional over $\kappa (Q)$ because $\mathfrak{M}$ is a finite $S$-module. On one hand, by Lemma \ref{lem_dimension_inequality} below,
\begin{equation}\label{prop_field_module_1}
\dim_{\mathrm{im} \phi''} \mathrm{im} \phi'' \otimes_{\phi (S)} \mathfrak{M} \geq [\kappa (Q): \mathrm{im} \phi''] \cdot \dim_{\kappa (Q)} \kappa (Q) \otimes_{S} \mathfrak{M}.
\end{equation}
On the other hand, by assumption, $\phi \colon  S \rightarrow \phi (S)$ is a ring isomorphism, $\eta \colon  \mathfrak{M} \rightarrow \mathfrak{M}$ is a group isomorphism, and $\eta (sx) = \phi (s) \eta (x)$ for all $s \in S$ and $x \in \mathfrak{M}$. Thus
\[
\phi'' \otimes \eta \colon  \kappa (Q) \otimes_{S} \mathfrak{M} \rightarrow \mathrm{im} \phi'' \otimes_{\phi (S)} \mathfrak{M}
\]
is a $\phi''$-twisted isomorphism, which implies
\begin{equation}\label{prop_field_module_2}
\dim_{\kappa (Q)} \kappa (Q) \otimes_{S} \mathfrak{M} = \dim_{\mathrm{im} \phi''} \mathrm{im} \phi'' \otimes_{\phi (S)} \mathfrak{M}.
\end{equation}
Combining \eqref{prop_field_module_1}, \eqref{prop_field_module_2}, and the fact $[\kappa (Q): \mathrm{im} \phi''] >1$, we infer $\dim_{\kappa (Q)} \kappa (Q) \otimes_{S} \mathfrak{M} =0$ which contradicts the fact $Q \in \mathrm{Ass}_{S} (\mathfrak{M})$ (cf. \cite[(7.D)]{Matsumura1}).
\end{proof}

\begin{lemma}\label{lem_dimension_inequality}
The inequality \eqref{prop_field_module_1} holds.
\end{lemma}
\begin{proof}
Firstly, we claim that a basis of $\kappa (Q) \otimes_{S} \mathfrak{M}$ over $\kappa (Q)$ can be chosen as $1 \otimes x_{1}, \dots, 1 \otimes x_{m}$ for some $x_{j} \in \mathfrak{M}$. Actually, if $y_{1}, \dots, y_{m}$ is a basis, then
\[
y_{j} = \sum_{h} \tfrac{n_{h}}{d_{h}} \otimes z_{h}
\]
for some $n_{h}$ and $d_{h}$ in $S/Q$ and $z_{h} \in \mathfrak{M}$. Multiplying $y_{j}$ by a nonzero element in $S/Q$, we may assume each $d_{h} =1$ and each $n_{h}$ can be represented by an $r_{h} \in S$. Then
\[
y_{j} = \sum_{h} n_{h} \otimes z_{h} = \sum_{h} 1 \otimes r_{h} z_{h} = 1 \otimes \sum_{h} r_{h} z_{h} = 1 \otimes x_{j}.
\]

Secondly, suppose $\alpha_{1}, \dots, \alpha_{k}$ are a family of elements of $\kappa (Q)$ which are linear independent over $\mathrm{im} \phi''$. To finish the proof, it suffices to verify $\dim_{\mathrm{im} \phi''} \mathrm{im} \phi'' \otimes_{\phi (S)} \mathfrak{M} \ge km$.

Multiplying the whole family $\alpha_{1}, \dots, \alpha_{k}$ by a nonzero element in $\kappa (Q)$ if necessary, we may assume they are represented by $s_{1}, \dots, s_{k}$ in $S$. We claim that $1 \otimes s_{i} x_{j}$ ($1 \leq i \leq k$, $1 \leq j \leq m$) are linear independent in $\mathrm{im} \phi'' \otimes_{\phi (S)} \mathfrak{M}$ over $\mathrm{im} \phi''$. Actually, suppose
\[
\sum_{i,j} l_{ij} \cdot 1 \otimes_{\phi (S)} s_{i} x_{j} =0
\]
in $\mathrm{im} \phi'' \otimes_{\phi (S)} \mathfrak{M}$ for some $l_{ij} \in \mathrm{im} \phi''$. Since $\phi (S) \subseteq S$, we see
\[
\sum_{j} \left( \sum_{i} l_{ij} \alpha_{i} \right) \cdot 1 \otimes_{S} x_{j} = \sum_{i,j} l_{ij} \cdot 1 \otimes_{S} s_{i} x_{j} = 0
\]
in $\kappa (Q) \otimes_{S} \mathfrak{M}$, which implies $\sum_{i} l_{ij} \alpha_{i} =0$ in $\kappa (Q)$ for all $j$ and hence all $l_{ij} =0$. Therefore, $\dim_{\mathrm{im} \phi''} \mathrm{im} \phi'' \otimes_{\phi (S)} \mathfrak{M} \ge km$, which completes the proof.
\end{proof}

Now we prove Theorem \ref{thm_finite_module}.
\begin{proof}[Proof of Theorem \ref{thm_finite_module}]
Again, we assume $\mathfrak{M} \neq 0$. By \cite[Theorem~G]{Qin_Su_Wang}, it suffices to verify two facts: firstly, $\kappa (P) \otimes_{R} \mathfrak{M}$ is finite dimensional over $\kappa (P)$ for all $P \in \mathrm{Ass}_{R} (\mathfrak{M})$; secondly, the eigenvalues of the actions of $t_{i}^{\pm 1}$ on $\kappa (P) \otimes_{R} \mathfrak{M}$ are integral over $R/P$ for all $P \in \mathrm{Ass}_{R} (\mathfrak{M})$ and $1 \leq i \leq n$. Here $\kappa (P)$ is the residue field of $R$ at $P$.

We firstly prove $\dim_{\kappa (P)} \kappa (P) \otimes_{R} \mathfrak{M}$ is finite. By Lemma \ref{lem_iteration} again, replacing $\phi$ with its power if necessary, we may assume $\phi^{-1} (P) = P$. Since $\phi|_{R} \in \mathrm{Aut} (R)$, it induces an automorphism $\sigma$ of $R/P$. This $\sigma$ extends to an automorphism of $\kappa (P)$ which is still denoted by $\sigma$. Clearly,
\[
\kappa (P) \otimes_{R} S = \kappa (P) [t_{1}^{\pm 1}, \dots, t_{n}^{\pm 1}],
\]
where $t_{1}, \dots, t_{n}$ are indeterminates over $\kappa (P)$. Since $\mathfrak{M}$ is a finite $S$-module, $\kappa (P) \otimes_{R} \mathfrak{M}$ is a finite $\kappa (P) \otimes_{R} S$-module. By Assumption \ref{asm_ring_hom}, we see
\[
\sigma \otimes \phi \colon \ \kappa (P) \otimes_{R} S \rightarrow \kappa (P) \otimes_{R} S
\]
is a ring monomorphism such that $(\sigma \otimes \phi)|_{\kappa (P)} = \sigma$ and $\forall j$, $\sigma \otimes \phi (t_{j}) = g_{j} \prod_{i=1}^{n} t_{i}^{a_{ij}}$, where we use the same notation $g_{j}$ to denote its image in $\kappa (P)$. Similarly,
\[
\sigma \otimes \eta \colon \ \kappa (P) \otimes_{R} \mathfrak{M} \rightarrow \kappa (P) \otimes_{R} \mathfrak{M},
\]
is a $\sigma \otimes \phi$-twisted isomorphism of $\kappa (P)$-modules. Taking the $(\mathfrak{M}, K, \phi, \eta)$ in Proposition \ref{prop_field_module} as $(\kappa (P) \otimes_{R} \mathfrak{M}, \kappa (P), \sigma \otimes \phi, \sigma \otimes \eta)$ here, we infer $\dim_{\kappa (P)} \kappa (P) \otimes_{R} \mathfrak{M}$ is finite.

It remains to show that the eigenvalues of the action of $t_{j}^{\pm 1}$ on $\kappa (P) \otimes_{R} \mathfrak{M}$ are integral over $R/P$. By the assumption that $\eta (sx) = \phi (s) \eta (x)$, we have the equation of actions
\[
\sigma \otimes \eta \cdot t_{j} = g_{j} \prod_{i=1}^{n} t_{i}^{a_{ij}} \cdot \sigma \otimes \eta
\]
or
\begin{equation}\label{thm_finite_module_2}
t_{j} = (\sigma \otimes \eta)^{-1} \cdot g_{j} \prod_{i=1}^{n} t_{i}^{a_{ij}} \cdot \sigma \otimes \eta = \sigma^{-1} (g_{j}) \cdot (\sigma \otimes \eta)^{-1} \cdot \prod_{i=1}^{n} t_{i}^{a_{ij}} \cdot \sigma \otimes \eta
\end{equation}
on $\kappa (P) \otimes_{R} \mathfrak{M}$. Here the action of $t_{j}$ is actually the action of $1 \otimes t_{j}$, and the one of $g_{j}$ is the multiplication with $1 \otimes g_{j}$ or $g_{j} \otimes 1$ (here $g_{j}$ is identified with its image in $\kappa (P)$). Note that the actions of $t_{j}$ and $g_{j}$ are $\kappa (P)$-linear, while $\sigma \otimes \eta$ is only $\sigma$-twisted $\kappa (P)$-linear. To remedy this $\sigma$-twisted trouble, we argue as follows. Choosing a basis, $\kappa (P) \otimes_{R} \mathfrak{M}$ is identified $\kappa (P)^{m}$ for some $m > 0$. Consider $t_{j}$, $g_{j}$, $\sigma \otimes \eta$ as maps on $\kappa (P)^m$. Define
\[
\Phi_{\sigma}  \colon \kappa (P)^{m} \to \kappa(P)^m, \ \ \ (z_{1}, \dots, z_{m})^{T} \mapsto (\sigma (z_{1}), \dots, \sigma (z_{m}))^{T},
\]
where $(z_{1}, \dots, z_{m})^{T}$ is the transpose of $(z_{1}, \dots, z_{m})$. By \eqref{thm_finite_module_2}, we have
\begin{eqnarray*}
\Phi_{\sigma} t_{j} \Phi_{\sigma}^{-1} & = & g_{j} \cdot \left( \Phi_{\sigma} (\sigma \otimes \eta)^{-1} \right) \cdot \prod_{i=1}^{n} t_{i}^{a_{ij}} \cdot \left( (\sigma \otimes \eta) \Phi_{\sigma}^{-1} \right) \\
& = & g_{j} \cdot \left( \Phi_{\sigma} (\sigma \otimes \eta)^{-1} \right) \cdot \prod_{i=1}^{n} t_{i}^{a_{ij}} \cdot \left( \Phi_{\sigma} (\sigma \otimes \eta)^{-1} \right )^{-1}.
\end{eqnarray*}
Now $t_{i}$, $\Phi_{\sigma} t_{j} \Phi_{\sigma}^{-1}$, and $\Phi_{\sigma} (\sigma \otimes \eta)^{-1}$ are invertible $\kappa (P)$-linear transformations on $\kappa (P)^{m}$ and hence can be identified with invertible matrices over $\kappa (P)$. If $t_{j}$ is represented by a matrix $B_{j}$, then $\Phi_{\sigma} t_{j} \Phi_{\sigma}^{-1}$ is represented by $\sigma (B_{j})$. Note also that $g_{j}$ is a unit in $R$, it hence represents a unit in $R/P$. Taking the $R$, $B_{j}$, $C$, and $u_{j}$ in Proposition \ref{prop_eigenvalue} as $R/P$, $t_{j}$, $\Phi_{\sigma} (\sigma \otimes \eta)^{-1}$, and $g_{j}$ here, respectively, we see all eigenvalues of $t_{j}^{\pm 1}$ are integral over $R/P$. The proof is completed.
\end{proof}

The following proposition is a general and elementary fact in commutative algebra.
\begin{proposition}\label{prop_ideal_dim_zero}
Suppose $K$ is a field, $K \subseteq S'$ and $S'$ is a finitely generated $K$-algebra. Suppose $\mathfrak{M}'$ is a nonzero finite $S'$-module. Then $\dim_{K} \mathfrak{M}'$ is finite if and only if $Q$ is maximal for each $Q \in \mathrm{Ass}_{S'} (\mathfrak{M}')$.
\end{proposition}
\begin{proof}
``$\Leftarrow$" By primary decomposition, we have $0= \bigcap_{i=1}^{k} \mathfrak{M}_{i}$, where $\mathfrak{M}_{i}$ is a $Q_{i}$-primary submodule of $\mathfrak{M}'$, i.e. $\mathrm{Ass}_{S'} (\mathfrak{M}'/\mathfrak{M}_{i}) = Q_{i}$, and $\mathrm{Ass}_{S'} (\mathfrak{M}') = \{ Q_{i} \mid 1 \leq i \leq k \}$ (see \cite[(8.G)]{Matsumura1}). Then there is a diagonal monomorphism of $S'$-module
\[
\mathfrak{M}' \rightarrow \bigoplus_{i=1}^{k} \mathfrak{M}'/\mathfrak{M}_{i}.
\]
Thus it suffices to show $\dim_{K} (\mathfrak{M}'/\mathfrak{M}_{i}) < \infty$ for all $i$. We may, therefore, assume $\mathrm{Ass}_{S'} (\mathfrak{M}') = \{ Q \}$. Then $\sqrt{\mathrm{Ann}_{S'} (\mathfrak{M}')} = Q$ (see \cite[Theorem~6.6]{Matsumura2}). Since $Q$ is maximal, we infer $S'/Q$ is a finite algebraic extension of $K$ (cf. \cite[I.~\S1]{Mumford}). By \cite[Proposition~8.6]{Atiyah_Macdonald}, $S'/ \mathrm{Ann}_{S'} (\mathfrak{M}')$ is an Artinian local ring with the maximal ideal $Q/ \mathrm{Ann}_{S'} (\mathfrak{M}')$, which further implies that $S'/ \mathrm{Ann}_{S'} (\mathfrak{M}')$ is a finite dimensional $K$-algebra. On the other hand, $\mathfrak{M}'$ is a finite $S'/ \mathrm{Ann}_{S'} (\mathfrak{M}')$-module, which implies the conclusion.

``$\Rightarrow$" Suppose $Q \in \mathrm{Ass}_{S'} (\mathfrak{M}')$. There exists $x \in \mathfrak{M}'$ such that $\mathrm{Ann}_{S'} (x) = Q$. Since $S'x$ is an $S'$-submodule of $\mathfrak{M}'$, we have $\dim_{K} S'x \leq \dim_{K} \mathfrak{M}' < \infty$. On the other hand, $S'x \cong S'/Q$. So $S'/Q$ is an algebraic extension of $K$, which means $Q$ is maximal.
\end{proof}

\begin{remark}
We actually only need the ``$\Leftarrow$" of Proposition \ref{prop_ideal_dim_zero}.
\end{remark}

\section{Eigenvalues}\label{sec_eigenvalue}
In this section, we shall prove the following proposition.

\begin{proposition}\label{prop_eigenvalue}
Suppose $R$ is a Noetherian domain. Suppose $K$ is the quotient field of $R$. Suppose $\sigma \in \mathrm{Aut} (R)$. By abusing notations, we also use $\sigma$ to denote the induced automorphism of $K$. Suppose $B_{i}$ ($1 \leq i \leq n$) and $C$ are invertible matrices with entries in $K$, where these $B_{i}$ are mutually commutative. Suppose further there exist units $u_{j} \in R$ ($1 \leq j \leq n$) such that, $\forall j$,
\begin{equation}\label{prop_eigenvalue_1}
\sigma \left( B_{j} \right) = u_{j} C \prod_{i=1}^{n} B_{i}^{a_{ij}} C^{-1},
\end{equation}
where $(a_{ij})_{n \times n} = A \in \mathrm{End} (\mathbb{Z}^{n})$ satisfying $\bigcap_{k=1}^{\infty} \mathrm{im} A^{k} =0$. Then the eigenvalues of these $B_{i}^{\pm 1}$ are integral over $R$.
\end{proposition}

Before proving Proposition \ref{prop_eigenvalue}, we make a remark on the assumption about $A$.
\begin{lemma}\label{lem_matrix_eigenvalue}
Suppose $A \in \mathrm{End} (\mathbb{Z}^{n})$. Then $\bigcap_{k=1}^{\infty} \mathrm{im} A^{k} =0$ if and only if, for each eigenvalue $\alpha$ of $A$, either $\alpha =0$ or $\alpha^{-1}$ is not integral over $\mathbb{Z}$.
\end{lemma}
\begin{proof}
``$\Rightarrow$" We prove by contradiction. Suppose $\alpha$ is a nonzero eigenvalue such that $\alpha^{-1}$ is integral over $\mathbb{Z}$. We infer there exists a polynomial
\[
\tilde{f} (x) = b_{0} x^{m} + \cdots + b_{m-1} x + 1 \in \mathbb{Z}[x]
\]
such that $\tilde{f} (\alpha) =0$. Then the rank of $\tilde{f} (A)$ is less than $n$. Let $V$ be the kernel of $\tilde{f} (A)$ on $\mathbb{Z}^{n}$. We infer $V \neq 0$ and $V$ is $A$-invariant. Furthermore, since
\[
A|_{V} (b_{0} A|_{V}^{m-1} + \cdots + b_{m-1}) = -1,
\]
$A|_{V}$ is invertible. Thus $0 \neq V \subseteq \bigcap_{k=1}^{\infty} \mathrm{im} A^{k}$, which is a contradiction.

``$\Leftarrow$" We also prove by contradiction. Suppose $W := \bigcap_{k=1}^{\infty} \mathrm{im} A^{k} \neq 0$. Then $W$ is $A$-invariant. Since $\ker A \cap \mathrm{im} A^{n} = 0$, we have $A|_{\mathrm{im} A^{n}}$ is injective and
\[
AW = A \bigcap_{k=n}^{\infty} \mathrm{im} A^{k} = \bigcap_{k=n}^{\infty} A \mathrm{im} A^{k} = \bigcap_{k=n+1}^{\infty} \mathrm{im} A^{k} = W.
\]
Thus $A|_{W}$ is invertible. Clearly, $W$ is a finite $\mathbb{Z}$-module. By Cayley-Hamilton Theorem, there exists a monic $f(x) \in \mathbb{Z}[x]$ such that $f(A|_{W}^{-1}) =0$. Thus, for each eigenvalue $\alpha$ of $A|_{W}$, we have $\alpha^{-1}$ is integral over $\mathbb{Z}$, which contradicts the assumption.
\end{proof}

\begin{corollary}\label{cor_transpose}
Suppose $A \in \mathrm{End} (\mathbb{Z}^{n})$ and $\bigcap_{k=1}^{\infty} \mathrm{im} A^{k} =0$. Then $\bigcap_{k=1}^{\infty} \mathrm{im} (A^{T})^{k} =0$, where $A^{T}$ is the transpose of $A$.
\end{corollary}

Let's go back to Proposition \ref{prop_eigenvalue}. Assume these matrices have size $m$. Then they are elements in $\mathrm{GL} (m,K)$. Let $L$ be the splitting field of the product of the characteristic polynomials of these $B_{j}$. Then $L/K$ is a finite algebraic extension.

\begin{lemma}
$\sigma$ extends to an automorphism of $L$.
\end{lemma}
\begin{proof}
Let $K'$ be the algebraic closure of $K$. Then $\sigma$ extends to an automorphism of $K'$. Considering $L$ as a subfield of $K'$, it suffices to show that $\sigma (L) = L$. Let $\lambda$ be an eigenvalue of $B_{j}$. Clearly, $\sigma (\lambda)$ is an eigenvalue of $\sigma (B_{j})$. By (\ref{prop_eigenvalue_1}) and the fact that these $B_{i}$ are mutually commutative, we infer $\sigma (\lambda)$ is the product of $u_{j}$ and some eigenvalues of these $B_{i}^{\pm 1}$. Since these factors of $\sigma (\lambda)$ are all in $L$, we see $\sigma (\lambda) \in L$. Thus $\sigma (L) \subseteq L$. Since $[L:K] = [\sigma (L): K]$ is finite, we obtain $\sigma (L) = L$.
\end{proof}

Now we introduce a new concept. Suppose $M_{i}$ ($i \in \mathfrak{I}$) are a family of matrices which are mutually commutative. All of them have dimension $m$, and their characteristic polynomials split over $L$. Then there is a subspace decomposition $L^{m} = \bigoplus_{s=1}^{r} V_{s}$ such that, for all $i$ and $s$, $M_{i}|_{V_{s}}$ has only one eigenvalue without counting multiplicity. If we require the number $r$ to be minimal, we call this decomposition a \textit{Jordan decomposition} of $\{ M_{i} \mid i \in \mathfrak{I} \}$. Obviously, a Jordan decomposition is unique.

Since the above $B_{j}$ are mutually commutative, there exists a Jordan decomposition of $\{ B_{j} \mid 1 \leq j \leq n \}$. Let $H$ be the subgroup of $\mathrm{GL} (m,K)$ generated by these $B_{j}$. Then $H$ shares the same Jordan decomposition as $\{ B_{j} \mid 1 \leq j \leq n \}$ does. If $H_{1} \le H$, then $H_{1}$ has a Jordan decomposition whose each summand is a sum of some of the summands of the Jordan decomposition of $H$.

\begin{proof}[Proof of Proposition \ref{prop_eigenvalue}]
Let $L^{m} = \bigoplus_{s=1}^{r} V_{s}$ be the Jordan decomposition of $H$, where $H$ is the subgroup generated by $\{ B_{j} \mid 1 \leq j \leq n \}$. Let $\lambda_{js}$ denote the eigenvalue of $B_{j}|_{V_{s}}$. It suffices to verify the integrity of $\lambda_{js}^{\pm 1}$.

Clearly, $L^{m} = \sigma L^{m} = \bigoplus_{s=1}^{r} \sigma V_{s}$ is the Jordan decomposition of $\sigma (H)$. On the other hand, each summand of the Jordan decomposition of $\{ u_{j} C \prod_{i=1}^{n} B_{i}^{a_{ij}} C^{-1} \mid 1 \leq j \leq n \}$ must be a sum of some $CV_{s}$. By \eqref{prop_eigenvalue_1}, we infer $\sigma V_{1}, \dots, \sigma V_{r}$ is just a permutation of $CV_{1}, \dots, CV_{r}$. In other words, there exists a permutation $\tau$ on $\{ 1, \dots, r \}$ such that $\sigma V_{s} = C V_{\tau(s)}$ for all $s$. By \eqref{prop_eigenvalue_1} again, for all $j$ and $s$,
\[
\sigma (\lambda_{js}) = u_{j} \prod_{i=1}^{n} \lambda_{i \tau(s)}^{a_{ij}}.
\]
Applying $\sigma$ to the above equation, by a direct calculation, we have
\[
\sigma^{2} (\lambda_{js}) = \sigma (u_{j}) \prod_{i'=1}^{n} u_{i'}^{a_{i'j}} \prod_{l=1}^{n} \lambda_{l \tau^{2}(s)}^{\sum_{i=1}^{n} a_{li} a_{ij}}.
\]
By induction, we see
\begin{equation}\label{prop_eigenvalue_2}
\sigma^{k} (\lambda_{js}) = u_{k,j} \prod_{i=1}^{n} \lambda_{i \tau^{k}(s)}^{\tilde{a}_{ij,k}},
\end{equation}
where $u_{k,j}$ are units in $R$ and $(\tilde{a}_{ij,k})_{n \times n} = A^{k}$. Clearly, $\tau^{k_{0}} = \mathrm{id}$ for some $k_{0} >0$, and $\bigcap_{k=1}^{\infty} \mathrm{im} A^{k_{0}k} = \bigcap_{k=1}^{\infty} \mathrm{im} A^{k} =0$. Replacing $(A, \sigma)$ with $(A^{k_{0}}, \sigma^{k_{0}})$ if necessary, replacing each $u_{j}$ with another unit in $R$ if necessary, we may assume, for all $j$ and $s$,
\[
\sigma (\lambda_{js}) = u_{j} \prod_{i=1}^{n} \lambda_{is}^{a_{ij}}.
\]
Now the integrity of $\lambda_{js}^{\pm 1}$ follows from the following Lemma \ref{lem_integrity}.
\end{proof}

\begin{lemma}\label{lem_integrity}
Let $L^{*} := L \setminus 0$. Suppose $(x_{1}, \dots, x_{n}) \in {L^{*}}^{n}$ such that
\begin{equation}\label{lem_integrity_1}
\sigma (x_{1}, \dots, x_{n}) := (\sigma x_{1}, \dots, \sigma x_{n}) = \left( u_{1} \prod_{i=1}^{n} x_{i}^{a_{i1}}, \dots, u_{n} \prod_{i=1}^{n} x_{i}^{a_{in}} \right),
\end{equation}
where $(a_{ij})_{n \times n}$ and all $u_{j}$ are as above. Then all $x_{i}^{\pm 1}$ are integral over $R$.
\end{lemma}
\begin{proof}
Let $\overline{R}$ be the integral closure of $R$ in $L$. It's equivalent to show all $x_{i}^{\pm 1}$ are in $\overline{R}$. Clearly, $\overline{R}$ is a domain with quotient field $L$ (cf. \cite[Corollary~5.3~\&~Proposition~5.12]{Atiyah_Macdonald}).

Since $R$ is Noetherian, by Mori-Nagata Theorem (\cite[Proposition~6]{Nishimura}), we have $\overline{R}$ is a Krull domain. By the definition of a Krull domain, there exists a family $\{ v_{\mathfrak{i}} \mid \mathfrak{i} \in \mathfrak{I} \}$ of discrete valuations $v_{\mathfrak{i}} \colon  L^{*} \rightarrow \mathbb{Z}$ such that
\begin{equation}\label{lem_integrity_2}
\overline{R} \setminus 0 = \bigcap_{\mathfrak{i} \in \mathfrak{I}} v_{\mathfrak{i}}^{-1} ([0, +\infty))
\end{equation}
and, for each $\mathfrak{l} \in L^{*}$, $v_{\mathfrak{i}} (\mathfrak{l}) = 0$ for all but finitely many $\mathfrak{i} \in \mathfrak{I}$. (See \cite[\S~12]{Matsumura2} for more about Krull domains, and see \cite[p.~94]{Atiyah_Macdonald} for the definition of a discrete valuation.) Thus, to finish the proof, it suffices to show $v_{\mathfrak{i}} (x) =0$ for all $\mathfrak{i} \in \mathfrak{I}$. Here $x=(x_{1}, \dots, x_{n})$ and $v_{\mathfrak{i}} (x) = (v_{\mathfrak{i}} (x_{1}), \dots, v_{\mathfrak{i}} (x_{n}))$. We consider $x$ and $v_{\mathfrak{i}} (x)$ as row vectors.

By \eqref{lem_integrity_1} and an inductive calculation similar to that for \eqref{prop_eigenvalue_2}, we obtian
\[
\sigma^{k} (x) = \left( u_{k,1} \prod_{i=1}^{n} x_{i}^{\tilde{a}_{i1,k}}, \dots, u_{k,n} \prod_{i=1}^{n} x_{i}^{\tilde{a}_{in,k}} \right)
\]
or
\[
x = \left( \sigma^{-k} (u_{k,1}) \prod_{i=1}^{n} \sigma^{-k} (x_{i})^{\tilde{a}_{i1,k}}, \dots, \sigma^{-k} (u_{k,n}) \prod_{i=1}^{n} \sigma^{-k} (x_{i})^{\tilde{a}_{in,k}} \right),
\]
where $u_{k,j}$ are units in $R$ and $(\tilde{a}_{ij,k})_{n \times n} = A^{k}$. Let $v$ be an arbitrary discrete valuation in \eqref{lem_integrity_2}, then
\begin{align*}
v (x) & = v (\sigma^{-k} (u_{k,1}), \dots, \sigma^{-k} (u_{k,n})) + \left( \sum_{i=1}^{n} v(\sigma^{-k} (x_{i})) \cdot \tilde{a}_{i1,k}, \dots, \sum_{i=1}^{n} v(\sigma^{-k} (x_{i})) \cdot \tilde{a}_{in,k}  \right) \\
& =  0 + v (\sigma^{-k} (x)) \cdot A^{k} = v (\sigma^{-k} (x)) \cdot A^{k}.
\end{align*}
Transposing the above matrices equation, we get
\[
v(x)^{T} = (A^{T})^{k} \cdot z
\]
for $z = v (\sigma^{-k} (x))^{T} \in \mathbb{Z}^{n}$. Thus $v (x)^{T} \in \mathrm{im} (A^{T})^{k}$ for all $k$. By Corollary \ref{cor_transpose}, we have $\bigcap_{k=1}^{\infty} \mathrm{im} (A^{T})^{k} =0$. Therefore, $v (x) =0$, which finishes the proof.
\end{proof}

\section{Algebraic Affine Tori}\label{sec_algebra_tori}
In this section, we shall prove the following Proposition \ref{prop_degree}. The proof is completely due to Botong Wang.

Suppose $K$ is a field, $\phi' \colon  K[t_{1}^{\pm 1}, \dots, t_{n}^{\pm 1}] \rightarrow K[t_{1}^{\pm 1}, \dots, t_{n}^{\pm 1}]$ is a ring monomorphism such that $\phi'|_{K} = \sigma$ and $\forall j$, $\phi'(t_{j}) = g_{j} \prod_{i=1}^{n} t_{i}^{a_{ij}}$, where $t_{1}, \dots, t_{n}$ are indeterminates, $g_{j} \in K^{*} := K \setminus 0$, $\sigma \in \mathrm{Aut}(K)$, and $(a_{ij})_{n \times n} = A$ satisfies Assumption \ref{asm_matrix}. For each $Q \in \mathrm{Spec} (K[t_{1}^{\pm 1}, \dots, t_{n}^{\pm 1}])$, let $\kappa (Q)$ be the residue field of $K[t_{1}^{\pm 1}, \dots, t_{n}^{\pm 1}]$ at $Q$.
\begin{proposition}\label{prop_degree}
Suppose $Q$ is a prime ideal of $K[t_{1}^{\pm 1}, \dots, t_{n}^{\pm 1}]$ and ${\phi'}^{-1} (Q) = Q$. Then $\phi'$ induces a field embedding $\phi'' \colon  \kappa (Q) \rightarrow \kappa (Q)$. Furthermore, if $Q$ is not maximal, then $[\kappa (Q): \mathrm{im} \phi''] >1$.
\end{proposition}

The verification of Proposition \ref{prop_degree} is via an argument of algebraic geometry. We shall actually translate this proposition to Proposition \ref{prop_variety} in algebraic geometry. The main task of this section is to prove that proposition. Since the above $K$ is not necessarily algebraically closed, it's convenient to employ the language of schemes.

All varieties of this section are irreducible and are considered as schemes. Actually, they are all irreducible and reduced affine schemes. The morphisms between two affine schemes are in $1-1$ correspondence with homomorphisms between coordinate rings (\cite[p.~79,~Theorem~1]{Mumford}). If $\psi$ is such a morphism, let $\psi^{*}$ denote its related map between coordinate rings. Recall that a morphism $\psi$ is over $K$, or a $K$-morphism, if and only if $\psi^{*}$ is a $K$-homomorphism, i.e. $\psi^{*}|_{K} = \mathrm{id}$.
\begin{definition}
Suppose $\rho \in \mathrm{Aut} (K)$ and $\psi \colon  X \rightarrow Y$ is a morphism between affine $K$-varieties. If $\psi^{*} (K) = K$ and $\psi^{*}|_{K} = \rho^{-1}$, then we call $\psi$ a \textit{$\rho$-twisted morphism}.
\end{definition}

Particularly, if $\rho = \mathrm{\mathrm{id}}$, then a $\rho$-twisted morphism is a $K$-morphism.

\begin{lemma}
If $\psi \colon  X \rightarrow Y$ is a $\rho$-twisted isomorphism, then $\psi^{-1}$ is a $\rho^{-1}$-twisted isomorphism.
\end{lemma}

An affine $K$-torus $\mathbb{G}_{m}^{n}$ is a scheme over $K$ which is isomorphic to $\mathrm{Spec}(K[t_{1}^{\pm 1}, \dots, t_{n}^{\pm 1}])$ via a $K$-isomorphism, where $t_{i}$'s are indeterminates. If $n=0$, $\mathbb{G}_{m}^{n}$ is a point. Obviously, $\mathbb{G}_{m}^{n}$ has dimension $n$ since
\[
\kappa (\mathbb{G}_{m}^{n}) \cong K(t_{1}, \dots, t_{n})
\]
has degree of transcendency $n$ over $K$. Here $\kappa (\mathbb{G}_{m}^{n})$ is the quotient field of the coordinate ring of $\mathbb{G}_{m}^{n}$.

In the following, $\mathbb{G}_{m}^{n} = \mathrm{Spec}(K[t_{1}^{\pm 1}, \dots, t_{n}^{\pm 1}])$ and $\mathbb{G}_{m}^{p} = \mathrm{Spec}(K[s_{1}^{\pm 1}, \dots, s_{p}^{\pm 1}])$.
\begin{lemma}\label{lem_morphism}
Suppose $\psi \colon  \mathbb{G}_{m}^{n} \rightarrow \mathbb{G}_{m}^{p}$ is a $\rho$-twisted morphism. Then $\psi$ is uniquely determined by $\psi^{*} (s_{j})$ for all $j$, and $\psi^{*} (s_{j}) = h_{j} \prod_{i=1}^{n} t_{i}^{m_{ij}}$ for some $h_{j} \in K^{*}$ and $m_{ij} \in \mathbb{Z}$.
\end{lemma}
\begin{proof}
By (\cite[p.~79,~Theorem~1]{Mumford}), we know $\psi$ is uniquely determined by $\psi^{*}$. Since $\psi^{*}|_{K} = \rho^{-1}$, it is uniquely determined by $\psi^{*} (s_{j})$.

Clearly, $\psi^{*} (s_{j})$ and $\psi^{*} (s_{j}^{-1})$ are Laurent polynomials. If the product of two nonzero Laurent polynomials is homogenous with respect to the multi-grading $\mathbb{Z}^{n}$, so are the factors. Since $\psi^{*} (s_{j}) \cdot \psi^{*} (s_{j}^{-1}) =1$ is homogenous, the conclusion follows.
\end{proof}

Let $\mathbf{M}_{\psi} = (m_{ij})_{n \times p}$ be the matrix with entries $m_{ij}$ in Lemma \ref{lem_morphism}. Clearly, $\mathbf{M}_{\psi} \in \mathrm{Hom} (\mathbb{Z}^{p}, \mathbb{Z}^{n})$, where elements in $\mathbb{Z}^{p}$ and $\mathbb{Z}^{n}$ are considered as column vectors.

\begin{lemma}\label{lem_matrix_morphism}
Suppose $\psi_{1} \colon  \mathrm{Spec}(K[u_{1}^{\pm 1}, \dots, u_{q}^{\pm 1}]) \rightarrow \mathbb{G}_{m}^{n}$ and $\psi_{2} \colon  \mathbb{G}_{m}^{n} \rightarrow \mathbb{G}_{m}^{p}$ are $\rho_{1}$-twisted and $\rho_{2}$-twisted morphisms, respectively. Then
\begin{enumerate}
\item $\psi_{2} \psi_{1}$ is a $(\rho_{2} \rho_{1})$-twisted morphism, $\mathbf{M}_{\psi_{2} \psi_{1}} = \mathbf{M}_{\psi_{1}} \mathbf{M}_{\psi_{2}}$;

\item $\psi_{2}$ is an isomorphism if and only if $\mathbf{M}_{\psi_{2}} \in \mathrm{Hom} (\mathbb{Z}^{p}, \mathbb{Z}^{n})$ is invertible.
\end{enumerate}
\end{lemma}
\begin{proof}
This follows from an easy calculation.
\end{proof}

\begin{lemma}\label{lem_tori_morphism}
Suppose $\psi \colon  \mathbb{G}_{m}^{n} \rightarrow \mathbb{G}_{m}^{p}$ is a $\rho$-twisted morphism. Then there exist $K$-automorphisms $\varphi_{1}$ and $\varphi_{2}$ of $\mathbb{G}_{m}^{n}$ and $\mathbb{G}_{m}^{p}$, respectively, such that $(\varphi_{2} \psi \varphi_{1})^{*} (s_{j}) = t_{j}^{d_{j}}$ and $d_{j} >0$ for $j \le r$, and $(\varphi_{2} \psi \varphi_{1})^{*} (s_{j}) = 1$ for $j>r$, where $0 \leq r \leq \min\{n, p \}$.
\end{lemma}
\begin{proof}
There exist $B \in \mathrm{Aut} (\mathbb{Z}^{p})$ and $C \in \mathrm{Aut} (\mathbb{Z}^{n})$ such that
\[
C \mathbf{M}_{\psi} B =
\begin{bmatrix}
D & 0 \\
0 & 0
\end{bmatrix},
\]
where $D = \mathrm{diag} (d_{1}, \dots, d_{r})$ and each $d_{i} >0$. Define a $K$-endomorphism $\varphi_{1}$ (resp. $\varphi_{21}$) of $\mathbb{G}_{m}^{n}$ (resp. $\mathbb{G}_{m}^{p}$) such that $\mathbf{M}_{\varphi_{1}} = C$ (resp. $\mathbf{M}_{\varphi_{21}} = B$). By Lemma \ref{lem_matrix_morphism}, $\varphi_{1}$ and $\varphi_{21}$ are $K$-isomorphisms, $(\varphi_{21} \psi \varphi_{1})^{*} (s_{j}) = h_{j} t_{j}^{d_{j}}$ for $j \leq r$, and $(\varphi_{21} \psi \varphi_{1})^{*} (s_{j}) = h_{j}$ for $j >r$, where $h_{j} \in K^{*}$ for all $j$. Define a $K$-morphism $\varphi_{22} \colon  \mathbb{G}_{m}^{p} \rightarrow \mathbb{G}_{m}^{p}$ such that $\varphi_{22}^{*} (s_{j}) = \rho (h_{j}^{-1}) s_{j}$. Then $\varphi_{22}$ is also an isomorphism. Define $\varphi_{2} = \varphi_{22} \varphi_{21}$. Now $\varphi_{1}$ and $\varphi_{2}$ are the desired ones.
\end{proof}

Recall that a morphism $\psi \colon  X \rightarrow Y$ between affine schemes is finite if $\mathcal{O} (X)$ is a finite $\psi^{*} \mathcal{O} (Y)$-module (see \cite[p.~124]{Mumford}). Here $\mathcal{O} (X)$ and $\mathcal{O} (Y)$ are the coordinate rings of $X$ and $Y$, respectively. Also recall that, if $\psi^{*}$ is injective, then $\mathrm{deg} (\psi) = [\kappa (X): \psi^{*} \kappa (Y)]$, where $\kappa (X)$ and $\kappa (Y)$ are the quotient field of $\mathcal{O} (X)$ and $\mathcal{O} (Y)$, respectively.
\begin{lemma}\label{lem_degree}
Suppose $\psi \colon  \mathbb{G}_{m}^{n} \rightarrow \mathbb{G}_{m}^{n}$ is a $\rho$-twisted morphism. If $|\det(\mathbf{M}_{\psi})| \neq 0$, then $\psi$ is finite and surjective, and $\mathrm{deg} (\psi) = |\det(\mathbf{M}_{\psi})|$.
\end{lemma}
\begin{proof}
By Lemma \ref{lem_tori_morphism}, there are $K$-automorphisms $\varphi_{1}$ and $\varphi_{2}$ such that $(\varphi_{2} \psi \varphi_{1})^{*} (t_{j}) = t_{j}^{d_{j}}$ and $d_{j} \ge 0$ for all $j$. By Lemma \ref{lem_matrix_morphism}, we have $|\det (\mathbf{M}_{\varphi_{1}})| = |\det (\mathbf{M}_{\varphi_{2}})| =1$ and
\[
\prod_{j=1}^{n} d_{j} = |\det (\mathbf{M}_{\varphi_{2} \psi \varphi_{1}})| = |\det (\mathbf{M}_{\varphi_{1}})| \cdot |\det (\mathbf{M}_{\psi})| \cdot |\det (\mathbf{M}_{\varphi_{2}})| = |\det (\mathbf{M}_{\psi})|.
\]
Now $d_{j} >0$ for all $j$ since $|\det(\mathbf{M}_{\psi})| \neq 0$. We infer $\varphi_{2} \psi \varphi_{1}$ is dominating and finite, and hence is surjective (\cite[p.~123,~Proposition~4]{Mumford}). Thus so is $\psi$. It's easy to see $\deg (\varphi_{2} \psi \varphi_{1}) = \prod_{j=1}^{n} d_{j}$. Since $\mathrm{deg} (\varphi_{2} \psi \varphi_{1}) = \mathrm{deg} (\psi)$, the proof is completed.
\end{proof}

\begin{definition}
If $W$ is a closed subscheme of $\mathbb{G}_{m}^{n}$ and $W$ is $K$-isomorphic to $\mathbb{G}_{m}^{r}$ for some $r \ge 0$, then we call $W$ an $r$-dimensional \textit{subtorus} of $\mathbb{G}_{m}^{n}$.
\end{definition}

\begin{lemma}\label{lem_morphism_image}
Suppose $\psi \colon  \mathbb{G}_{m}^{n} \rightarrow \mathbb{G}_{m}^{p}$ is a $\rho$-twisted morphism. Then $\mathrm{im} \psi$ is a subtorus of $\mathbb{G}_{m}^{p}$. Furthermore, $[\overline{\psi^{*} \kappa} : \psi^{*} \kappa] < \infty$, where $\psi^{*} \kappa$ is the image of $\psi^{*} \colon  \kappa (\mathrm{im} \psi) \rightarrow \kappa (\mathbb{G}_{m}^{n})$, and $\overline{\psi^{*} \kappa}$ is the algebraic closure of $\psi^{*} \kappa$ in $\kappa (\mathbb{G}_{m}^{n})$.
\end{lemma}
\begin{proof}
Similar to the proof of Lemma \ref{lem_degree}, by composing $K$-isomorphisms if necessary, we may assume $\psi^{*} (s_{j}) = t_{j}^{d_{j}}$ with $d_{j} >0$ for $j \le r$ and $\psi^{*} (s_{j}) = 1$ for $j>r$. Let $\mathbb{G}_{m}^{r} = \mathrm{Spec}(K[t_{1}^{\pm 1}, \dots, t_{r}^{\pm 1}])$ and $\psi_{1} \colon  \mathbb{G}_{m}^{n} \rightarrow \mathbb{G}_{m}^{r}$ be the projection, then $\psi_{1}$ is surjective. Let $W:= \mathrm{Spec}(K[s_{1}^{\pm 1}, \dots, s_{r}^{\pm 1}])$, then $W$ is a subtorus of $\mathbb{G}_{m}^{p}$ and $\mathrm{im} \psi \subseteq W$.

Define a $\rho$-twisted morphism $\psi_{2} \colon  \mathbb{G}_{m}^{r} \rightarrow W$ such that $\psi_{2}^{*} (s_{j}) = t_{j}^{d_{j}}$ for $j \le r$. Then $\psi = \psi_{2} \psi_{1}$. Since $\psi_{2}$ is dominating and finite, it is surjective. So $\mathrm{im} \psi = \mathrm{im} \psi_{2} = W$.

Note that $\psi_{1}^{*} \kappa (\mathbb{G}_{m}^{r})$ is algebraically closed in $\kappa (\mathbb{G}_{m}^{n})$, and $[\kappa (\mathbb{G}_{m}^{r}): \psi_{2}^{*} \kappa (W)]$ is finite. Thus $\psi^{*} \kappa = \psi_{1}^{*} \psi_{2}^{*} \kappa (W)$ and $\overline{\psi^{*} \kappa} = \psi_{1}^{*} \kappa (\mathbb{G}_{m}^{r})$, which finishes the proof.
\end{proof}

\begin{lemma}\label{lem_subtorus}
A closed subscheme $W$ of $\mathbb{G}_{m}^{n}$ is an $r$-dimensional subtorus if and only if there exists a $K$-automorphism $\varphi$ of $\mathbb{G}_{m}^{n}$ such that $\varphi (W) = \mathbb{G}_{m}^{r} := \mathrm{Spec}(K[t_{1}^{\pm 1}, \dots, t_{r}^{\pm 1}])$.
\end{lemma}
\begin{proof}
It suffices to prove ``$\Rightarrow$". By the assumption, there exists a closed immersion
\[
\psi \colon \ \mathbb{G}_{m}^{r} = \mathrm{Spec}(K[t_{1}^{\pm 1}, \dots, t_{r}^{\pm 1}]) \rightarrow \mathbb{G}_{m}^{n} = \mathrm{Spec}(K[t_{1}^{\pm 1}, \dots, t_{n}^{\pm 1}])
\]
such that $\mathrm{im} \psi = W$. By Lemma \ref{lem_tori_morphism}, there exists a $K$-automorphism $\varphi_{1}$ (resp.~$\varphi$) of $\mathbb{G}_{m}^{r}$ (resp.~$\mathbb{G}_{m}^{n}$) such that $(\varphi \psi \varphi_{1})^{*} (t_{j}) = t_{j}^{d_{j}}$ with $d_{j} \ge 0$ for $j \le r$ and $(\varphi \psi \varphi_{1})^{*} (t_{j}) = 1$ for $j>r$. Since $\psi$ is a closed immersion, so is $\varphi \psi \varphi_{1}$, i.e. $(\varphi \psi \varphi_{1})^{*}$ is surjective. We infer each $d_{j} =1$. Then $\varphi (W) = \mathrm{im} (\varphi \psi \varphi_{1}) = \mathbb{G}_{m}^{r}$.
\end{proof}

\begin{lemma}\label{lem_subtorus_degree}
Suppose $\psi$ is a $\rho$-twisted endomorphism of $\mathbb{G}_{m}^{n}$, and $A := \mathbf{M}_{\psi}$ satisfies Assumption \ref{asm_matrix}. If $W$ is subtorus, $\dim W >0$, and $\psi (W) = W$, then $\mathrm{deg} (\psi|_{W}) >1$.
\end{lemma}
\begin{proof}
By Lemma \ref{lem_subtorus}, there exists a $K$-automorphism $\varphi$ of $\mathbb{G}_{m}^{n}$ such that $\varphi (W) = \mathbb{G}_{m}^{r}$. Then $\varphi \psi \varphi^{-1} (\mathbb{G}_{m}^{r}) = \mathbb{G}_{m}^{r}$. Therefore, for some matrices $A_{i}$ with integer entries,
\[
\mathbf{M}_{\varphi \psi \varphi^{-1}} =
\begin{bmatrix}
A_{1} & 0 \\
A_{2} & A_{3}
\end{bmatrix},
\]
where $A_{1}$ is the matrix corresponding to $\varphi \psi \varphi^{-1}|_{\mathbb{G}_{m}^{r}}$.

By Lemma \ref{lem_matrix_morphism}, $\mathbf{M}_{\varphi \psi \varphi^{-1}}$ is conjugate to $A= \mathbf{M}_{\psi}$. Since $\det (A) \ne 0$ by Assumption \ref{asm_matrix}, we have $\det (A_{1}) \ne 0$. We claim $|\det (A_{1})| >1$. Suppose not, then $|\det (A_{1})| =1$. So, for each eigenvalue $\alpha$ of $A_{1}$, $\alpha^{-1}$ is integral over $\mathbb{Z}$. On the other hand, $\bigcap_{k=1}^{\infty} \mathrm{im} A^{k} =0$ by Assumption \ref{asm_matrix}. Since $\alpha$ is also an eigenvalue of $A$, this contradicts Lemma \ref{lem_matrix_eigenvalue}.

By Lemma \ref{lem_degree}, we infer $\mathrm{deg} (\psi|_{W}) = \mathrm{deg} (\varphi \psi \varphi^{-1}|_{\mathbb{G}_{m}^{r}}) = |\det (A_{1})| >1$.
\end{proof}

To prove Proposition \ref{prop_variety}, we also need a general construction. Suppose $V$ is a variety over $K$. We shall construct an affine torus $T(V)$ which is called the \textit{intrinsic torus} of $V$ (see e.g. \cite[p.~280]{Maclagan_Sturmfels}). Let $\mathcal{O}^{*} (V)$ denote the multiplicative group of $\mathcal{O} (V)$, i.e. $\mathcal{O}^{*} (V)$ consists of invertible elements of multiplication of $\mathcal{O} (V)$.

The following Proposition \ref{prop_torus} is a slight extension of \cite[Proposition~6.4.4]{Maclagan_Sturmfels} to the case of twisted morphisms.
\begin{proposition}\label{prop_torus}
Suppose $V$ is an affine variety over $K$. Then there exists an affine torus $T(V)$ and a $K$-morphism $\tau_{V} \colon  V \rightarrow T(V)$ such that, for each $\rho$-twisted morphism $\alpha \colon  V \rightarrow \mathbb{G}_{m}^{n}$, there exists a unique $\rho$-twisted morphism $\tilde{\alpha} \colon  T(V) \rightarrow \mathbb{G}_{m}^{n}$ such that the following diagram commutes.
\[
  \xymatrix{
  V \ar[r]^{\tau_{V}} \ar[dr]_{\alpha} & T(V) \ar[d]^{\tilde{\alpha}} \\
  & \mathbb{G}_{m}^{n} }
\]
\end{proposition}
\begin{proof}
By \cite[p.~28,~Lemma]{Rosenlicht} (see also \cite[Lemma~6.4.1]{Maclagan_Sturmfels}), the group $\mathcal{O}^{*} (V)/K^{*} \cong \mathbb{Z}^{p}$ for some $p \geq 0$. We choose $f_{1}, \dots, f_{p} \in \mathcal{O}^{*} (V)$ representing a basis of $\mathcal{O}^{*} (V)/K^{*}$. Let $T(V) = \mathrm{Spec}(K[s_{1}^{\pm 1}, \dots, s_{p}^{\pm 1}])$. Define a $K$-morphism $\tau_{V} \colon  V \rightarrow T(V)$ as $\tau_{V}^{*} (s_{j}) = f_{j}$. Suppose $\mathbb{G}_{m}^{n} = \mathrm{Spec}(K[t_{1}^{\pm 1}, \dots, t_{n}^{\pm 1}])$. Then $\alpha^{*} (t_{j}) \in \mathcal{O}^{*} (V)$ since $t_{j} \in \mathcal{O}^{*} (\mathbb{G}_{m}^{n})$. Thus $\alpha^{*} (t_{j}) = h_{j} \prod_{i=1}^{p} f_{i}^{m_{ij}}$ with $h_{j} \in K^{*}$. We define a $\rho$-twisted morphism $\tilde{\alpha} \colon  \mathbb{G}_{m}^{n} \rightarrow T(V)$ such that $\tilde{\alpha}^{*} (t_{j}) = h_{j} \prod_{i=1}^{p} s_{i}^{m_{ij}}$. Clearly, $\alpha = \tilde{\alpha} \tau_{V}$ since $\mathbb{G}_{m}^{n}$ is affine.

For the uniqueness of $\tilde{\alpha}$, note that $\tau_{V}^{*} \colon  \mathcal{O}^{*} (T(V)) \rightarrow \mathcal{O}^{*} (V)$ is a group isomorphism. Since $t_{j} \in \mathcal{O}^{*} (\mathbb{G}_{m}^{n})$, we infer $\tilde{\alpha}^{*} (t_{j}) \in \mathcal{O}^{*} (T(V))$ no matter how $\tilde{\alpha}$ is defined. Thus $\tilde{\alpha}^{*} (t_{j}) = (\tau_{V}^{*})^{-1} \alpha^{*} (t_{j})$ is uniquely determined.
\end{proof}

\begin{lemma}\label{lem_functor}
For each $\rho$-twisted morphism $\psi \colon  V_{1} \rightarrow V_{2}$, there exists a unique $\rho$-twisted morphism $T(\psi) \colon  T(V_{1}) \rightarrow T(V_{2})$ such that the following diagram commutes. If $\psi$ is the identity or an isomorphism, so is $T(\psi)$.
\[
\xymatrix{
V_{1} \ar[d]_{\tau_{V_{1}}} \ar[r]^{\psi} & V_{2} \ar[d]^{\tau_{V_{2}}} \\
T(V_{1}) \ar[r]^{T(\psi)} & T(V_{2}) }
\]
\end{lemma}
\begin{proof}
Consider the morphism $\tau_{V_{2}} \psi$, by Proposition \ref{prop_torus}, the existence and uniqueness of $T(\psi)$ is trivial. The uniqueness also implies that $T(\psi)$ is the identity if so is $\psi$.

Also by the uniqueness, we see $T(\psi_{1}) T(\psi_{2}) = T(\psi_{1} \psi_{2})$ for two composable twisted morphisms $\psi_{1}$ and $\psi_{2}$. Thus, if $\psi$ is an isomorphism, so is $T(\psi)$.
\end{proof}

\begin{lemma}\label{lem_diagram_lifting}
If $\varphi$, $\psi$, $\alpha$, and $\beta$ are twisted, and $\varphi \alpha = \beta \psi$, then the following diagram commutes.
\[
\xymatrix{
V_{1} \ar[dd]_{\alpha} \ar[rrr]^{\psi} \ar[dr]^{\tau_{V_{1}}} & & & V_{2} \ar[dd]^{\beta} \ar[dl]_{\tau_{V_{2}}} \\
& T(V_{1}) \ar[dl]^{\tilde{\alpha}} \ar[r]^{T(\psi)} & T(V_{2}) \ar[dr]_{\tilde{\beta}} & \\
\mathbb{G}_{m}^{n} \ar[rrr]^{\varphi} & & & \mathbb{G}_{m}^{p} }
\]
\end{lemma}
\begin{proof}
By Proposition \ref{prop_torus} and Lemma \ref{lem_functor}, it remains to show $\varphi \tilde{\alpha} = \tilde{\beta} T(\psi)$. By a diagram chasing, we see $\varphi \tilde{\alpha} \tau_{V_{1}} = \tilde{\beta} T(\psi) \tau_{V_{1}}$. Then $\tau_{V_{1}}^{*} \tilde{\alpha}^{*} \varphi^{*} = \tau_{V_{1}}^{*} T(\psi)^{*} \tilde{\beta}^{*}$. Suppose $\mathcal{O}(\mathbb{G}_{m}^{p}) = K[s_{1}^{\pm 1}, \dots, s_{p}^{\pm 1}]$. Then both $\tilde{\alpha}^{*} \varphi^{*} (s_{j})$ and $T(\psi)^{*} \tilde{\beta}^{*} (s_{j})$ are in $\mathcal{O}^{*} (T(V_{1}))$. Since $\tau_{V_{1}}^{*} \colon  \mathcal{O}^{*} (T(V_{1})) \rightarrow \mathcal{O}^{*} (V_{1})$ is injective, $\tilde{\alpha}^{*} \varphi^{*} (s_{j}) = T(\psi)^{*} \tilde{\beta}^{*} (s_{j})$ and, therefore, $\tilde{\alpha}^{*} \varphi^{*} = T(\psi)^{*} \tilde{\beta}^{*}$. This further implies $\varphi \tilde{\alpha} = \tilde{\beta} T(\psi)$ because $\mathbb{G}_{m}^{p}$ is affine.
\end{proof}

\begin{proposition}\label{prop_variety}
Suppose $\psi$ is a $\rho$-twisted endomorphism of $\mathbb{G}_{m}^{n}$ such that $A := \mathbf{M}_{\psi}$ satisfies Assumption \ref{asm_matrix}. Suppose $V$ is a closed subvariety of $\mathbb{G}_{m}^{n}$ such that $\dim V >0$ and $\psi (V) = V$. Then $\mathrm{deg} (\psi|_{V}) >1$.
\end{proposition}
\begin{proof}
We prove by contradiction. Let's assume $\mathrm{deg} (\psi|_{V}) =1$.

Let $\overline{V} \rightarrow V$ be a normalization, i.e.~$\mathcal{O}(\overline{V})$ is the integral closure of $\mathcal{O}(V)$ in $\kappa (V)$. Then $\psi|_{V} \colon  V \rightarrow V$ is lifted to a $\rho$-twisted $\bar{\psi}$. The following diagram commutes.
\[
\xymatrix{
\overline{V} \ar[r]^{\bar{\psi}} \ar[d] & \overline{V} \ar[d] \\
V \ar[r]^{\psi|_{V}} & V }
\]
Since $\det (A) \neq 0$, by Lemma \ref{lem_degree}, $\psi$ and hence $\psi|_{V}$ are finite. Since $\mathcal{O} (V)$ is a finitely generated $K$-algebra, $\overline{V} \rightarrow V$ is also finite (\cite[(31.H)]{Matsumura1}). Therefore, $\bar{\psi}$ is finite. We would have $\mathrm{deg} (\bar{\psi}) =1$ because $\mathrm{deg} (\psi|_{V}) =1$. By Zariski's Main Theorem (\cite[p.~209,~I]{Mumford} and \cite[p.~280]{Hartshorne}), we infer $\bar{\psi}$ is an isomorphism. (Note that Zariski's Main Theorem usually deals with $K$-morphisms. Our $\rho$-twisted case follows from the usually one immediately. There is an evident $\rho^{-1}$-twisted automorphism of $\mathbb{G}_{m}^{n}$ which induces a $\rho^{-1}$-twisted automorphism $\varphi$ of $\overline{V}$. Then $\bar{\psi} \varphi$ is a $K$-morphism, and we apply the theorem to $\bar{\psi} \varphi$.)

By Lemma \ref{lem_diagram_lifting}, we obtain the following commutative diagram.
\[
\xymatrix{
\overline{V} \ar[d] \ar[rrr]^{\bar{\psi}}_{\cong} \ar[dr] & & & \overline{V} \ar[d] \ar[dl] \\
V \ar[d] & T(\overline{V}) \ar[dl]^{\theta} \ar[r]^-{T(\bar{\psi})}_-{\cong} & T(\overline{V}) \ar[dr]_{\theta} & V \ar[d] \\
\mathbb{G}_{m}^{n} \ar[rrr]^{\psi} & & & \mathbb{G}_{m}^{n} }
\]
By Lemma \ref{lem_functor}, since $\bar{\psi}$ is an isomorphism, so is $T(\bar{\psi})$. Thus
\[
\psi (\mathrm{im} \theta) = \mathrm{im} (\psi \theta) = \mathrm{im} (\theta T(\bar{\psi})) = \mathrm{im} \theta.
\]
Let $W$ denote $\mathrm{im} \theta$, by Lemma \ref{lem_morphism_image}, $W$ is a subtorus of $\mathbb{G}_{m}^{n}$. We obtain the following commutative diagram, where all morphisms are surjective.
\begin{equation}\label{prop_variety_1}
\xymatrix{
T(\overline{V}) \ar[d]_{\theta} \ar[r]^-{T(\bar{\psi})}_-{\cong} & T(\overline{V}) \ar[d]^{\theta} \\
W \ar[r]^{\psi|_{W}} & W }
\end{equation}

Let $\theta^{*} \kappa$ denote the image of $\theta^{*} \colon  \kappa (W) \rightarrow \kappa (T(\overline{V}))$. Let $\overline{\theta^{*} \kappa}$ be the algebraic closure of $\theta^{*} \kappa$ in $\kappa (T(\overline{V}))$. Similarly, we define $\psi|_{W}^{*} \kappa$, $(\psi|_{W} \theta)^{*} \kappa$ and $\overline{(\psi|_{W} \theta)^{*} \kappa}$. Since $\psi|_{W}$ is finite, we infer $\overline{\theta^{*} \kappa} = \overline{(\psi|_{W} \theta)^{*} \kappa}$. By Lemma \ref{lem_morphism_image}, $[\overline{\theta^{*} \kappa}: \theta^{*} \kappa] < \infty$. Furthermore, $\dim W >0$ since $W \supseteq V$ and $\dim V >0$. By Lemma \ref{lem_subtorus_degree}, we see $\deg(\psi|_{W}) >1$. Put these data into \eqref{prop_variety_1}, we would have
\begin{align*}
[\overline{\theta^{*} \kappa}: \theta^{*} \kappa] & < [\overline{\theta^{*} \kappa}: \theta^{*} \kappa] \cdot \deg(\psi|_{W}) = [\overline{\theta^{*} \kappa}: \theta^{*} \kappa] \cdot [\kappa (W): \psi|_{W}^{*} \kappa] \\
& = [\overline{(\psi|_{W} \theta)^{*} \kappa}: (\psi|_{W} \theta)^{*} \kappa] = [\overline{(\theta T(\bar{\psi}))^{*} \kappa} : (\theta T(\bar{\psi}))^{*} \kappa] = [\overline{\theta^{*} \kappa}: \theta^{*} \kappa],
\end{align*}
which is a contradiction.
\end{proof}

Proposition \ref{prop_degree} is nothing but an algebraic formulation of Proposition \ref{prop_variety}.
\begin{proof}[Proof of Proposition \ref{prop_degree}]
Now $\phi'$ defines a $\sigma^{-1}$-twisted finite morphism $\psi \colon  \mathbb{G}_{m}^{n} \rightarrow \mathbb{G}_{m}^{n}$ with $\psi^{*} = \phi'$. The ideal $Q$ defines a closed subvariety $V$ of $\mathbb{G}_{m}^{n}$. Since $\psi$ is finite and $(\phi')^{-1} (Q) = Q$, we have $\psi (V) = V$. Then $\deg (\psi|_{V}) = [\kappa (Q): \mathrm{im} \phi'']$. If $Q$ is  not maximal, then $\dim V >0$. The conclusion follows from Proposition \ref{prop_variety}.
\end{proof}

\section{CW Complexes}\label{sec_CW}
In this section, we shall prove Theorem \ref{thm_homotopy_finite} and Corollary \ref{cor_group}. They follow from the more general Theorem \ref{thm_homology_finite} below on homological finiteness which is also a generalization of \cite[Theorem~3.2]{Qin_Su_Wang}.

Suppose $X$ is a CW complex, $q \colon  X' \rightarrow X$ is a nontrivial covering, and $h \colon  X \rightarrow X'$ is a continuous map preserving base points. Suppose $h$ induces an isomorphism
\[
h_{\#} \colon  \pi_{1} (X) \overset{\cong}{\rightarrow} \pi_{1} (X').
\]
Let $G := \bigcap_{k=1}^{\infty} \mathrm{im} h_{\#}^{k}$ and $\overline{X}$ be the cover of $X$ with $\pi_{1} (\overline{X}) =G$.
\begin{theorem}\label{thm_homology_finite}
Let $X$, $X'$, $h$, $G$ and $\overline{X}$ be as the above, here $\pi_{1} (X)$ is unnecessarily abelian. Let $G_{0}$ be a subgroup of $G$. Suppose further $\pi_{1} (X)$ is finitely generated, $[\pi_{1} (X), \pi_{1} (X)] \le G_{0}$ and $h_{\#} (G_{0}) = G_{0}$. Let $R_{0}$ be a Noetherian commutative ring, and let $\pi_{1} (\overline{X}) =G$ act on $R:= R_{0} [G/G_{0}]$ by multiplication. If $X$ has only finitely many $m$-cells and $h \colon  X \rightarrow X'$ is $(m+1)$-connected, then $H_{m} (\overline{X};R)$ is a finite $R$-module, where $H_{m} (\overline{X};R)$ is the homology with local coefficient system $R$.
\end{theorem}
\begin{proof}
We have the following exact sequence of finitely generated abelian groups
\[
0 \rightarrow G/G_{0} \rightarrow \pi_{1} (X)/G_{0} \rightarrow \pi_{1} (X)/G \rightarrow 0.
\]
By Lemma \ref{lem_fundamental_group} below, $\pi_{1} (X)/G \cong \mathbb{Z}^{n}$ for some $n>0$. Thus $\pi_{1} (X)/G_{0} \cong G/G_{0} \times \mathbb{Z}^{n}$. By $h_{\#} (G) = G$ (Lemma \ref{lem_fundamental_group}) and the assumption $h_{\#} (G_{0}) = G_{0}$, we see $h_{\#}$ induces an endomorphism $\varphi$ of $G/G_{0} \times \mathbb{Z}^{n}$ such that $\varphi (G/G_{0}) = G/G_{0}$. Since $h_{\#}$ is injective, we infer $h_{\#}^{-1} (G_{0}) = G_{0}$ and $\varphi$ is injective. The fact $\bigcap_{k=1}^{\infty} \mathrm{im} h_{\#}^{k} =G$ implies that $\bigcap_{k=1}^{\infty} \mathrm{im} \varphi^{k} = G/G_{0}$. Consequently, the composition of the following maps
\[
\mathbb{Z}^{n} \rightarrow G/G_{0} \times \mathbb{Z}^{n} \overset{\varphi}{\longrightarrow} G/G_{0} \times \mathbb{Z}^{n} \rightarrow \mathbb{Z}^{n}
\]
can be represented by a matrix $A \in \mathrm{End} (\mathbb{Z}^{n})$ such that $\det (A) \neq 0$ and $\bigcap_{k=1}^{\infty} \mathrm{im} A^{k} =0$.

Now $R= R_{0} [G/G_{0}]$ and $S: = R_{0} [(G/G_{0}) \times \mathbb{Z}^{n}]$ are commutative Noetherian rings because so is $R_{0}$. Obviously,
\[
S= R[\mathbb{Z}^{n}] = R[t_{1}^{\pm 1}, \dots, t_{n}^{\pm 1}]
\]
is the Laurent polynomial ring over $R$ with indeterminates $t_{1}, \dots, t_{n}$. Here the indeterminates correspond to the natural basis of $\mathbb{Z}^{n}$.

Clearly, the cellular chain complex $C_{\bullet} (\overline{X}; R)$ with local coefficient system $R$ is a chain complex of free $S$-module with a basis corresponding to the cells of $X$. Since $X$ has only finite many $m$-cells, we see $C_{m} (\overline{X}; R)$ is finite over $S$. Since $S$ is Noetherian, $H_{m} (\overline{X}; R)$ is a finite $S$-module.

We have to show a stronger conclusion that $H_{m} (\overline{X}; R)$ is actually a finite $R$-module. Note that the group monomorphism $\varphi \colon  G/G_{0} \times \mathbb{Z}^{n} \rightarrow G/G_{0} \times \mathbb{Z}^{n}$ induces a ring monomorphism $\phi \colon  S \rightarrow S$ such that $\phi (R) =R$. For $1 \leq j \leq n$, we have
\[
\phi (t_{j}) = \varphi (t_{j}) = g_{j} \prod_{i=1}^{n} t_{i}^{a_{ij}},
\]
where $(a_{ij})_{n \times n} =A$ is the matrix above and $g_{j} \in G/G_{0}$ is a unit of $R$.

Let $Y$ be the cover of $X$ with $\pi_{1} (Y) = G_{0}$. By assumption, $h_{\#} \colon  \pi_{1} (X) \rightarrow \pi_{1} (X')$ is an isomorphism such that $h_{\#} (G) = G$ and $h_{\#} (G_{0}) = G_{0}$. The $(m+1)$-connected map $h \colon  X \rightarrow X'$ is lifted to $(m+1)$-connected maps $\bar{h} \colon  \overline{X} \rightarrow \overline{X}$ and $\bar{h}_{Y} \colon  Y \rightarrow Y$. This $\bar{h}$ induces an automorphism $\bar{h}_{\#}$ of $\pi_{1} (\overline{X}) = G$. Clearly, $\bar{h}_{\#} = {h_{\#}}|_{G}$. Meanwhile, by Hurewicz Theorem (\cite[Theorem~(7.1)]{G.Whitehead}) and the Universal Coefficient Theorem, the $(m+1)$-connected $\bar{h}$ induces a group isomorphism
\[
\bar{h}_{*} \colon \ H_{m} (\overline{X}; R) = H_{m} (Y; R_{0}) \rightarrow H_{m} (\overline{X}; R) = H_{m} (Y; R_{0})
\]
which is $\phi$-twisted linear, i.e. $\forall s \in S$, $\forall x \in H_{m} (\overline{X}; R)$, we have $\bar{h}_{*} (sx) = \phi (s) \bar{h}_{*} (x)$. Now we finish the proof by taking the $R$, $S$, $\mathfrak{M}$, $\phi$ and $\eta$ in Theorem \ref{thm_finite_module} as $R$, $S$, $H_{m} (\overline{X}; R)$, $\phi$ and $\bar{h}_{*}$ here.
\end{proof}

\begin{lemma}\label{lem_free_abelian}
Suppose $H$ is a finitely generated nontrivial abelian group. Suppose
\[
H_{1} \ge H_{2} \ge \cdots \ge H_{k} \ge \cdots
\]
is a nested sequence of subgroups of $H$ such that $\bigcap_{k=1}^{\infty} H_{k} = 0$ and $H_{k} \cong H$ for each $k$. Then $H \cong \mathbb{Z}^{n}$ for some $n>0$.
\end{lemma}
\begin{proof}
It suffices to show $\mathrm{Tor} H =0$, where $\mathrm{Tor} H$ is the torsion of $H$. Since $\mathrm{Tor} H$ is finite, $H_{k} \ge H_{k+1}$ and $\bigcap_{k=1}^{\infty} H_{k} = 0$, we infer $H_{k_{0}} \cap \mathrm{Tor} H = 0$ for some $k_{0}$. Thus $H_{k_{0}}$ is torsion free, and hence so is $H$.
\end{proof}

\begin{lemma}\label{lem_fundamental_group}
Let $X$, $X'$, $h$ and $G$ be the ones in Theorem \ref{thm_homology_finite}, and $[\pi_{1} (X), \pi_{1} (X)] \le G$. Then $h_{\#} (G) = G$, $\pi_{1} (X)/G \cong \mathbb{Z}^{n}$ for some $n>0$, and $X'$ is a finite cover of $X$.
\end{lemma}
\begin{proof}
Since $h_{\#}$ is injective, we have
\[
h_{\#} (G)= h_{\#} \left( \bigcap_{k=1}^{\infty} \mathrm{im} h_{\#}^{k} \right) = \bigcap_{k=1}^{\infty} \mathrm{im} h_{\#}^{k+1} =G
\]
and $h_{\#}^{-1} (G) = G$. Thus $h_{\#}$ induces a monomorphism $[h_{\#}] \colon  \pi_{1} (X)/G \rightarrow \pi_{1} (X)/G$ such that $\bigcap_{k=1}^{\infty} \mathrm{im} [h_{\#}]^{k} =0$. Since $\pi_{1} (X)$ is finitely generated, $[\pi_{1} (X), \pi_{1} (X)] \le G$ and $G \ne \pi_{1} (X)$, we have $\pi_{1} (X)/G$ is a finitely generated nontrivial abelian group. By Lemma \ref{lem_free_abelian}, we see $\pi_{1} (X)/G \cong \mathbb{Z}^{n}$ for some $n>0$. By $G \le \mathrm{im} h_{\#}$, we infer
\[
[\pi_{1} (X): \mathrm{im} h_{\#}] = [\pi_{1} (X)/G : \mathrm{im} [h_{\#}]] < \infty,
\]
which implies that $X'$ is a finite cover.
\end{proof}

Now we can prove a theorem on homotopy finiteness which is more general than Theorem \ref{thm_homotopy_finite}.
\begin{theorem}\label{thm_homotopy_skeleton}
Let $X$, $X'$, $h$ and $\overline{X}$ be the ones in Theorem \ref{thm_homology_finite}. Suppose $\pi_{1} (X)$ is finitely generated and abelian. Suppose the $m$-skeleton of $X$ is finite, and $h \colon  X \rightarrow X'$ is $(m+1)$-connected. Then $\overline{X}$ is homotopy equivalent to CW complex whose $m$-skeleton is finite.
\end{theorem}
\begin{proof}
Clearly, both $X$ and $\overline{X}$ are connected. Since $\pi_{1} (X)$ is finitely generated and abelian, so is $\pi_{1} (\overline{X})$. Thus the conclusion is trivial when $m \le 1$.

Now we assume $m \ge 2$. Since $G= \pi_{1} (\overline{X})$ is finitely presented and $\mathbb{Z}[G]$ is a Noetherian ring, by Wall's Theorems A and B in \cite{Wall65}, it suffices to show that
\[
H_{i} (\widetilde{X}; \mathbb{Z}) = H_{i} (\overline{X}; \mathbb{Z}[G])
\]
is a finite $\mathbb{Z}[G]$-module for each $i \le m$, where $\widetilde{X}$ is the universal cover of $X$. Taking $R_{0} = \mathbb{Z}$ and $G_{0} =0$ in Theorem \ref{thm_homology_finite}, we finish the proof.
\end{proof}

It's ready to prove Theorem \ref{thm_homotopy_finite} and Corollary \ref{cor_group}.
\begin{proof}[Proof of Theorem \ref{thm_homotopy_finite}]
(1). This follows from Lemma \ref{lem_fundamental_group}.

(2). Suppose $X$ is homotopy equivalent to a CW complex of finite type, we may assume $X$ itself is of finite type. By Theorem \ref{thm_homotopy_skeleton}, we have $\overline{X}$ satisfies the condition $\mathrm{F}_{m}$ in \cite{Wall65} for each $m$. By Wall's Theorems A in \cite{Wall65}, we infer $\overline{X}$ is homotopy equivalent to a CW complex of finite type.

If $X$ is finitely dominated, by Wall's Theorem A in \cite{Wall65}, $X$ is homotopy equivalent to a CW complex of finite type. By the above arguments, so is $\overline{X}$. Clearly, $X$ satisfies the condition $\mathrm{D}_{k}$ in Wall's Theorem F for some $k$, hence so does $\overline{X}$. By Wall's Theorems A and F again, $\overline{X}$ is finitely dominated.
\end{proof}

\begin{proof}[Proof of Corollary \ref{cor_group}]
Let $X$ be the Eilenberg-MacLane space $K(\pi, 1)$, and let $X'$ be the cover of $X$ with $\pi_{1} (X') = \mathrm{im} \varphi$. Then there is a homotopy equivalence $h \colon  X \rightarrow X'$ such that $h_{\#} \colon  \pi_{1} (X) \rightarrow \pi_{1} (X')$ equals $\varphi$. Since $\pi /G$ is abelian, $[\pi, \pi] \le G$. By Lemma \ref{lem_fundamental_group}, we have $\pi / G\cong \mathbb{Z}^{n}$ for some $n>0$ and $[\pi : \mathrm{im} \varphi] < \infty$.

Since $\pi$ is finitely generated, we may assume $X$ has only finitely many $1$-cells. Taking $R_{0} = \mathbb{Z}$ and $G_{0} = G$ in Theorem \ref{thm_homology_finite}, we have
\[
G/[G,G] = H_{1} (\overline{X}; \mathbb{Z}) = H_{1} (\overline{X}; \mathbb{Z} [G/G])
\]
is finitely generated, where $\overline{X}$ is the cover of $X$ with $\pi_{1} (\overline{X}) = G$.
\end{proof}

\section{Poincar\'{e} Duality Spaces}\label{sec_Poincare}
In this section, we shall prove Theorem \ref{thm_Poincare}, Corollary \ref{cor_torus} and other related results.

\begin{proof}[Proof of Theorem \ref{thm_Poincare}]
(1). This follows from the (1) in Theorem \ref{thm_homotopy_finite}.

(2). Also by the (1) in Theorem \ref{thm_homotopy_finite}, there is an epimorphism $\psi \colon  \pi_{1} (X) \rightarrow \mathbb{Z}^{n}$ such that $\ker \psi = G$. Since $\mathbb{T}^{n}$ is the Eilenberg-MacLane space $K(\mathbb{Z}^{n}, 1)$, there is a continuous $p \colon  X \rightarrow \mathbb{T}^{n}$ such that $p_{\#} \colon  \pi_{1} (X) \rightarrow \pi_{1} (\mathbb{T}^{n})$ equals $\psi$ and the homotopy fiber is $\overline{X}$.

(3). By the (2) in Theorem \ref{thm_homotopy_finite}, we see $\overline{X}$ is finitely dominated. Since $X$ is a finitely dominated Poincar\'{e} space, by \cite{Gottlieb} (see also \cite[Theorem~G]{Klein_Qin_Su}), $\overline{X}$ is also a Poincar\' e space, whose formal dimension is $d-n$.

Since $q_{\#} h_{\#}$ is injective and $G = \bigcap_{k=1}^{\infty} \mathrm{im} h_{\#}^{k}$, we know $q_{\#} h_{\#} G = G$. (Note we abuse the notation $h_{\#}^{k}$ to denote $(qh)_{\#}^{k}$ for brevity.) Then $q_{\#} h_{\#}$ induces a monomorphism
\[
A \colon \ \ \mathbb{Z}^{n} = \pi_{1} (\mathbb{T}^{n}) = \pi_{1} (X)/ G \longrightarrow \pi_{1} (\mathbb{T}^{n}) = \pi_{1} (X)/ G
\]
such that $p_{\#}^{-1} (\mathrm{im} A) =  \mathrm{im} (qh)_{\#}$. Particularly, $\det (A) \ne 0$. We also have $\bigcap_{k=1}^{\infty} \mathrm{im} A^{k} =0$ because $G = \bigcap_{k=1}^{\infty} \mathrm{im} h_{\#}^{k}$. Let $A_{\mathbb{T}}$ be the linear endomorphism of $\mathbb{T}^{n}$ induced by $A$. By $\mathrm{im} A = (pq)_{\#} \pi_{1} (X')$, we obtain the following commutative diagram.
\begin{equation}\label{thm_Poincare_2}
\xymatrix{
  X' \ar[d]_{p'} \ar[r]^{q} & X \ar[d]^{p} \\
  \mathbb{T}^{n} \ar[r]^{A_{\mathbb{T}}} & \mathbb{T}^{n}   }
\end{equation}
Here we get $p'$ by identifying $\mathbb{T}^{n}$ with its cover $(\mathbb{T}^{n})'$ such that $\pi_{1} ((\mathbb{T}^{n})') = \mathrm{im} A$. Since $p'_{\#}$ is surjective and $\ker p'_{\#} = \ker p_{\#} = G$, the homotopy fiber of $p'$ is also $\overline{X}$ as that of $p$. Note that $q$ is lifted to $\mathrm{id} \colon  \overline{X} \rightarrow \overline{X}$. By \cite[Proposition~7.6.1]{Selick}, we infer \eqref{thm_Poincare_2} is a homotopy pullback.

Furthermore, we have the following diagram.
\begin{equation}\label{thm_Poincare_3}
\xymatrix{
  X \ar[d]_{p} \ar[r]^{h} & X' \ar[d]^{p'} \\
  \mathbb{T}^{n} \ar[r]^{=} & \mathbb{T}^{n}   }
\end{equation}
Since $p' \circ h$ and $p$ induce the same homomorphism on fundamental groups, and since $\mathbb{T}^{n}$ is aspherical, we see \eqref{thm_Poincare_3} commutes up to homotopy.

Combining \eqref{thm_Poincare_2} and \eqref{thm_Poincare_3} together, we infer \eqref{thm_Poincare_1} commutes up to homotopy. Since $h$ is a homotopy equivalence, and since \eqref{thm_Poincare_2} is a homotopy pullback, so is \eqref{thm_Poincare_1}.
\end{proof}

By the classification of low dimensional Poincar\'{e} spaces, we obtain a few corollaries.
\begin{corollary}
Let $d$ and $n$ be the ones in Theorem \ref{thm_Poincare}. Then $d-n \ge 0$.
\end{corollary}

\begin{corollary}\label{cor_Poincare-0}
If $d-n=0$, then $\overline{X}$ is contractible and $X \simeq \mathbb{T}^{d}$.
\end{corollary}

\begin{corollary}\label{cor_Poincare-1}
If $d-n=1$, then $\overline{X} \simeq S^{1}$ and $X \simeq \mathbb{T}^{d}$.
\end{corollary}

\begin{proof}[Proof of Corollaries \ref{cor_Poincare-0} and \ref{cor_Poincare-1}]
By Theorem \ref{thm_Poincare}, $\overline{X}$ is a finitely dominated Poincar\'{e} space of formal dimension $0$ (resp.~$1$). Hence, by \cite[Theorem~4.2]{Wall67}, $\overline{X}$ is contractible (resp.~$\overline{X} \simeq S^{1}$). In either case, $\pi_{1} (X) \cong \mathbb{Z}^{d}$ and $\widetilde{X}$ is contractible, thus $X \simeq \mathbb{T}^{d}$.
\end{proof}

\begin{corollary}\label{cor_Poincare-2}
If $d-n=2$, then $(X, \overline{X})$ satisfies one of the followings:
\begin{enumerate}
\item $\overline{X} \simeq S^{2}$ and $\pi_{1} (X) \cong \mathbb{Z}^{n}$;

\item $\overline{X} \simeq \mathbb{RP}^{2}$ and $\pi_{1} (X) \cong \mathbb{Z}^{n} \oplus \mathbb{Z}/2$;

\item $\overline{X} \simeq \mathbb{T}^{2}$ and $X \simeq \mathbb{T}^{d}$.
\end{enumerate}
\end{corollary}
\begin{proof}
Now $\overline{X}$ is a finitely dominated Poincar\'{e} space of formal dimension $2$. By \cite[Theorem~4]{Eckmann_Linnell} (see also \cite{Eckmann_Muller}), $\overline{X}$ is homotopy equivalent to a closed surface. The conclusion follows.
\end{proof}

\begin{corollary}\label{cor_Poincare-3}
If $d-n=3$, then $(X, \overline{X})$ satisfies one of the followings:
\begin{enumerate}
\item $\overline{X} \simeq S^{3}$, and $\pi_{1} (X) \cong \mathbb{Z}^{n}$;

\item $\overline{X}$ is homotopy equivalent to a $3$-dimensional lens space, and $\pi_{1} (X) \cong \mathbb{Z}^{n} \oplus \pi_{1} (\overline{X})$;

\item $\overline{X} \simeq S^{1} \times S^{2}$ or $\overline{X} \simeq S^{1} \tilde{\times} S^{2}$, and $\pi_{1} (X) \cong \mathbb{Z}^{n+1}$, where $S^{1} \tilde{\times} S^{2}$ is the $S^{2}$-bundle over $S^{1}$ with monodromy $(x_{1}, x_{2}, x_{3}) \mapsto (x_{1}, x_{2}, -x_{3})$;

\item $\overline{X} \simeq S^{1} \times \mathbb{RP}^{2}$, and $\pi_{1} (X) \cong \mathbb{Z}^{n+1} \oplus \mathbb{Z}/2$;

\item $\overline{X} \simeq \mathbb{T}^{3}$ and $X \simeq \mathbb{T}^{d}$.
\end{enumerate}
\end{corollary}
\begin{proof}
Now $\overline{X}$ is a finitely dominated Poincar\'{e} space of formal dimension $3$. Let $e$ denote the number of ends of $\pi_{1} (\overline{X})$ (see \cite[p.~231]{Wall67}). Since $\pi_{1} (\overline{X})$ is finitely generated and abelian, we have $e=0$, $1$ or $2$ (see e.g.~\cite[$\S$8.2]{Loh}).

When $e=0$, by \cite[Theorem~4.3~and~p.~235]{Wall67}, we have $\widetilde{X} \simeq S^{3}$, $\overline{X}$ is orientable, $\pi_{1} (\overline{X})$ is finite, and the cohomology of $\pi_{1} (\overline{X})$ has period $4$. Again, since $\pi_{1} (\overline{X})$ is abelian, by the periodicity and \cite[p.~262,~Theorem~11.6]{Cartan_Eilenberg}, we have $\pi_{1} (\overline{X})$ is cyclic. Clearly, $\pi_{1} (X) \cong \mathbb{Z}^{n} \oplus \pi_{1} (\overline{X})$. If $\pi_{1} (\overline{X}) = 0$, we obtain the (1) of the statement. Otherwise, by \cite[p.~94]{Thomas}, we know $\overline{X}$ is homotopy equivalent to a $3$-dimensional lens space, which yields the (2).

When $e=1$, by \cite[p.~235]{Wall67}, we have $\widetilde{X}$ is contractible. Hence $\overline{X}$ is the Eilenberg-MacLane space $K(\pi_{1} (\overline{X}), 1)$. Since $\pi_{1} (\overline{X})$ is abelian, the (5) of the statement follows.

For the rest case of $e=2$, by \cite[Theorem~4.4]{Wall67} and the commutativity of $\pi_{1} (\overline{X})$, we see $\overline{X} \simeq S^{1} \times S^{2}$, $\overline{X} \simeq S^{1} \tilde{\times} S^{2}$, or $\overline{X} \simeq S^{1} \times \mathbb{RP}^{2}$, which implies the (3) and (4) of the statement.
\end{proof}

\begin{corollary}\label{cor_Poincare-4}
If $d-n=4$ and $\overline{X}$ is simply connected, then $\overline{X}$ is homotopy equivalent to a topological closed $4$-manifold.
\end{corollary}
\begin{proof}
Now $\overline{X}$ is a $1$-connected Poincar\'{e} space of formal dimension $4$. It's a well-known fact that every such space is homotopy equivalent to a closed TOP $4$-manifold. This follows from \cite[p.~161,~Theorem]{Freedman_Quinn} and \cite[p.~103,~Theorem.~1.5]{MH}. (Note that \cite{MH} proved the result for manifolds. Its argument can be extended for Poincar\'e spaces thanks to \cite[Theorem~2.4]{Wall67}.)
\end{proof}

\begin{remark}
There is no speciality of $\pi_{1} (\overline{X})$ when $d-n=4$. Actually, every finitely presented group can be realized as the fundamental group of a closed DIFF $4$-manifold. Furthermore, if $\overline{X}$ is not simply connected, then it is unnecessarily homotopy equivalent to a topological closed $4$-manifold, see \cite[Corollaries~5.4.1~\&~5.4.2]{Wall67}.
\end{remark}

Now we can prove Corollary \ref{cor_torus}.
\begin{proof}[Proof of Corollary \ref{cor_torus}]
By the above corollaries, $M$ is homotopy equivalent to $\mathbb{T}^{m}$. It had been known that every closed TOP manifold homotopy equivalent to $\mathbb{T}^{m}$ is in fact homeomorphic to $\mathbb{T}^{m}$. For $m \le 2$, this is trivial. By the solution to Poincar\'{e} conjecture and \cite[Main~Theorem]{Luft_Sjerve}, this is true for $m=3$. For $m=4$, see e.g.~\cite[\S~11.5]{Freedman_Quinn}. For $m \ge 5$, see \cite[Theorem~A]{Farrell_Hsiang}.
\end{proof}

\section{Low Dimensional Manifolds}\label{sec_low_dim}
In this section, we shall prove Theorem \ref{thm_4-manifold}. First of all, we list some canonical closed self-covering $4$-manifolds with abelian fundamental groups.

The following lemma is obvious.
\begin{lemma}\label{lem_4-manifold}
The followings are closed smooth $4$-manifolds diffeomorphic to certain nontrivial covers of themselves. They are $S^{1} \times S^{3}$, $S^{1} \tilde{\times} S^{3}$, $S^{1} \times \mathbb{R}P^{3}$, $S^{1} \tilde{\times} \mathbb{R}P^{3}$, $S^{1} \times L$, and $\mathbb{T}^{4}$, where $S^{1} \tilde{\times} S^{3}$ (resp.~$S^{1} \tilde{\times} \mathbb{R}P^{3}$) is the $S^{3}$-bundle (resp.~$\mathbb{R}P^{3}$-bundle) over $S^{1
}$ with monodromy $(x_{1}, x_{2}, x_{3}, x_{4}) \mapsto (x_{1}, x_{2}, x_{3}, -x_{4})$ (resp.~$[x_{1}, x_{2}, x_{3}, x_{4}] \mapsto [x_{1}, x_{2}, x_{3}, -x_{4}]$), and $L$ is a $3$-dimensional lens space.
\end{lemma}

\begin{lemma}\label{lem_S1_lens}
Suppose $L_{1}$ and $L_{2}$ are two $3$-dimensional lens spaces. Then the following are equivalent:
\begin{enumerate}
\item $S^{1} \times L_{1}$ is simple homotopy equivalent to $S^{1} \times L_{2}$;

\item $S^{1} \times L_{1}$ is homotopy equivalent to $S^{1} \times L_{2}$;

\item $L_{1}$ is homotopy equivalent to $L_{2}$.
\end{enumerate}
\end{lemma}
\begin{proof}
(1)$\Rightarrow$(2). Trivial.

(2)$\Rightarrow$(3). A homotopy equivalence $h \colon  S^{1} \times L_{1} \rightarrow S^{1} \times L_{2}$ is lifted to be a homotopy equivalence $\bar{h} \colon  \mathbb{R}^{1} \times L_{1} \rightarrow \mathbb{R}^{1} \times L_{2}$ between infinite cyclic covers. The conclusion follows.

(3)$\Rightarrow$(1). Suppose $g \colon  L_{1} \rightarrow L_{2}$ is a homotopy equivalence. By \cite[(23.2)]{Cohen}, we infer $\mathrm{id} \times g \colon  S^{1} \times L_{1} \rightarrow S^{1} \times L_{2}$ is a simple homotopy equivalence.
\end{proof}

\begin{example}\label{exa_S2_T2}
Let $\Phi (t)$ ($t \in \mathbb{R}$) and $\Psi$ be the following elements in $\mathrm{O} (3)$:
\[
\Phi (t) =
\begin{bmatrix}
\cos 2 \pi t & - \sin 2 \pi t & 0 \\
\sin 2 \pi t & \cos 2 \pi t & 0 \\
0 & 0 & 1
\end{bmatrix},
\qquad
\Psi =
\begin{bmatrix}
1 & 0 & 0 \\
0 & 1 & 0 \\
0 & 0 & -1
\end{bmatrix}.
\]
Define a $\mathbb{Z}^{2}$-action on $\mathbb{R}^{2} \times \mathbb{R}^{3}$ as
\begin{align*}
\mathbb{Z} \times \mathbb{Z} \times \mathbb{R} \times \mathbb{R} \times \mathbb{R}^{3} & \rightarrow \mathbb{R} \times \mathbb{R} \times \mathbb{R}^{3} \\
(m,n,a,b,v) & \mapsto (a+m, b+n, \Phi^{m} (b) v)
\end{align*}
or
\begin{align*}
\mathbb{Z} \times \mathbb{Z} \times \mathbb{R} \times \mathbb{R} \times \mathbb{R}^{3} & \rightarrow \mathbb{R} \times \mathbb{R} \times \mathbb{R}^{3} \\
(m,n,a,b,v) & \mapsto (a+m, b+n, \Psi^{m} \Phi^{m} (b) v).
\end{align*}
Let $\xi_{1}$ and $\xi_{2}$ denote the quotient spaces $\mathbb{R}^{2} \times \mathbb{R}^{3} / \sim$ resulted from the first and second actions, respectively. As the projection $\mathbb{R}^{2} \times \mathbb{R}^{3} \rightarrow \mathbb{R}^{2}$ is $\mathbb{Z}^{2}$-equivariant with respect to these actions, $\xi_{1}$ and $\xi_{2}$ are smooth vector bundles over $\mathbb{T}^{2}$ with fiber $\mathbb{R}^{3}$.

Let $E_{i}$ (resp.~$PE_{i}$) denote the associated sphere bundle (resp.~projective bundle) of $\xi_{i}$, i.e.~the fiber of $E_{i}$ (resp.~$PE_{i}$) is the unit sphere (resp.~projective plane) of the corresponding fiber of $\xi_{i}$.

Note that $PE_{1} = PE_{2}$. In fact, the curves $\Phi (t)$ and $\Psi \Phi (t)$, $t \in [0,1]$, represent two loops in $\mathrm{PO} (3) : = \mathrm{O}(3) /\{ \pm 1 \}$. Since $\mathrm{PO} (3) = \mathrm{PSO} (3)$ is connected, there is a path connecting $\Psi$ and $\mathrm{id}$ in $\mathrm{PO} (3)$ which induces a free homotopy between these two loops.

Furthermore, we claim $\pi_{1} (PE_{1}) = \mathbb{Z}^{2} \oplus \mathbb{Z} /2$. Consider the projection of $PE_{1}$ on the first factor $S^{1}$ of $\mathbb{T}^{2}$, then $PE_{1}$ is a $S^{1} \times \mathbb{R}P^{2}$-bundle over $S^{1}$. The monodromy action on $\pi_{1} (S^{1} \times \mathbb{R}P^{2})$ is trivial. The conclusion follows.
\end{example}

\begin{proposition}\label{prop_Z2_4-manifold}
The $E_{1}$, $E_{2}$, and $PE_{1}$ in Example \ref{exa_S2_T2}, $\mathbb{T}^{2} \times S^{2}$, $S^{1} \times (S^{1} \tilde{\times} S^{2})$ and $\mathbb{T}^{2} \times \mathbb{R} P^{2}$ have the following properties:
\begin{enumerate}
\item $\mathbb{T}^{2} \times \mathbb{R} P^{2}$ and $PE_{1}$ are smooth $\mathbb{R}P^{2}$-bundles over $\mathbb{T}^{2}$, others are smooth $S^{2}$-bundles over $\mathbb{T}^{2}$.

\item $\mathbb{T}^{2} \times S^{2}$ is a $2$-cover of $S^{1} \times (S^{1} \tilde{\times} S^{2})$, $\mathbb{T}^{2} \times \mathbb{R} P^{2}$, $E_{1}$ and $E_{2}$; moreover, $E_{1}$ is a $2$-cover of $PE_{1}$.

\item They are diffeomorphic to certain nontrivial covers of themselves.

\item They are of distinct homotopy types.
\end{enumerate}
\end{proposition}
\begin{proof}
(1). This is obvious.

(2). It suffices to prove $\mathbb{T}^{2} \times S^{2}$ is a $2$-cover of $E_{1}$ and $E_{2}$. Let $f \colon  \mathbb{T}^{2} \rightarrow \mathbb{T}^{2}$ be $f([a],[b]) = ([2a],[b])$. Since the loop $\Phi (t)$, $t \in [0,1]$, in Example \ref{exa_S2_T2} represents an element in $\pi_{1} (\mathrm{SO}(3)) = \mathbb{Z}/2$, we know $f^{*} \xi_{1}$ is smoothly isomorphic to a trivial bundle. Thus $\mathbb{T}^{2} \times S^{2}$ is a $2$-cover of $E_{1}$. Similarly, $\Psi^{2} \Phi^{2} (t) = \Phi^{2} (t)$, it is also a $2$-cover of $E_{2}$.

(3). It suffices to verify the cases of $E_{1}$, $E_{2}$, and $PE_{1}$. Let $f \colon  \mathbb{T}^{2} \rightarrow \mathbb{T}^{2}$ be $f([a],[b]) = ([3a],[b])$. Similar to (2), we know $f^{*} \xi_{i}$ is isomorphic to $\xi_{i}$ as smooth vector bundles. Thus $E_{i}$ and $PE_{i}$ are diffeomorphic to their certain $3$-covers.

(4). Clearly, $\mathbb{T}^{2} \times \mathbb{R} P^{2}$ and $PE_{1}$ have fundamental groups $\mathbb{Z}^{2} \oplus \mathbb{Z}/2$ (see Example \ref{exa_S2_T2}), others have fundamental groups $\mathbb{Z}^{2}$; additionally, $\mathbb{T}^{2} \times S^{2}$ and $E_{1}$ are orientable, others are nonorientable. We separate these manifolds into three families. The first (resp.~second) one consists of $\mathbb{T}^{2} \times S^{2}$ and $E_{1}$ (resp.~$S^{1} \times (S^{1} \tilde{\times} S^{2})$ and $E_{2}$), the third one contains the remained. One manifold in a family is not homotopy equivalent to one in another family.

Firstly, we prove $\mathbb{T}^{2} \times S^{2}$ is not homotopy equivalent to $E_{1}$ by comparing their Stiefel-Whitney classes $w_{2}$. Since $\mathbb{T}^{2} \times S^{2}$ is parallelizable, $w_{2} (\mathbb{T}^{2} \times S^{2}) = 0$. On the other hand, the tangent bundle of $E_{1}$ has the decomposition $TE_{1} \cong p^{*} T \mathbb{T}^{2} \oplus \eta$, where $p \colon  E_{1} \rightarrow \mathbb{T}^{2}$ is the projection and $\eta$ is the subbundle of $TE_{1}$ tangent to the fiber of $p$. Since $T \mathbb{T}^{2}$ is trivial,
\[
w_{2} (E_{1}) = w_{2} (T E_{1}) = w_{2} (\eta).
\]
Recall that $E_{1}$ is the associated sphere bundle of the vector bundle $\xi_{1}$ in Example \ref{exa_S2_T2}. Clearly, $p^{*} \xi_{1} = \eta \oplus \varepsilon$, where $\varepsilon$ is a trivial line bundle. Then
\[
w_{2} (\eta) = w_{2} (p^{*} \xi_{1}) = p^{*} w_{2} (\xi_{1}),
\]
where the last $p^{*}$ is the homomorphism $H^{2} (\mathbb{T}^{2}; \mathbb{Z}/2) \rightarrow H^{2} (E_{1}; \mathbb{Z}/2)$. Note that the loop $\Phi (t)$, $t \in [0,1]$, in Example \ref{exa_S2_T2} represents the nontrivial element in $\pi_{1} (\mathrm{SO} (3)) = \mathbb{Z} /2$. By the obstruction theoretic interpretation of Stiefel-Whitney classes (\cite[Theorem~12.1]{Milnor_Stasheff}), we infer $w_{2} (\xi_{1}) \ne 0$. Since the fiber of $\xi_{1}$ is $3$-dimensional, the Euler class $e (\xi_{1}) =0$. By Gysin sequence (\cite[Theorem~12.2]{Milnor_Stasheff}), $p^{*} w_{2} (\xi_{1}) \ne 0$ too, which implies $w_{2} (E_{1}) \ne 0$. By Wu's formula (\cite[Theorem~11.14]{Milnor_Stasheff}), Stiefel-Whitney classes are homotopy invariants of closed manifolds. Thus $\mathbb{T}^{2} \times S^{2}$ is not homotopy equivalent to $E_{1}$.

Secondly, we compare $S^{1} \times (S^{1} \tilde{\times} S^{2})$ with $E_{2}$. By a calculation similar to the above, we also have $w_{2} (S^{1} \times (S^{1} \tilde{\times} S^{2})) =0$ and $w_{2} (E_{2}) \ne 0$, which implies that $S^{1} \times (S^{1} \tilde{\times} S^{2})$ is not homotopy equivalent to $E_{2}$.

It remains to compare $\mathbb{T}^{2} \times \mathbb{R} P^{2}$ with $PE_{1}$. Consider their Stiefel-Whitney classes $w_{1}$ as homomorphisms from their fundamental groups to $\mathbb{Z} /2$. Then $\mathbb{T}^{2} \times S^{2}$ (resp. $E_{1}$) is the $2$-cover of $\mathbb{T}^{2} \times \mathbb{R} P^{2}$ (resp. $PE_{1}$) with fundamental group $\ker w_{1}$. We prove $\mathbb{T}^{2} \times \mathbb{R} P^{2}$ and $PE_{1}$ are of different homotopy types by contradiction. Suppose not, there would be a homotopy equivalence $h \colon  \mathbb{T}^{2} \times \mathbb{R} P^{2} \rightarrow PE_{1}$. By Wu's formula again,
\[
w_{1} (\mathbb{T}^{2} \times \mathbb{R} P^{2}) = h^{*} w_{1} (PE_{1}).
\]
Thus $h$ would be lifted to be a homotopy equivalence $\bar{h} \colon  \mathbb{T}^{2} \times S^{2} \rightarrow E_{1}$, which is a contradiction.
\end{proof}

\begin{lemma}\label{lem_self-homotopy}
Suppose $f$ is a self-homotopy equivalence of $S^{1} \times S^{2}$ or $S^{1} \times \mathbb{R} P^{2}$ which induces the identity automorphism on the homology with $\mathbb{Z}$-coefficient. Then $f$ is homotopic to $\mathrm{id}$ or the map $([t], v) \mapsto ([t], \Phi (t) v)$, where $\Phi$ is the one in Example \ref{exa_S2_T2}.
\end{lemma}
\begin{proof}
We first verify the case of $S^{1} \times \mathbb{R} P^{2}$. Consider the bundle $p \colon  S^{1} \times \mathbb{R} P^{2} \rightarrow S^{1}$. Since the action of $f$ on $\pi_{1} (S^{1} \times \mathbb{R} P^{2})$ is the identity, we have $p f \simeq p$. By a homotopy perturbation, we may assume $f$ is a fiber-homotopy equivalence, i.e. it has the form $([t], v) \mapsto ([t],  g([t]) v)$, where $g$ is a loop in $G(\mathbb{R} P^{2})$ which is the space of self-homotopy equivalence of $\mathbb{R} P^{2}$. There is an evident inclusion $\mathrm{SO} (3) \hookrightarrow G(\mathbb{R} P^{2})$. By \cite[Theorem~4.3]{Yamanoshita} and \cite[Remark~3.2]{Goncalves_Spreafico}, this inclusion induces an isomorphism between fundamental groups. As the elements in $\pi_{1} (\mathrm{SO} (3)) = \mathbb{Z}/2$ are represented by a constant loop and the loop $\Phi$, we obtain the conclusion for $S^{1} \times \mathbb{R} P^{2}$.

Duplicating the above argument, we obtain the conclusion for $S^{1} \times S^{2}$. Now we need the fact $\pi_{1} (\mathrm{SO} (3)) = \pi_{1} (G_{1} (S^{2}))$ which is proved in \cite{Hansen}, where $G_{1} (S^{2})$ is the component of $G (S^{2})$ containing the identity.
\end{proof}

\begin{proposition}\label{prop_S2_T2}
Suppose $M$ is a closed $4$-manifold which is homotopy equivalent to a fibration over $\mathbb{T}^{2}$ with fiber $S^{2}$ (resp.~$\mathbb{R}P^{2}$). Then $M$ is homeomorphic (resp.~homotopy equivalent) to $\mathbb{T}^{2} \times S^{2}$, $S^{1} \times (S^{1} \tilde{\times} S^{2})$, $E_{1}$ or $E_{2}$ (resp.~$\mathbb{T}^{2} \times \mathbb{R} P^{2}$ or $PE_{1}$).
\end{proposition}
\begin{proof}
We first consider the case that the fiber is $S^{2}$. Obviously, the Euler characteristic of $M$ is $0$. Since $\pi_{1} (X) = \mathbb{Z}^{2}$, by \cite[Theorem~6.11]{Hillman}, it suffices to determine the homotopy type of $M$.

Let $K$ be the set consisting of elements of $\pi_{1} (\mathbb{T}^{2})$ whose monodromy actions on $H_{2} (S^{2}; \mathbb{Z})$ are trivial. Then $K \le \pi_{1} (\mathbb{T}^{2})$ and $[\pi_{1} (\mathbb{T}^{2}): K] \le 2$.

Suppose $K = \pi_{1} (\mathbb{T}^{2})$. Then $M$ is homotopy equivalent to a fibration over $S^{1}$ with fiber $S^{1} \times S^{2}$. We already know the monodromy action on
\[
H_{2} (S^{1} \times S^{2}; \mathbb{Z}) \cong H_{2} (S^{2}; \mathbb{Z})
\]
is trivial. Since $\pi_{1} (M) = \pi_{1} (\mathbb{T}^{2})$ is abelian, the monodromy action on
\[
H_{1} (S^{1} \times S^{2}; \mathbb{Z}) \cong H_{1} (S^{1}; \mathbb{Z}),
\]
and hence on the whole homology of $S^{1} \times S^{2}$, is trivial too. By Lemma \ref{lem_self-homotopy}, the monodromy is homotopic to $\mathrm{id}$ or the map $([t], v) \mapsto ([t], \Phi (t) v)$, which implies $M$ is homotopy equivalent to $\mathbb{T}^{2} \times S^{2}$ or $E_{1}$.

Otherwise, $K \ne \pi_{1} (\mathbb{T}^{2})$. Changing a basis of $\pi_{1} (\mathbb{T}^{2}) = \mathbb{Z}^{2}$ if necessary, we may assume
\[
K = (2 \mathbb{Z}) \oplus \mathbb{Z} \le \mathbb{Z}^{2} = \pi_{1} (\mathbb{T}^{2}).
\]
Then $M$ is homotopy equivalent to a fibration over $S^{1}$ with fiber $S^{1} \times S^{2}$. The monodromy action on $H_{1} (S^{1} \times S^{2}; \mathbb{Z})$ is trivial, whereas the action on $H_{2} (S^{1} \times S^{2}; \mathbb{Z})$ is $-\mathrm{id}$. Note that the map $([t],v) \mapsto ([t], \Psi v)$ is self-homeomorphism of $S^{1} \times S^{2}$ whose action on homology is the same as that of the monodromy, where $\Psi$ is the one in Example \ref{exa_S2_T2}. By Lemma \ref{lem_self-homotopy}, the monodromy is homotopic to $([t],v) \mapsto ([t], \Psi v)$ or $([t], v) \mapsto ([t], \Psi \Phi (t) v)$, which implies $M$ is homotopy equivalent to $S^{1} \times (S^{1} \tilde{\times} S^{2})$ or $E_{2}$.

The remained case is that the fiber is $\mathbb{R}P^{2}$. Now $M$ is homotopy equivalent to a fibration over $S^{1}$ with fiber $S^{1} \times \mathbb{R}P^{2}$. (Note that every fibration over $S^{1}$ with fiber $\mathbb{R}P^{2}$ is trivial.) The conclusion also follows from Lemma \ref{lem_self-homotopy}.
\end{proof}

We are at a position to prove Theorem \ref{thm_4-manifold}.
\begin{proof}[Proof of Theorem \ref{thm_4-manifold}]
Clearly, the concrete manifolds are closed $\mathrm{DIFF}$ $4$-manifolds which are diffeomorphic to certain nontrivial covers of themselves. The only probably nontrivial part of this claim is proved in the (3) in Proposition \ref{prop_Z2_4-manifold}. By comparing fundamental groups and the (4) in Proposition \ref{prop_Z2_4-manifold}, they are of distinct homotopy types.

By Corollaries \ref{cor_Poincare-0}, \ref{cor_Poincare-1}, \ref{cor_Poincare-2} and \ref{cor_Poincare-3}, the $M$ is homotopy equivalent to a fibration over a torus whose fiber is a closed manifold with abelian fundamental group, and the fiber has dimension less than $4$. Thus $M$ is a fibration over certain $\mathbb{T}^{n}$ whose fiber is a closed $(4-n)$-manifold with finite abelian fundamental group. Thus the only possibilities of $\pi_{1} (M)$ are $\mathbb{Z}$, $\mathbb{Z} \oplus \mathbb{Z} /2$, $\mathbb{Z} \oplus \mathbb{Z} /k$ with $k>2$, $\mathbb{Z}^{2}$, $\mathbb{Z}^{2} \oplus \mathbb{Z} /2$ or $\mathbb{Z}^{4}$. We study these possibilities one by one.

(1). The conclusion follows from \cite[Theorem~E]{Qin_Su_Wang}.

(2). We know $M$ is homotopy equivalent to a fibration over $S^{1}$ with fiber $\mathbb{R}P^{3}$. Since $\pi_{2} (\mathbb{R}P^{3}) =0$, by \cite[Theorem~IIb]{Olum}, the monodromy is homotopic to the $\mathrm{id}$ or the homeomorphism $[x_{1}, x_{2}, x_{3}, x_{4}] \mapsto [x_{1}, x_{2}, x_{3}, -x_{4}]$. Thus $M$ is homotopy equivalent to $S^{1} \times \mathbb{R}P^{3}$ or $S^{1} \tilde{\times} \mathbb{R}P^{3}$. A $2$-cover is homotopy equivalent to $S^{1} \times S^{3}$ or $S^{1} \tilde{\times} S^{3}$ respectively, and hence the conclusion follows from (1).

(3). We know $M$ is homotopy equivalent to a fibration over $S^{1}$ with fiber $L$. Let $f \colon  L \rightarrow L$ denote the monodromy. We shall prove $f \simeq \mathrm{id}$.

Since $\pi_{1} (M)$ is abelian, $f_{\#} \colon  \pi_{1} (L) \rightarrow \pi_{1} (L)$ is the identity. Let's construct an Eilenberg-MacLane space $K(\mathbb{Z}/k, 1)$ by attaching cells with dimension above $3$ to $L$. Let $\iota \colon  L \rightarrow K(\mathbb{Z}/k, 1)$ be the inclusion. Then $\iota f \simeq \iota$. Let $1 \in H_{3} (L; \mathbb{Z}) = \mathbb{Z}$ be the chosen orientation. Then $f_{*} (1) = \pm 1$ and
\[
\iota_{*} (1) = [1] \in H_{3} (K(\mathbb{Z}/k, 1), \mathbb{Z}) = \mathbb{Z}/k.
\]
As $\iota_{*} (1) = \iota_{*} f_{*} (1)$, we obtain $1 \equiv \pm 1 (\mathrm{mod}\ k)$. Since $k>2$, we infer $f_{*} (1) = 1$. Since $\pi_{2} (L) =0$, by \cite[Theorem~IIb]{Olum} again, we get $f \simeq \mathrm{id}$, and hence $M \simeq S^{1} \times L$.

A $k$-cover is homotopy equivalent to $S^{1} \times S^{3}$. The conclusion follows from (1) again.

(4) and (5). We know $M$ is homotopy equivalent to a fibration over $\mathbb{T}^{2}$ with fiber $S^{2}$ or $\mathbb{R}P^{2}$. The conclusion follows from Proposition \ref{prop_S2_T2}.

(6). Now $\pi_{1} (M) = \mathbb{Z}^{4}$. We know $M$ is homotopy equivalent to $\mathbb{T}^{4}$. The conclusion follows from \cite[\S~11.5]{Freedman_Quinn}.
\end{proof}

\section{High Dimensional Manifolds: Fibering}\label{sec_fibering}
In this section, we shall prove fibering Theorems \ref{thm_bundle_s1}, \ref{thm_approximate_fibration}, \ref{thm_block_bundle}, \ref{thm_stable_bundle}, \ref{thm_bundle_T2} and \ref{thm_bundle}.

Let's recall the definition of a $\mathrm{CAT}$ (fiber) bundle at first (see e.g. \cite[p.~10]{Kirby_Siebenmann}).
\begin{definition}
Suppose $E$, $B$ and $F$ are $\mathrm{CAT}$ manifolds, $p \colon  E \rightarrow B$ is a $\mathrm{CAT}$ morphism. Suppose $\forall b \in B$, there exist an open neighborhood $U$ of $b$ and a $\mathrm{CAT}$ isomorphism $\Phi \colon  U \times F \rightarrow p^{-1} (U)$ such that the following diagram commutes
\[
\xymatrix{
  U \times F \ar[rr]^{\Phi} \ar[dr]
                &  &    p^{-1} (U) \ar[dl]^{p}    \\
                & U                 },
\]
where $U \times F \rightarrow U$ is the projection. Then we call $p \colon  E \rightarrow B$ a \textit{$\mathrm{CAT}$ fiber bundle}, or a \textit{$\mathrm{CAT}$ bundle}, with fiber $F$.
\end{definition}

\begin{proof}[Proof of Theorem \ref{thm_bundle_s1}]
(1). This follows from the (1) in Theorem \ref{thm_Poincare}.

(2). Let $\overline{M}$ be the cover of $M$ with $\pi_{1} (\overline{M}) =G$. By the (2) in Theorem \ref{thm_Poincare}, $\overline{M}$ is finitely dominated. By the assumption that $\pi_{1} (M)$ is free abelian, so is $G$. We then know the reduced $K$-group $\widetilde{K}_{0} (\mathbb{Z} [G]) =0$, which implies that the Wall's finiteness obstruction of $\overline{M}$ vanishes. Thus $\overline{M}$ is homotopy equivalent to a finite CW complex. Let $Y$ be the cover of $M$ with $\pi_{1} (Y) = \ker \theta$. We see $Y$ is homotopy equivalent to a fibration over $\mathbb{T}^{n-1}$ with fiber $\overline{M}$. So $Y$ is also homotopy equivalent to a finite CW complex. Clearly, there exists a continuous map $p \colon  M \rightarrow S^{1}$ such that $p_{\#} = \theta$.

Since $\pi_1(M)$ is free abelian, the Cappell Splitting Theorem \cite[Corollary 6]{Cappell} together with the $S$-Cobordism Theorem implies that
$p$ is homotopic to a $\mathrm{CAT}$ bundle projection. (The details are identical to those in the 7th and 8th paragraphs of \cite[p.~1899]{Qin_Su_Wang}.)
\end{proof}

Let's recall the definition of approximate fibrations. Following \cite[Definition~2.2]{FLS2018}, we only consider the case that the base space is compact. (For general case, see \cite[\S~12]{HTW90}.)
\begin{definition}\label{def_approximate_fibration}
Let $E$ be a topological space, $(B,d)$ be a compact metric space, $p \colon  E \rightarrow B$ be a continuous map. Suppose $\forall \epsilon >0$, for any topological space $X$, for any homotopy $\lambda  \colon  X \times [0,1] \rightarrow B$ and any continuous map $f \colon  X \rightarrow E$ with $pf = \lambda_{0}$, there exits a homotopy $\Lambda  \colon  X \times [0,1] \rightarrow E$ such that $\Lambda_{0} = f$ and, $\forall (x,t) \in X \times [0,1]$, $d(p \Lambda (x,t), \lambda (x,t)) < \epsilon$. Then we call $p \colon  E \rightarrow B$ an \textit{approximate fibration}.
\end{definition}
Since $B$ is compact, it's easy to see Definition \ref{def_approximate_fibration} is independent of the choice of the metric $d$ compatible with the topology of $B$.

\begin{proof}[Proof of Theorem \ref{thm_approximate_fibration}]
By Theorem \ref{thm_Poincare}, $\overline{M}$ is finitely dominated. Since $\pi_{1} (\overline{M})$ is free abelian, we further infer $\overline{M}$ is homotopy equivalent to a finite CW complex. By Theorem \ref{thm_Poincare} again, the homotopy fiber of $p$ is homotopy equivalent to a finite Poincar\'{e} complex.

We know $\mathbb{T}^{n}$ satisfies the condition of the $B$ in \cite[Theorem~1.4]{FLS2018}. Since $\pi_{1} (M)$ is free abelian, the Whitehead group $\mathrm{Wh} (\pi_{1} (M))$, and hence all quotient groups of $\mathrm{Wh} (\pi_{1} (M))$, vanish. Thus the $N \tau (p)$ in \cite[Conjecture~1.3]{FLS2018} vanishes. Additionally, $\dim M \ne 4$ by the assumption, the conclusion follows from \cite[Theorem~1.4]{FLS2018}.
\end{proof}

We recall the definition of block bundles in the special case that the total spaces are $\mathrm{CAT}$ manifolds. (See e.g. \cite[p.~473]{Rourke_Sanderson71} and \cite[Definition~2.4]{Ebert_Randal} for a general definition.)
\begin{definition}\label{def_block_bundle}
Suppose $E$ and $F$ are $\mathrm{CAT}$ manifolds, $K$ is a simplicial complex, $p \colon  E \rightarrow |K|$ is a continuous map. Suppose for any simplex $\sigma \in K$ with faces $\partial_{0} \sigma, \dots, \partial_{n} \sigma$, we have $(p^{-1} (\sigma); p^{-1} (\partial_{0} \sigma), \dots, p^{-1} (\partial_{n} \sigma))$ is an $(n+2)$-ad of $\mathrm{CAT}$ submanifolds of $E$. Furthermore, there is a $\mathrm{CAT}$ isomorphism of $(n+2)$-ads
\[
\Phi \colon  \ \ (\sigma \times F; \partial_{0} \sigma \times F, \dots, \partial_{n} \sigma \times F) \rightarrow (p^{-1} (\sigma); p^{-1} (\partial_{0} \sigma), \dots, p^{-1} (\partial_{n} \sigma)).
\]
Then we call $p \colon  E \rightarrow |K|$ a \textit{$\mathrm{CAT}$ block bundle} with fiber $F$.
\end{definition}

Obviously, a $\mathrm{CAT}$ bundle is a $\mathrm{CAT}$ block bundle, but not vice versa. (Here we assume the base of this $\mathrm{CAT}$ bundle carries a triangulation compatible with its $\mathrm{CAT}$ structure.) One remarkable difference between a bundle and a block bundle is that, for a block bundle, the following diagram needs \textit{not} to commute.
\[
\xymatrix{
  \sigma \times F \ar[rr]^{\Phi} \ar[dr]
                &  &    p^{-1} (\sigma) \ar[dl]^{p}    \\
                & \sigma                 }
\]

\begin{proof}[Proof of Theorem \ref{thm_block_bundle}]
By Theorem \ref{thm_approximate_fibration}, we can arrange $p \colon  M \rightarrow \mathbb{T}^{n}$ as an approximate fibration with homotopy fiber $\overline{M}$. Since $\pi_{1} (\overline{M})$ is free abelian, we infer $\overline{M}$ is homotopy equivalent to a finite CW complex and $\mathrm{Wh} (\pi_{1} (\overline{M}) \times \mathbb{Z}^{k}) =0$ for all $k \ge 0$. Since $m-n \ge 5$, by \cite[Theorem~3.3.2]{Quinn79}, choosing a suitable triangulation of $\mathbb{T}^{n}$ compatible with its standard $\mathrm{PL}$ structure, we can modify $p$ further to be a block bundle projection via a homotopy perturbation.
\end{proof}

\begin{remark}
The definition of approximate fibration in \cite[\S~3.3]{Quinn79} is slightly different from \cite[Definition~2.2]{FLS2018} and hence our Definition \ref{def_approximate_fibration}. However, by \cite[Theorem~12.13]{HTW90}, these definitions are equivalent when $E$ and $B$ are compact $\mathrm{TOP}$ manifolds.
\end{remark}

Before proving Theorems \ref{thm_stable_bundle}, \ref{thm_bundle_T2} and \ref{thm_bundle}, we need some general results promoting $\mathrm{TOP}$ block bundles to $\mathrm{TOP}$ bundles. In the following Lemmas \ref{lem_stable}, \ref{lem_base_2_dim} and \ref{lem_fiber_high_connected}, let $p \colon  E \rightarrow B$ be a $\mathrm{TOP}$ block bundle with fiber a closed $\mathrm{TOP}$ manifold $F$, let $B$ be a $\mathrm{PL}$ manifold with a fixed triangulation, $\dim B =b$, and $\dim F = f$.

Following the convention of \cite{BLR1975}, let $A(F)$ denote the $\mathrm{CAT}$ automorphism group of $F$, let $\tilde{A}(F)$ denote the semisimplicial group of $\mathrm{CAT}$ automorphisms of $F$. More precisely, an $n$-simplex of $\tilde{A}(F)$ is a $\mathrm{CAT}$ automorphism of $\Delta^{n} \times F$, where $\Delta^{n}$ is the standard geometric $n$-simplex. Furthermore, $\Delta^{n} \times F$ has a face structure induced by that of $\Delta^{n}$, and the automorphisms of $\Delta^{n} \times F$ preserve faces. The semisipicial structure of $\tilde{A}(F)$ follows from that of $\{ \Delta^{n} \mid n \geq 0 \}$ in an evident way. We can also consider $A(F)$ as a subgroup of $\tilde{A}(F)$, i.e. $A(F)_{n} = \{ g \in \tilde{A}(F)_{n} \mid g= \mathrm{id} \times a, a \in A(F) \}$.

\begin{lemma}\label{lem_stable}
Suppose $B$ is compact. Then the composition map $E \times \mathbb{T}^{s} \rightarrow E \overset{p}{\rightarrow} B$ is homotopic to a $\mathrm{TOP}$ bundle projection, where $s= \tfrac{1}{2} b(b-1)$.
\end{lemma}
\begin{proof}
This follows immediately from Corollary 2 in \cite[p.~42]{BLR1975}. That corollary essentially relies on Lemmas 3.12 and 3.14 in \cite{BLR1975} which hold for both $\mathrm{PL}$ and $\mathrm{TOP}$.
\end{proof}

\begin{lemma}\label{lem_base_2_dim}
Suppose $b=2$, $F$ is $1$-connected and $f \ge 4$. Then $p$ is homotopic to a $\mathrm{TOP}$ bundle projection.
\end{lemma}
\begin{proof}
By the extension of Cerf's Theorem to $\mathrm{TOP}$, we have $C(F)$ is connected, where $C(F)$ is the concordance group of $F$. (When $f \ge 5$, see \cite[p.~116]{Kirby_Siebenmann} or \cite[Theorem~3.1]{Hatcher}. When $f = 4$, see \cite[Corollaire~1]{Perron} or \cite[Theorem~1.4]{Quinn86}.)

We have the Hatcher's spectral sequence \cite[Proposition~2.1]{Hatcher}
\[
E_{ij}^{1} = \pi_{j} C(F \times [0,1]^{i}) \Longrightarrow \pi_{i+j+1} (\tilde{A} (F) / A(F)).
\]
Thus $\pi_{1} (\tilde{A} (F) / A(F)) =0$. Clearly, $E$ can always be reduced to a bundle on the $1$-skeleton of $B$. Since $b=2$, the conclusion follows.
\end{proof}

\begin{lemma}\label{lem_fiber_high_connected}
Suppose $F$ is $r$-connected, $f \ge r+4$, and $\min \{ 2r-1, r+4 \} \ge b$. (Hence $f \ge 5$.) Then $p$ is homotopic to a $\mathrm{TOP}$ bundle projection.
\end{lemma}
\begin{proof}
It suffices to show $\tilde{A} (F) / A(F)$ is $(b-1)$-connected.

Note that this has been essentially proved in \cite{BLR1975} for $\mathrm{PL}$. Let's recall some details. By \cite[Corollary~3.2]{BLR1975},
\begin{equation}\label{lem_fiber_high_connected_1}
\forall j \le b-1, \qquad \pi_{j} (\tilde{A} (F) / A(F), \tilde{A} (D^{f}) / A(D^{f})) =0.
\end{equation}
Here $D^{f}$ is an $f$-dimensional closed disk, $\tilde{A} (D^{f})$ (resp.~$A(D^{f})$) is defined similarly as $\tilde{A} (F)$ (resp.~$A(F)$) except that its $n$-simplices fix $\Delta^{n} \times \partial D^{f}$. By Remark c in \cite[p.~30]{BLR1975}, $\tilde{A} (D^{f}) / A(D^{f})$ is contractible. Thus $\tilde{A} (F) / A(F)$ is $(b-1)$-connected.

A duplication of the above argument provides a proof for $\mathrm{TOP}$ because all facts needed for $\mathrm{PL}$ are also true for $\mathrm{TOP}$. For the convenience of the reader, we mention the key points with references. First of all, the Morlet's Lemma of disjunction holds for $\mathrm{TOP}$ (see \cite[p.~142]{BLR1975}). Second, $\pi_{j} (\widetilde{\mathrm{TOP}}_{n}, \mathrm{TOP}_{n}) =0$ for $j \le n+2$ (see \cite[Proposition~5.6]{Burghelea_Lashof}), where $\mathrm{TOP}_{n}$ (resp.~$\widetilde{\mathrm{TOP}}_{n}$) is the automorphism group (resp.~semisimplicial automorphism group) of $\mathbb{R}^{n}$ in $\mathrm{TOP}$. Thus, following the proof of \cite[Corollary~3.2]{BLR1975}, we see \eqref{lem_fiber_high_connected_1} also holds for $\mathrm{TOP}$. Finally, $\tilde{A} (D^{f}) / A(D^{f})$ is contractible in $\mathrm{TOP}$ by the Alexander's trick (see e.g.~the proof of \cite[Lemma~5.1]{Browder66}).
\end{proof}

\begin{proof}[Proof of Theorem \ref{thm_stable_bundle}]
This follows from Theorem \ref{thm_block_bundle} and Lemma \ref{lem_stable}.
\end{proof}

\begin{proof}[Proof of Theorem \ref{thm_bundle_T2}]
(1).  For $m \ge 5$ and $m=4$, this had been proved in Theorems B and E in \cite{Qin_Su_Wang}, respectively. For $m<4$, this is well-known, see e.g.~the third paragraph in \cite[p.~1887]{Qin_Su_Wang}.

(2). By Theorem \ref{thm_block_bundle}, we can arrange $p \colon  M \rightarrow \mathbb{T}^{n}$ as a block bundle projection with $1$-connected fiber $F$, and $\dim F \ge 5$. By Lemma \ref{lem_base_2_dim}, we can modify $p$ further to be a bundle projection via a homotopy pertubation.
\end{proof}

\begin{proof}[Proof of Theorem \ref{thm_bundle}]
By Theorem \ref{thm_block_bundle}, we can arrange $p \colon  M \rightarrow \mathbb{T}^{n}$ as a block bundle projection with $1$-connected fiber $F$. Since $F$ is homotopy equivalent to the universal cover of $M$, and $\pi_{i} (M) =0$ for $2 \le i \le r$, we see $F$ is $r$-connected. Now the conclusion follows from Lemma \ref{lem_fiber_high_connected}.
\end{proof}

\section{High Dimensional Manifolds: Non-Fibering}\label{sec_nonfibering}
In this section, we shall prove the non-fibering Theorem \ref{thm_nonfibering}.

\begin{lemma}\label{lem_Tn_S2}
Suppose $M$ is a closed $\mathrm{TOP}$ manifold homotopy equivalent to $\mathbb{T}^{n} \times S^{2}$ or $\mathbb{T}^{n} \times S^{3}$. If $M$ is a $\mathrm{TOP}$ bundle over $\mathbb{T}^{n}$, then $M$ is smoothable and the rational Pontryagin class $p_{2} (M) =0$.
\end{lemma}
\begin{proof}
By the assumption and the solution to Poincar\'{e} conjecture, $M$ is a $\mathrm{TOP}$ $S^{i}$-bundle over $\mathbb{T}^{n}$, where $i=2$ or $3$.

We know the inclusion $\mathrm{O} (i+1) \hookrightarrow \mathrm{Homeo} (S^{i})$ is a homotopy equivalence, where $\mathrm{Homeo} (S^{i})$ is the group of self-homeomorphims of $S^{i}$. For $i=2$, this fact had been proved by Kneser \cite{Kneser} (see also \cite{Friberg}). For $i=3$, this is a consequence of the combination of Cerf's \cite[3.2.1]{Cerf61} and Hatcher's \cite{Hatcher83}.

Therefore, $M$ is homeomorphic to the sphere bundle associated to a smooth vector bundle over $\mathbb{T}^{n}$ with fiber $\mathbb{R}^{i+1}$, let alone it is smoothable.

Since rational Pontryagin classes are topological invariants (see \cite{Novikov} and the Theorem 5.3 in \cite[p.~317]{Kirby_Siebenmann}), to analyze $p_{2} (M)$, we may assume $M$ itself is a smooth bundle $p \colon  M \rightarrow \mathbb{T}^{n}$ with fiber $S^{i}$. The tangent bundle $TM$ of $M$ has a decomposition
\[
TM \cong \zeta \oplus p^{*} T \mathbb{T}^{n},
\]
where $\zeta$ is the vertical vector bundle over $M$ whose fiber is tangent to the fiber of $p$. Since $T \mathbb{T}^{n}$ is trivial, $p_{2} (M) = p_{2} (TM) = p_{2} (\zeta)$. Since the fiber of $\zeta$ is $\mathbb{R}^{i}$ with $i \le 3$, we infer $p_{2} (\zeta) =0$ and hence $p_{2} (M) = 0$.
\end{proof}

\begin{proof}[Proof of Theorem \ref{thm_nonfibering}]
Let's firstly assume $n \ge 4$. The Spivak normal fibration of $\mathbb{T}^{n}$ is trivial. Consider $\mathrm{TOP}$ surgery problems with target $\mathbb{T}^{n}$. Each element of $[\mathbb{T}^{n}, \mathrm{G}/ \mathrm{TOP}]$ is represented by a surgery problem
\[
(f,b) \colon  \ \ (N, \nu N) \rightarrow (\mathbb{T}^{n}, \xi),
\]
where $N$ is a closed $\mathrm{TOP}$ manifold with stable normal bundle $\nu N$, $\xi$ is a $\mathrm{TOP}$ Euclidean space bundle over $\mathbb{T}^{n}$ whose image in $B \mathrm{G}$ is trivial, $f \colon  N \rightarrow \mathbb{T}^{n}$ is a degree $1$ map, and $b \colon  \nu N \rightarrow \xi$ is a bundle map covering $f$ (see e.g.~\cite[Proposition~9.43]{Ranicki}). By the Theorem 15.1 in \cite[p.~328]{Kirby_Siebenmann}, there is an epimorphism
\[
\Delta_{*} \colon  \ \ [\mathbb{T}^{n}, \mathrm{G}/ \mathrm{TOP}] \rightarrow H^{4} (\mathbb{T}^{n}; \mathbb{Z}/2)
\]
such that $\Delta_{*}$ maps the element represented by the above $(f,b)$ to the Kirby-Siebenmann invariant $\mathrm{ks} (\xi)$ of $\xi$. We have $H^{4} (\mathbb{T}^{n}; \mathbb{Z}/2) \ne 0$ since $n \ge 4$. Thus we may choose a surgery problem $(f,b)$ with $\mathrm{ks} (\xi) \ne 0$.

By the product formula for surgery obstruction (see e.g.~\cite[Theorem IV.1.1]{Morgan}), the surgery obstruction of
\[
(f \times 1, b \times1) \colon  (N \times S^{i}, \nu (N \times S^{i})) \rightarrow (\mathbb{T}^{n} \times S^{i}, \xi \times \nu S^{i})
\]
vanishes, where $i=2$ or $3$. Thus we obtain a homotopy equivalence $\Phi \colon  M \rightarrow \mathbb{T}^{n} \times S^{i}$ such that $M$ is a closed $\mathrm{TOP}$ manifold with $\nu M \cong \Phi^{*} (\xi \times \nu S^{i})$. Since $\nu S^{i}$ is trivial and $\Phi$ induces an isomorphism $\Phi^{*}$ of cohomologies, we have
\[
\mathrm{ks} (\nu M) = \Phi^{*} (\mathrm{ks} (\xi \times \nu S^{i})) = \Phi^{*} (\mathrm{ks} (\xi) \times 1) \ne 0.
\]
So this $M$ is not smoothable. By Lemma \ref{lem_Tn_S2}, it is not a $\mathrm{TOP}$ bundle over $\mathbb{T}^{n}$. We obtain the desired $M$ for $n \ge 4$.

Now let's assume $n \ge 8$. The desired $M$ needs to be a $\mathrm{DIFF}$ manifold. The argument almost duplicates the above. We consider $\mathrm{DIFF}$ surgery problems and analyze rational Pontryagin classes. The Pontryagin character provides a well-known isomorphism (which follows immediately from the property of the Chern character, cf.~\cite[p.~19]{Atiyah_Hirzebruch})
\[
[\mathbb{T}^{n}, B\mathrm{O}] \otimes \mathbb{Q} \overset{\cong}{\longrightarrow} \bigoplus_{j=1}^{\infty} H^{4j} (\mathbb{T}^{n}; \mathbb{Q}).
\]
Since $n \ge 8$, there is a vector bundle $\eta$ over $\mathbb{T}^{n}$ with rational Pontryagin classes $p_{1} (\eta) =0$ and $p_{2} (\eta) \ne 0$. We also have the isomorphism (see e.g.~\cite[\S~9.2]{Ranicki})
\[
[\mathbb{T}^{n}, \mathrm{G} /\mathrm{O}] \otimes \mathbb{Q} \overset{\cong}{\longrightarrow} [\mathbb{T}^{n}, B\mathrm{O}] \otimes \mathbb{Q}
\]
which is induced by the inclusion $\mathrm{G} /\mathrm{O} \rightarrow B\mathrm{O}$. Each element of $[\mathbb{T}^{n}, \mathrm{G} /\mathrm{O}]$ is represented by a $\mathrm{DIFF}$ surgery problem. So we can get a $\mathrm{DIFF}$ surgery problem with target $(\mathbb{T}^{n}, \xi)$ such that $\xi$ is a vector bundle and $\xi = k \eta$ in $[\mathbb{T}^{n}, B\mathrm{O}]$ for some integer $k>0$. Thus $p_{2} (\xi) \ne 0$. Similar to the above, we obtain a homotopy equivalence $\Phi \colon  M \rightarrow \mathbb{T}^{n} \times S^{i}$ such that $M$ is a closed $\mathrm{DIFF}$ manifold with $\nu M \cong \Phi^{*} (\xi \times \nu S^{i})$. Then
\[
p_{2} (M) = - p_{2} (\nu M) = - \Phi^{*} (p_{2} (\xi) \times 1) \ne 0.
\]
By Lemma \ref{lem_Tn_S2} again, $M$ is not a $\mathrm{TOP}$ bundle over $\mathbb{T}^{n}$.
\end{proof}

\section{Noncommutative Fundamental Groups}\label{sec_noncommutative}
In this section, we shall prove Theorems \ref{thm_BS_manifold}, \ref{thm_cw_manifold} and \ref{thm_cw_4-manifold}.

We firstly prove Theorems \ref{thm_cw_manifold} and \ref{thm_cw_4-manifold}. In the proof, a domain in $\mathbb{R}^{m+1}$ inherits the positive orientation of $\mathbb{R}^{m+1}$, the boundary of an oriented manifold carries its boundary orientation.

\begin{proof}[Proof of Theorem \ref{thm_cw_manifold}]
We firstly verify the special case that $X$ is the underlying space of a finite simplicial complex, $\hat{h}$ is a simple homotopy equivalence, and $m \ge \max \{ 2r, r+2, 5 \}$. The idea is to take $M$ as the boundary of a regular neighborhood of $X$ in $\mathbb{R}^{m+1}$.

Since $m \ge 2r$, $X$ can be simplicially embedded in $\mathbb{R}^{m+1}$, we can take a smooth regular neighborhood $N$ of $X$ in $\mathbb{R}^{m+1}$ (see \cite[Theorem~1]{Hirsch}). Then the inclusion $X \hookrightarrow N$ is a simple homotopy equivalence. Let $M$ denote $\partial N$. Then $M$ is a compact smooth hypersurface of $\mathbb{R}^{m+1}$. (See the left part of Figure \ref{fig_neighborhood} for an illustration, where the shadowed domain is $N$, and the line segment inside $N$ stands for $X$.) Note that $M$ is a deformation retract of $N \setminus X$. (Actually, $N \setminus X$ is homeomorphic to $M \times [0, +\infty)$. This can be proved by applying the Corollary 2 of Theorem 8 in \cite[Chapter~3]{Zeeman} to a sequence of nested regular neighborhoods of $X$.) Since $m \geq r+2$, we have $\pi_{1} (M) = \pi_{1} (N) = \pi_{1} (X)$. Here $N$ and $X$ share the same base points, $\pi_{1} (M)$ is identified with $\pi_{1} (X)$ via a chosen path $\gamma$ in $N$ connecting base points.

Let $N'$ be the $k$-cover of $N$ with $\pi_{1} (N') = \pi_{1} (X')$. Let $M'$ denote $\partial N'$. Then $M'$ is also a $k$-cover of $M$ with $\pi_{1} (M') = \pi_{1} (N')$. Here $N'$ and $X'$ share the same base points, the covering map from $M'$ (resp.~$X'$) to $M$ (resp.~$X$) preserves base points, and $\pi_{1} (M')$ is identified with $\pi_{1} (X')$ via a path $\gamma'$ which is a lifting of $\gamma$. (See the right part of Figure \ref{fig_neighborhood} for an illustration, where the shadowed domain is $N'$, and the line segment inside $N'$ stands for $X'$.) Note that $\hat{h} \colon  X \rightarrow X'$ is a simple homotopy equivalence preserving base points. It extends to a simple homotopy equivalence $\hat{h} \colon  N \rightarrow N'$.

\begin{figure}[htbp]
\centering
  \includegraphics[width=0.8\textwidth]{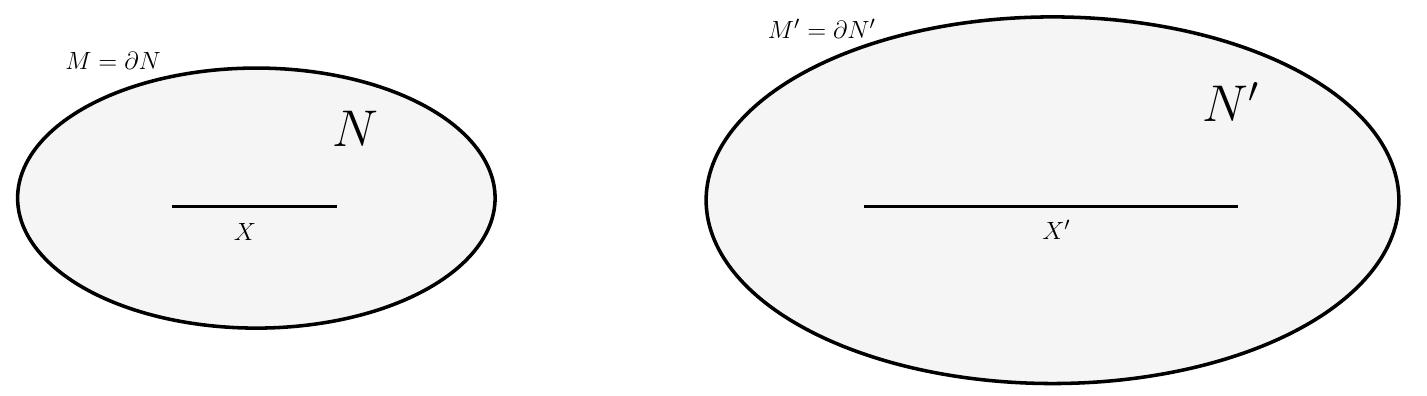}
  \caption{Regular Neighborhood}
  \label{fig_neighborhood}
\end{figure}

The main part of the proof is to construct a diffeomorphism $f \colon  N \rightarrow N'$ which maps the base point of $M$ (resp.~$X$) to that of $M'$ (resp.~$X'$). Furthermore, there is a homotopy between $f|_{X} \colon  X \rightarrow N'$ and $\hat{h}|_{X} \colon  X \rightarrow N'$ which is stationary at the base point.

Clearly, the tangent bundle of $N$ is trivial, so is that of $N'$. Since $N$ has boundary, by Smale-Hirsch immersion theory (see \cite[p.~7,~Theorem~(A)]{Gromov}), $\hat{h}$ is homotopic to an immersion $\phi \colon  N \rightarrow \mathrm{int} N'$, where $\mathrm{int} N'$ is the interior of $N'$ and the homotopy is stationary at the base point of $X$. Furthermore, $\phi$ can be arranged to preserve or reverse the orientations as one wishes. Since $\dim N' \ge 2 \dim X +1$, we may assume $\phi|_{X}$ is an embedding. Thus we can find another smooth regular neighborhood $N_{1}$ of $X$ in $\mathbb{R}^{m+1}$ such that $N_{1} \subseteq \mathrm{int} N$ and $\phi|_{N_{1}}$ is an embedding. (See the left part of Figure \ref{fig_immersion} for an illustration, where the large domain is $N$, the shadowed domain inside $N$ is $N_{1}$.) Since the composition
\[
X \hookrightarrow N_{1} \overset{\phi|_{N_{1}}}{\rightarrow} \phi (N_{1}) \hookrightarrow N'
\]
is a simple homotopy equivalence, so is $\phi (N_{1}) \hookrightarrow N'$.

Let $M_{1}$ denote $\partial N_{1}$, let $M_{2}$ denote $\partial \phi (N_{1})$, let $W = N' \setminus (\mathrm{int} \phi (N_{1}))$. (See the right part of Figure \ref{fig_immersion} for an illustration, where the large domain is $N'$, the shadowed domain inside $N'$ is $\phi (N_{1})$, the line segment inside $\phi (N_{1})$ stands for $\phi (X)$, the unshadowed part inside $N'$ is $W$, and $\partial W = M' \cup M_{2}$. Compare also with Figure \ref{fig_neighborhood}.) Since $\dim N' \ge \dim \phi (X) +3$, we have $\pi_{1} (N') = \pi_{1} (N' \setminus \phi (X))$. Since $\phi (N_{1})$ is a regular neighborhood of $\phi (X)$, $W$ is a deformation retract of $N' \setminus \phi (X)$. We further infer
\[
\pi_{1} (W) = \pi_{1} (N' \setminus \phi (X)) = \pi_{1} (N') = \pi_{1} (M') = \pi_{1} (\phi (N_{1})) = \pi_{1} (M_{2}),
\]
where $W$ and $M'$ share the same base point, $\pi_{1} (W)$ is identified with $\pi_{1} (M_{2})$ via a path in $W$. We also set base points of $M_{1}$ and $M_{2}$ such that $\phi|_{M_{1}} \colon  M_{1} \rightarrow M_{2}$ preserves base points.

\begin{figure}[htbp]
\centering
  \includegraphics[width=0.8\textwidth]{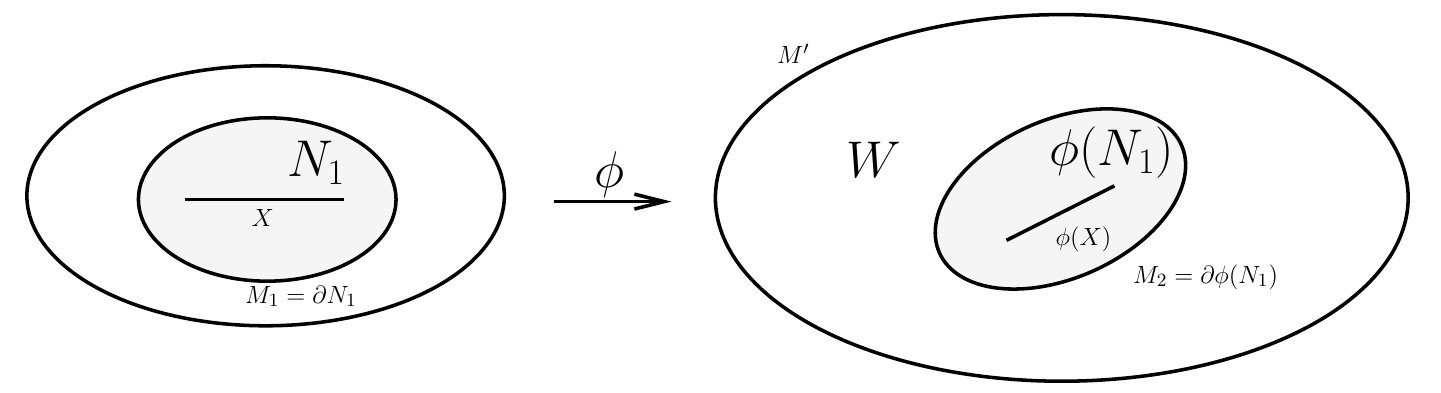}
  \caption{Immersion}
  \label{fig_immersion}
\end{figure}

Recall that $\phi (N_{1}) \hookrightarrow N'$ is a simple homotopy equivalence. By the excision of homology with local coefficient system, $(W; M', M_{2})$ is an h-cobordism. Clearly, $\tau (W, M_{2}) = \tau (N', \phi (N_{1}))$, where $\tau$ is the Whitehead torsion. Hence $(W; M', M_{2})$ is an s-cobordism. Since $\dim W \ge 6$, by the S-Cobordism Theorem, there is a diffeomorphism $\phi_{2} \colon  \phi (N_{1}) \rightarrow N'$ satisfying the following properties: (i) $\phi_{2}$ fixes $\phi (X)$; (ii) $\phi_{2}$ maps the base points of $M_{2}$ to that of $M'$; (iii) $\phi_{2}$ is homotopic to the inclusion $\phi (N_{1}) \hookrightarrow N'$, hence it preserves orientations.

On the other hand, since $N$ and $N_{1}$ are nested smooth regular neighborhoods of $X$, there is also a diffeomorphism (see \cite[Theorem~1]{Hirsch}) $\phi_{1} \colon  N \rightarrow N_{1}$ satisfying the following properties: (i) $\phi_{1}$ fixes $X$; (ii) $\phi_{1}$ maps the base points of $M$ to that of $M_{1}$; (iii) the map $N \overset{\phi_{1}}{\rightarrow} N_{1} \hookrightarrow N$ is homotopic to the identity of $N$, hence $\phi_{1}$ preserves orientations.

Define $f \colon  N \rightarrow N'$ as the composition of diffeomorphisms
\[
N \overset{\phi_{1}}{\rightarrow} N_{1} \overset{\phi|_{N_{1}}}{\rightarrow} \phi (N_{1}) \overset{\phi_{2}}{\rightarrow} N'.
\]
Then $f$ is a diffeomorphism which maps the base point of $M$ (resp.~$X$) to that of $M'$ (resp.~$X'$), and there is a homotopy between $f|_{X} \colon  X \rightarrow N'$ and $\hat{h}|_{X} \colon  X \rightarrow N'$ which is stationary at the base point. Furthermore, $f|_{M} \colon  M \rightarrow M'$ preserves orientations if and only if so does $\phi \colon  N \rightarrow N'$. So $f|_{M}$ can be arranged to preserve or reverse orientations as one wishes.

We would have defined the desired $h \colon  M \rightarrow M'$ as $f|_{M}$. However, to guarantee that $h_{\#} \colon  \pi_{1} (M) \rightarrow \pi_{1} (M')$ equals $\hat{h}_{\#} \colon  \pi_{1} (X) \rightarrow \pi_{1} (X')$, a slight modification is needed. Let $m_{0}$ (resp.~$x_{0}$, $m'_{0}$, $x'_{0}$) denote the base point of $M$ (resp.~$X$, $M'$, $X'$). Since there is a homotopy between $f|_{X} \colon  X \rightarrow N'$ and $\hat{h}|_{X} \colon  X \rightarrow N'$ which is stationary at $x_{0}$,  we have $f_{\#} \colon  \pi_{1} (N, x_{0}) \rightarrow \pi_{1} (N', x'_{0})$ equals $\hat{h}_{\#} \colon  \pi_{1} (X, x_{0}) \rightarrow \pi_{1} (X', x'_{0})$. Here $\pi_{1} (X, x_{0})$ (resp.~$\pi_{1} (X', x'_{0})$) is identified with $\pi_{1} (N, x_{0})$ (resp.~$\pi_{1} (N', x'_{0})$) via the inclusion. (See Figure \ref{fig_neighborhood} for an illustration.) Recall that $\pi_{1} (M, m_{0})$ is identified with $\pi_{1} (N, x_{0})$ via a path $\gamma$ from $m_{0}$ to $x_{0}$. Accordingly, $\pi_{1} (M', m'_{0})$ is identified with $\pi_{1} (N', x'_{0})$ via a path $\gamma'$ in $N'$ which is a lifting of $\gamma$. Note that $\gamma'$ and $f \circ \gamma$ share the same ends. If $\gamma'$ is homotopic to $f \circ \gamma$ relative to the ends, we can define $h= f|_{M}$. If not, we modify $f$ as follows. There is a loop $\alpha$ in $M'$ based at $m'_{0}$ such that $\alpha \cdot (f \circ \gamma)$ is homotopic to $\gamma'$ relative to the ends. By a homotopy pertuabtion relative to the ends, we may assume $f \circ \gamma$ is smoothly embedded and transverse to $M'$. Then a collar neighborhood of $M'$ in $N'$ can be identified with $M' \times [0,1]$, and the intersection of $f \circ \gamma$ with $M' \times [0,1]$ can be identified with $\{ m'_{0} \} \times [0,1]$. We can find a diffeomorphism $\phi_{3} \colon  N' \rightarrow N'$ satisfying: (i) $\phi_{3}$ is isotopic to the identity of $N'$; (ii) $\phi_{3}$ is the identity outside $M' \times [0,1]$; (iii) on $M' \times [0,1]$, $\phi_{3} (m'_{0}, t) = (\beta (t), t)$ such that $\beta$ is a loop representing the same element in $\pi_{1} (M', m'_{0})$ as $\alpha$ does. (Actually, $\phi_{3}|_{M' \times [0,1]}$ represents an isotopy on $M'$ which pushes $m'_{0}$ along $\beta$.) Then $\gamma'$ is homotopic to $\phi_{3} \circ f \circ \gamma$ relative to the ends. We obtain the desired $h= \phi_{3} \circ f|_{M}$, which finishes the proof of the special case.

Finally, let's consider the general case, i.e.~$X$ is only assumed to be finitely dominated. By \cite[Theorem 0.1]{Gersten} and \cite[Theorem~13]{Whitehead}, $X \times S^{3}$ is homotopy equivalent to the underlying space $Y$ of a finite simplicial complex  with $\dim Y = \dim X +3$. The homotopy equivalence $\hat{h} \times \mathrm{id} \colon  X \times S^{3} \rightarrow X' \times S^{3}$ induces a homotopy equivalence $\hat{h}_{Y} \colon  Y \rightarrow Y'$, where $Y'$ is a cover of $Y$ with $\pi_{1} (Y') = \pi_{1} (X')$. Let $Z = Y \times S^{3}$. Then $Z$ is the underlying space of a finite simplicial complex with $\dim Z = \dim X +6$. Let $\hat{h}_{Z} = \hat{h}_{Y} \times \mathrm{id}$. By \cite[(23.2)]{Cohen}, we have $\hat{h}_{Z} \colon  Z \rightarrow Z'$ is a simple homotopy equivalence, where $\pi_{1} (Z) = \pi_{1} (X)$ and $Z'$ is a cover of $Z$ with $\pi_{1} (Z') = \pi_{1} (X')$. Taking the $X$ in the special case as $Z$ now, the conclusion follows.
\end{proof}

\begin{proof}[Proof of Theorem \ref{thm_cw_4-manifold}]
The proof duplicates the main part of that proof of Theorem \ref{thm_cw_manifold}. Let's go back to that proof. We still get an $h$-cobordism (resp.~$s$-cobordism) $(W; M', M_{2})$ if $\hat{h}$ is a homotopy (resp.~simple homotopy) equivalence. However, even if $(W; M', M_{2})$ is an $s$-cobordism, that proof does not go through now because $\dim W =5$ and the usual $S$-Cobordism Theorem is not applicable.

On the other hand, we can still define a homotopy (resp.~simple homotopy) equivalence of pairs  $\phi_{2} \colon  (\phi (N_{1}), M_{2}) \rightarrow (N',M')$ which yields a desired homotopy (resp.~simple homotopy) equivalence $h$. Furthermore, if  $(W; M', M_{2})$ is an $s$-cobordism and $\pi_{1} (W) = \pi_{1} (X)$ is good, by the $S$-Cobordism Theorem in \cite[p.~511]{Freedman_Teichner}, we can arrange $\phi_{2}$ as a homeomorphism which provides a desired homeomorphism $h$.
\end{proof}

Finally, let's prove Theorem \ref{thm_BS_manifold}. Recall the notation $BS(2,3) = \langle a,t \mid t a^{2} t^{-1} = a^{3} \rangle$.
\begin{proposition}[\cite{BDT}]\label{prop_BS_space}
The Eilenberg-MacLane space $K(BS(2,3), 1)$ can be represented by the underlying space $X$ of a $2$-dimensional finite simplicial complex. There are a $5$-cover $X'$ of $X$ with $\pi_{1} (X') = \langle a^{5}, t \rangle \le BS(2,3)$ and a base point preserving homeomorphism $\hat{h} \colon  X \rightarrow X'$ such that $\hat{h}_{\#} (a) = a^{-5}$ and $\hat{h}_{\#} (t) = t$.
\end{proposition}
\begin{proof}
This has been proved in \cite[Example~2.3]{BDT}. We can choose $X$ as the space $K_{2}$ there. In that example, $\pi_{1} (K_{2}) = \langle x,y \mid y^{-1} x^{2} y = x^{3} \rangle$. We take our $a$ and $t$ as $x$ and $y^{-1}$ there, respectively. Our desired $\hat{h}$ follows from the self-covering map $f$ there. (See also there Fig.~2 and Fig.~4 which are very helpful.)
\end{proof}

Let $\mathbb{Z}[\frac{1}{6}] = \{ \frac{m}{6^{n}} \mid m,n \in \mathbb{Z} \}$. Then $\mathbb{Z}[\frac{1}{6}]$ is an abelian group with usual addition. Let $Q(2,3) = \mathbb{Z}[\frac{1}{6}] \rtimes_{\frac{3}{2}} \mathbb{Z}$, where the conjugate action of $\mathbb{Z}$ on $\mathbb{Z}[\frac{1}{6}]$ is given by
\[
(0,1) (x,0) (0,1)^{-1} = (\tfrac{3}{2} x,0).
\]
The following lemma is well-known and easy.

\begin{lemma}\label{lem_BS_quotient}
There exists an epimorphism $\varphi \colon  BS(2,3) \rightarrow Q(2,3)$ such that $\varphi (a) = (1,0)$ and $\varphi (t) = (0,1)$.
\end{lemma}

\begin{proposition}\label{prop_Q(2,3)_present}
$Q(2,3)$ is not finitely presented.
\end{proposition}
\begin{proof}
Neither $\tfrac{2}{3}$ nor $\tfrac{3}{2}$ is an algebraic integer. Taking the $N$ in \cite[Theorem~C]{Bieri_Strebel} as $\mathbb{Z}[\frac{1}{6}]$, this conclusion follows from the (i)$\Leftrightarrow$(iv) of that theorem. (See also \cite[Lemma~11]{Groves}.)
\end{proof}

\begin{lemma}\label{lem_Q(2,3)_hom}
There exists a monomorphism $f \colon  Q(2,3) \rightarrow Q(2,3)$ such that $f(x,y) = (-5x,y)$ and $[Q(2,3): \mathrm{im} f] =5$. Furthermore, $\bigcap_{k=1}^{\infty} \mathrm{im} f^{k} = 0 \rtimes \mathbb{Z}$.
\end{lemma}
\begin{proof}
By a straightforward calculation, we see $f$ is a well-defined monomorphism and $\mathrm{im} f = 5 \mathbb{Z}[\frac{1}{6}] \rtimes \mathbb{Z}$.

To prove $[Q(2,3): \mathrm{im} f] =5$, it suffices to show $[\mathbb{Z}[\frac{1}{6}]: 5 \mathbb{Z}[\frac{1}{6}]] =5$. By
\[
\tfrac{1}{6^{n}} - \tfrac{1}{6^{n-1}} = -\tfrac{5}{6^{n}},
\]
we infer $\tfrac{m}{6^{n}} + 5 \mathbb{Z}[\frac{1}{6}] = m + 5 \mathbb{Z}[\frac{1}{6}]$ and hence $[\mathbb{Z}[\frac{1}{6}]: 5 \mathbb{Z}[\frac{1}{6}]] \leq 5$. Since $5$ is prime to $6$, we see integers $i$, for $0 \leq i \leq 4$, are in distinct cosets of $\mathbb{Z}[\frac{1}{6}] / 5 \mathbb{Z}[\frac{1}{6}]$. Thus $[\mathbb{Z}[\frac{1}{6}]: 5 \mathbb{Z}[\frac{1}{6}]] =5$.

The fact $\bigcap_{k=1}^{\infty} \mathrm{im} f^{k} = 0 \rtimes \mathbb{Z}$ is also because $5$ is prime to $6$.
\end{proof}

\begin{lemma}\label{lem_coset_index}
Suppose $\psi \colon G_{1} \rightarrow G_{2}$ is an epimorphism of groups, $H_{i} \le G_{i}$, $\psi (H_{1}) = H_{2}$ and $[G_{2} : H_{2}] < \infty$. Then $[G_{1} : H_{1}] = [G_{2} : H_{2}]$ if and only if $\psi^{-1} (H_{2}) = H_{1}$.
\end{lemma}
\begin{proof}
Let $H_{3} = \psi^{-1} (H_{2})$, then $H_{1} \le H_{3}$. Since $\psi$ is surjective, it induces a bijection $G_{1} / H_{3} \rightarrow G_{2} / H_{2}$. So
\[
[G_{1} : H_{1}] \ge [G_{1} : H_{3}] = [G_{2} : H_{2}].
\]
It's easy to see the conclusion holds.
\end{proof}

\begin{lemma}\label{lem_BS_diagram}
The following diagram commutes.
\[
\xymatrix{
  BS(2,3) \ar[d]_{\varphi} \ar[r]^{\hat{h}_{\#}} & BS(2,3) \ar[d]^{\varphi} \\
  Q(2,3) \ar[r]^{f} & Q(2,3)   }
\]
Furthermore, for all integers $k>0$, $\varphi (\mathrm{im} \hat{h}_{\#}^{k} )= \mathrm{im} f^{k}$, and $\ker \varphi \le \bigcap_{k=1}^{\infty} \mathrm{im} \hat{h}_{\#}^{k}$.
\end{lemma}
\begin{proof}
The commutativity is trivial. By Lemma \ref{lem_BS_quotient}, $\varphi$ is surjective, $\varphi \mathrm{im} \hat{h}_{\#}^{k} = \mathrm{im} f^{k}$. By Proposition \ref{prop_BS_space} and Lemma \ref{lem_Q(2,3)_hom},
\[
[BS(2,3): \mathrm{im} \hat{h}_{\#}] = [Q(2,3): \mathrm{im} f] =5.
\]
By Lemma \ref{lem_coset_index}, we have $\varphi^{-1} (\mathrm{im} f) = \mathrm{im} \hat{h}_{\#}$. Thus $\ker \varphi \le \mathrm{im} \hat{h}_{\#}$.

Suppose $x \in \ker \varphi$. Then there is a $y \in  BS(2,3)$ such that $x = \hat{h}_{\#} (y)$. By the above diagram, $f \varphi (y) = \varphi \hat{h}_{\#} (y) = \varphi (x) =0$. By Lemma \ref{lem_Q(2,3)_hom}, $f$ is injective, we see $\varphi (y) = 0$, i.e. $y \in \ker \varphi \le \mathrm{im} \hat{h}_{\#}$. Thus $x \in \mathrm{im} \hat{h}_{\#}^{2}$. By induction, we infer $\ker \varphi \le \bigcap_{k=1}^{\infty} \mathrm{im} \hat{h}_{\#}^{k}$.
\end{proof}

\begin{proposition}\label{prop_BS_intersection}
We have $\bigcap_{k=1}^{\infty} \mathrm{im} \hat{h}_{\#}^{k} = \varphi^{-1} (0 \rtimes \mathbb{Z}) = \ker \varphi \rtimes \langle t \rangle$, and it is not a normal subgroup of $BS(2,3)$.
\end{proposition}
\begin{proof}
By Lemmas \ref{lem_Q(2,3)_hom} and \ref{lem_BS_diagram}, we have
\[
\varphi \left( \bigcap_{k=1}^{\infty} \mathrm{im} \hat{h}_{\#}^{k} \right) \le \bigcap_{k=1}^{\infty} \varphi (\mathrm{im} \hat{h}_{\#}^{k} )= \bigcap_{k=1}^{\infty} \mathrm{im} f^{k} = 0 \rtimes \mathbb{Z}.
\]
Clearly, $\langle t \rangle \le \bigcap_{k=1}^{\infty} \mathrm{im} \hat{h}_{\#}^{k}$ and $\varphi (\langle t \rangle) = 0 \rtimes \mathbb{Z}$. Thus $\varphi (\bigcap_{k=1}^{\infty} \mathrm{im} \hat{h}_{\#}^{k}) = 0 \rtimes \mathbb{Z}$.

By Lemma \ref{lem_BS_diagram}, $\ker \varphi \le \bigcap_{k=1}^{\infty} \mathrm{im} \hat{h}_{\#}^{k}$, we infer $\bigcap_{k=1}^{\infty} \mathrm{im} \hat{h}_{\#}^{k} = \varphi^{-1} (0 \rtimes \mathbb{Z})$. Since $\ker \varphi$ is normal in $BS(2,3)$, $\varphi|_{\langle t \rangle}$ is injective, and $\varphi \langle t \rangle = 0 \rtimes \mathbb{Z}$, we see $\varphi^{-1} (0 \rtimes \mathbb{Z}) = \ker \varphi \rtimes \langle t \rangle$. Since $0 \rtimes \mathbb{Z}$ is not normal in $Q(2,3)$, nor is $\bigcap_{k=1}^{\infty} \mathrm{im} \hat{h}_{\#}^{k}$ in $BS(2,3)$.
\end{proof}

\begin{proposition}\label{prop_infinite_rank}
The group $\bigcap_{k=1}^{\infty} \mathrm{im} \hat{h}_{\#}^{k}$ is a free group with infinite rank.
\end{proposition}
\begin{proof}
Taking the $(x,y)$ in \cite[Proposition~2.1]{Farrell_Wu} as $(0,1)$, we have $\varphi^{-1} (0 \rtimes \mathbb{Z})$ is free. By Proposition \ref{prop_BS_intersection}, $\bigcap_{k=1}^{\infty} \mathrm{im} \hat{h}_{\#}^{k}$ is free.

It remains to show $\bigcap_{k=1}^{\infty} \mathrm{im} \hat{h}_{\#}^{k}$ is not finitely generated. We argue this by contradiction. If not, by Proposition \ref{prop_BS_intersection}, we would have that $\bigcap_{k=1}^{\infty} \mathrm{im} \hat{h}_{\#}^{k}$ can be generated by $g_{1} t^{r_{1}}, \dots, g_{n} t^{r_{n}}$, where each $g_{i} \in \ker \varphi$. Thus it would be generated by $g_{1}, \dots, g_{n}, t$. We further infer $\ker \varphi$ would be generated by
\[
\{ t^{k} g_{i} t^{-k} \mid 1 \leq i \leq n, k \in \mathbb{Z} \}.
\]
It's easy to see $\ker \varphi$ would be the normal closure of $\{ g_{1}, \dots, g_{n} \}$ in $BS(2,3)$. Note that $\varphi$ is surjective by Lemma \ref{lem_BS_quotient}. Since $BS(2,3)$ is finitely presented, we would have so is $Q(2,3)$, which is a contradiction to Proposition \ref{prop_Q(2,3)_present}.
\end{proof}

\begin{proof}[Proof of Theorem \ref{thm_BS_manifold}]
Since the $\hat{h} \colon  X \rightarrow X'$ in Proposition \ref{prop_BS_space} is a homeomorphism, it is a simple homotopy equivalence by \cite[Theorem~1]{Chapman}. This theorem follows from Theorems \ref{thm_cw_manifold} and \ref{thm_cw_4-manifold} together with Propositions \ref{prop_BS_space}, \ref{prop_BS_intersection} and \ref{prop_infinite_rank}.
\end{proof}

\begin{remark}\label{rmk_BS_proof}
We now give more details of Remark \ref{rmk_BS_general}. By \cite[Corollary~5.4]{Howie} and \cite[Theorem~13]{Whitehead}, $K(BS(q_{1}, q_{2}), 1)$ can be represented by the underlying space $X$ of a $2$-dimensional finite simplicial complex. By a purely group theoretic argument, it can be shown that there is a monomorphism $\theta \colon  BS(q_{1}, q_{2}) \rightarrow BS(q_{1}, q_{2})$ such that $\theta (a) = a^{s}$ or $\theta (a) = a^{-s}$ as one wishes, $\theta (t) = t$, and $[BS(q_{1}, q_{2}) : \mathrm{im} \theta] = s$. We have a base points preserving homotopy equivalence $\hat{h} \colon  X \rightarrow X'$ such that $\hat{h}_{\#} = \theta$, where $X'$ is the cover of $X$ with $\pi_{1} (X') = \mathrm{im} \theta$. By \cite{Farrell_Wu}, $BS(q_{1}, q_{2})$ satisfies the Farrell-Jones conjecture for $K$--theory. Thus the Whitehead group $\mathrm{Wh} (BS(q_{1}, q_{2}))$ is trivial (see e.g.~\cite[Remark~2.2]{Luck}). We infer $\hat{h}$ is a simple homotopy equivalence. The remaining argument duplicates that of Theorem \ref{thm_BS_manifold}.
\end{remark}

\section{Further Questions}\label{sec_question}
We firstly propose questions in order to seek positive solutions, probably under some extra assumptions, to Question \ref{que_nonabelian}.

To obtain a result similar to Theorem \ref{thm_Poincare}, one has to show the homotopy fiber of the $p$ in \eqref{que_nonabelian_1} is finitely dominated. Therefore, it's desire to extend Theorem \ref{thm_homotopy_finite} to the case of general fundamental groups.

Suppose $X$ is a CW complex, $q \colon  X' \rightarrow X$ a finite nontrivial covering, and $h \colon  X \rightarrow X'$ is a homotopy equivalence preserving base points. For brevity, we abuse the notation $h_{\#}$ to denote the composition
\[
(qh)_{\#} \colon  \pi_{1} (X) \rightarrow \pi_{1} (X') \hookrightarrow \pi_{1} (X).
\]
\begin{question}\label{que_homotopy_finite}
Let $X$, $X'$ and $h$ be as the above. Let $G: = \bigcap_{k=1}^{\infty} \mathrm{im} h_{\#}^{k}$ and $\overline{X}$ be the cover of $X$ with $\pi_{1} (\overline{X}) = G$. Assume further $G$ is a finitely presented normal subgroup of $\pi_{1} (X)$. If $X$ is homotopy equivalent to a CW complex of finite type (resp.~is finitely dominated), then is so $\overline{X}$?
\end{question}

By Theorem \ref{thm_BS_manifold}, the answer to Question \ref{que_homotopy_finite} will be no if $G$ is not assumed finitely presented. As a special case, when $\pi_{1} (X) / G \cong \mathbb{Z}$, Question \ref{que_homotopy_finite} becomes the Question 7.1 in \cite{Qin_Su_Wang}.

We also need to find a suitable aspherical manifold $B$ in \eqref{que_nonabelian_1}. It's necessary to point out, to ensure the homotopy fiber of $p$ is homotopy equivalent to the cover $\overline{M}$ of $M$ with $\pi_{1} (\overline{M}) = \ker p_{\#}$, the manifold $B$ has to be aspherical. If the square \eqref{que_nonabelian_1} was a homotopy pullback with $\bigcap_{k=1}^{\infty} \mathrm{im} \psi_{\#}^{k} =1$, we would have $\ker p_{\#} = G$ and $\pi_{1} (B) = \pi_{1} (M) / G$.
\begin{question}\label{que_Poincare}
Let $X$ and $G$ be as in Question \ref{que_homotopy_finite}. Let $Q: = \pi_{1} (X) / G$. Is the Eilenberg-MacLane space $K(Q,1)$ a finitely dominated Poincar\'{e} duality space?
\end{question}

Assuming $X$ is a finitely dominated Poincar\'{e} space, Questions \ref{que_homotopy_finite} and \ref{que_Poincare} are closely related. In this situation, by \cite{Gottlieb} (see also \cite[Theorem~G]{Klein_Qin_Su}), if the answer to Question \ref{que_homotopy_finite} is yes, and if $K(Q,1)$ is finitely dominated, then both $K(Q,1)$ and $\overline{X}$ are automatically Poincar\'{e} spaces.

Question \ref{que_Poincare} is also directly related to a few questions in the literature. Note that $h_{\#} \colon  \pi_{1} (X) \rightarrow \pi_{1} (X)$ induces a monomorphism $[h_{\#}] \colon  Q \rightarrow Q$ such that $[Q: \mathrm{im} [h_{\#}]] < \infty$ and $\bigcap_{k=1}^{\infty} \mathrm{im} [h_{\#}]^{k} =1$. Thus $Q$ is both a strongly scale-invariant group in \cite{Nekrashevych_Pete} and a dis-cohopfian group in \cite{Cornulier}.

\begin{definition}[\cite{Nekrashevych_Pete}]
Suppose $\Gamma$ is a group. If there is a monomorphism $\phi \colon  \Gamma \rightarrow \Gamma$ such that $[\Gamma: \mathrm{im} \phi] < \infty$ and $\bigcap_{k=1}^{\infty} \mathrm{im} \phi^{k}$ is finite, then $\Gamma$ is \textit{strongly scale-invariant}.
\end{definition}

\begin{definition}[\cite{Cornulier}]
Suppose $\Gamma$ is a group. If there is a monomorphism $\phi \colon  \Gamma \rightarrow \Gamma$ such that $\bigcap_{k=1}^{\infty} \mathrm{im} \phi^{k} = 1$, then $\Gamma$ is \textit{dis-cohopfian}.
\end{definition}

Motivated by Benjamini, Nekrashevych-Pete asked the following (\cite[Question~1.1]{Nekrashevych_Pete}):
\begin{question}[Benjamini-Nekrashevych-Pete]\label{que_scale_invariant}
If $\Gamma$ is a finitely generated strongly scale-invariant group, then does $\Gamma$ have polynomial growth, or equivalently, is $\Gamma$ virtually nilpotent?
\end{question}

To the best of our knowledge, Question \ref{que_scale_invariant} remains open in general. If $\Gamma$ is further assumed virtually polycyclic, it has been confirmed by Der\'{e} in \cite[Theorem~3.10]{Dere22}. Note that the above $Q$ is also finitely generated and hence satisfies the condition in Question \ref{que_scale_invariant}. Therefore, any progress on Question \ref{que_scale_invariant} will be definitely helpful for our Question \ref{que_Poincare}. Actually, it's known that $K(\pi, 1)$ is a finitely dominated Poincar\'{e} spaces if $\pi$ is torsion free and virtually nilpotent (cf. \cite[p.~596]{Johnson_Wall}).

We aim to enhance $K(Q,1)$ to make it a closed manifold.
\begin{question}\label{que_manifold}
Let $Q$ and $[h_{\#}]$ be as above. Can $K(Q,1)$ be represented by a closed $\mathrm{CAT}$ manifold $B$? And is there a base points preserving $\mathrm{CAT}$ covering map $\psi \colon  B \rightarrow B$ such that $\psi_{\#} = [h_{\#}]$?
\end{question}

Our Question \ref{que_manifold} is related to two famous conjectures (see e.g. Conjectures 1.1 and 1.2 in \cite{Luck}).
\begin{conjecture}\label{cjt_Poincare}
Every finitely dominated aspherical Poincar\'{e} spaces is homotopy equivalent to a closed $\mathrm{TOP}$ manifold.
\end{conjecture}

\begin{conjecture}[Borel Conjecture]\label{cjt_Borel}
Suppose $B_{1}$ and $B_{2}$ are closed aspherical $\mathrm{TOP}$ manifolds. Then every homotopy equivalence $h \colon  B_{1} \rightarrow B_{2}$ is homotopic to a homeomorphism.
\end{conjecture}

Conjectures \ref{cjt_Poincare} and \ref{cjt_Borel} are neither proved nor disproved, but are confirmed for many types of fundamental groups. If these two conjectures are true, and if the answer to Question \ref{que_Poincare} is yes, then so is the answer to Question \ref{que_manifold} in the case of $\mathrm{TOP}$. Furthermore, our $Q$ is rather special. Very likely, these conjectures are true for $Q$.

If, probably under some mild extra assumptions, the answers to our Questions \ref{que_homotopy_finite}, \ref{que_Poincare} and \ref{que_manifold} are all yes, then one will obtain a positive solution to Question \ref{que_nonabelian} in a weak sense. In other words, there will be the desired \eqref{que_nonabelian_1} except that $p$ is merely a continuous map whose homotopy fiber is a finitely dominated Poincar\'{e} space. By the argument in Section \ref{sec_fibering}, imposing additional conditions, the $p$ can be further promoted to fibering in various senses.

To conclude this paper, we raise questions concerning the diversity of closed self-covering manifolds.

It's interesting to modify the example in Theorem \ref{thm_BS_manifold} into a closed aspherical manifold.
\begin{question}\label{que_aspherical}
Are there a closed aspherical $\mathrm{CAT}$ manifold (resp.~finitely dominated aspherical Poincar\'{e} space) $M$ and a base points preserving $\mathrm{CAT}$ isomorphism (resp.~homotopy equivalence) $h \colon  M \rightarrow M'$ such that $\bigcap_{k=1}^{\infty} \mathrm{im} h_{\#}^{k}$ is not finitely presented? Or even not finitely generated? Here $M'$ is a nontrivial cover of $M$.
\end{question}

Since Conjectures \ref{cjt_Poincare} and \ref{cjt_Borel} have been confirmed quite widely, a positive solution to Question \ref{que_aspherical} in the case of Poincar\'{e} spaces would probably produce examples in the case of $\mathrm{TOP}$ manifolds.

It's known every finitely presented group can be realized as the fundamental group of a closed $\mathrm{DIFF}$ manifold. Our final question concerns the realization of a finitely presented group, which is isomorphic to a proper finite-index subgroup of itself, as the fundamental group of a closed self-covering manifold.
\begin{question}\label{que_group_manifold}
Suppose $\Gamma$ is a finitely presented group. Suppose $\Gamma'$ is a subgroup of $\Gamma$ such that $\Gamma' \cong \Gamma$ and $1 < [\Gamma: \Gamma'] < \infty$. Is there a closed $\mathrm{DIFF}$ manifold $M$ such that $\pi_{1} (M) = \Gamma$ and $M$ is diffeomorphic to $M'$? Here $M'$ is the cover of $M$ with $\pi_{1} (M') = \Gamma'$.
\end{question}

If $\Gamma$ is abelian, the answer to Question \ref{que_group_manifold} is trivially yes. In this typical situation, $\Gamma = \mathbb{Z}^{n} \times \Gamma_{0}$ and $\Gamma' = H \times \Gamma_{0}$, where $H \le \mathbb{Z}^{n}$ and $\Gamma_{0}$ is a finite group. We can construct $M$ as $\mathbb{T}^{n} \times M_{0}$, where $M_{0}$ is a closed $\mathrm{DIFF}$ manifold with $\pi_{1} (M_{0}) = \Gamma_{0}$.

By Theorem \ref{thm_cw_manifold}, Question \ref{que_group_manifold} is equivalent to the following one whose positive answer is apparently, but not actually, weaker.
\begin{question}\label{que_group_space}
Let $\Gamma$ and $\Gamma'$ be as in Question \ref{que_group_manifold}. Is there a finitely dominated CW complex $X$ such that $\pi_{1} (X) = \Gamma$ and $X$ is homotopy equivalent to $X'$? Here $X'$ is the cover of $X$ with $\pi_{1} (X') = \Gamma'$.
\end{question}

Particularly, if $K(\Gamma, 1)$ is finitely dominated, then the answers to Questions \ref{que_group_manifold} and \ref{que_group_space} are positive.



\begin{thebibliography}{1234}

\bibitem{Arkowitz} M. Arkowitz, Introduction to homotopy theory, Universitext, Springer, New York, 2011.

\bibitem{Atiyah_Hirzebruch} M. Atiyah and F. Hirzebruch, Vector bundles and homogeneous spaces, Proc. Sympos. Pure Math., Vol. III, pp. 7-38, American Mathematical Society, Providence, RI, 1961.

\bibitem{Atiyah_Macdonald} M. Atiyah and I. Macdonald, Introduction to commutative algebra, Addison-Wesley Publishing Co., 1969.

\bibitem{BDT} T. Bedenikovic, A. Delgado and M. Timm, A classification of 2-complexes with nontrivial self-covers, Topology Appl. 153 (2006), no. 12, 2073-2091.

\bibitem{Bieri_Strebel} R. Bieri and R. Strebel, Almost finitely presented soluble groups, Comment. Math. Helv. 53 (1978), no. 2, 258-278.

\bibitem{Browder66} W. Browder, Open and closed disk bundles, Ann. of Math. (2) 83 (1966), 218-230.

\bibitem{Burghelea_Lashof} D. Burghelea and R. Lashof, The homotopy type of the space of diffeomorphisms, II., Trans. Amer. Math. Soc. 196 (1974), 37-50.

\bibitem{BLR1975} D. Burghelea, R. Lashof, and M. Rothenberg, Groups of automorphisms of manifolds, Lecture Notes in Math., Vol. 473, Springer-Verlag, Berlin-New York, 1975.

\bibitem{Cappell} S. Cappell, A splitting theorem for manifolds, Invent. Math. 33 (1976), no. 2, 69-170.

\bibitem{Cartan_Eilenberg} H.~Cartan and S.~Eilenberg, Homological algebra, Princeton University Press, Princeton, NJ, 1956.

\bibitem{Cerf61} J. Cerf, Topologie de certains espaces de plongements, Bull. Soc. Math. France 89 (1961), 227-380.

\bibitem{Chapman} T. Chapman, Topological invariance of Whitehead torsion, Amer. J. Math. 96 (1974), 488-497.

\bibitem{Cohen} M. Cohen, A course in simple-homotopy theory, Grad. Texts in Math., Vol. 10, Springer-Verlag, New York-Berlin, 1973.

\bibitem{Coram_Duvall1} D. Coram and P. Duvall, Approximate fibrations, Rocky Mountain J. Math. 7 (1977), no. 2, 275-288.

\bibitem{Coram_Duvall2} D. Coram and P. Duvall, Approximate fibrations and a movability condition for maps, Pacific J. Math. 72 (1977), no. 1, 41-56.

\bibitem{Cornulier} Y. Cornulier, Gradings on Lie algebras, systolic growth, and cohopfian properties of nilpotent groups, Bull. Soc. Math. France 144 (2016), no. 4, 693-744.

\bibitem{Dekimpe} K. Dekimpe, What an infra-nilmanifold endomorphism really should be $\dots$, Topol. Methods Nonlinear Anal. 40 (2012), no. 1, 111-136.

\bibitem{Dere17} J. Der\'{e}, Gradings on Lie algebras with applications to infra-nilmanifolds, Groups Geom. Dyn. 11 (2017), no. 1, 105-120.

\bibitem{Dere22} J. Der\'{e}, Strongly scale-invariant virtually polycyclic groups, Groups Geom. Dyn. 16 (2022), no. 3, 985-1004.

\bibitem{Ebert_Randal} J. Ebert and O. Randal-Williams, Generalised Miller-Morita-Mumford classes for block bundles and topological bundles, Algebr. Geom. Topol. 14 (2014), no. 2, 1181-1204.

\bibitem{Eckmann_Linnell} B.~Eckmann and P.~Linnell, Poincar\'{e} duality groups of dimension two. II, Comment. Math. Helv. 58 (1983), no. 1, 111-114.

\bibitem{Eckmann_Muller} B.~Eckmann and H.~M\"{u}ller, Poincar\'{e} duality groups of dimension two, Comment. Math. Helv. 55 (1980), no. 4, 510-520.

\bibitem{Farrell_Gogolev} F.~T.~Farrell and A.~Gogolev, Examples of expanding endomorphisms on fake tori, J. Topol. 7 (2014), no. 3, 805-816.

\bibitem{Farrell_Hsiang} F.~T.~Farrell and W.~C.~Hsiang, The topological-Euclidean space form problem, Invent. Math. 45 (1978), no. 2, 181-192.

\bibitem{Farrell_Jones78} F.~T.~Farrell and L.~Jones, Examples of expanding endomorphisms on exotic tori, Invent. Math. 45 (1978), no. 2, 175~179.

\bibitem{FLS2018} F.~T.~Farrell, W. L\"{u}ck, and W. Steimle, Approximately fibering a manifold over an aspherical one, Math. Ann. 370 (2018), no. 1-2, 669-726.

\bibitem{Farrell_Wu} F.~T.~Farrell and X. Wu, Isomorphism conjecture for Baumslag-Solitar groups, Proc. Amer. Math. Soc. 143 (2015), no. 8, 3401-3406.

\bibitem{Ferry} S. Ferry, Approximate fibrations with nonfinite fibers, Proc. Amer. Math. Soc. 64 (1977), no. 2, 335-345.

\bibitem{Freedman_Quinn} M. Freedman and F. Quinn, Topology of 4-manifolds, Princeton Mathematical Series, 39, Princeton University Press, Princeton, NJ, 1990.

\bibitem{Freedman_Teichner} M. Freedman and P. Teichner, $4$-manifold topology I: subexponential groups, Invent. Math. 122 (1995), no. 3, 509-529.

\bibitem{Friberg} B. Friberg, A topological proof of a theorem of Kneser, Proc. Amer. Math. Soc. 39 (1973), 421-426.

\bibitem{Gersten} S. Gersten, A product formula for Wall's obstruction, Amer. J. Math. 88 (1966), 337-346.

\bibitem{Goncalves_Spreafico} D. Gon\c{c}alves and M. Spreafico, The fundamental group of the space of maps from a surface into the projective plane, Math. Scand. 104 (2009), no. 2, 161-181.

\bibitem{Gottlieb} H. Gottlieb, Poincar\'{e} duality and fibrations, Proc. Amer. Math. Soc. 76 (1979), no. 1, 148-150.

\bibitem{Gromov} M. Gromov, Partial differential relations, Springer-Verlag, Berlin, 1986.

\bibitem{Groves} J. Groves, Soluble groups in which every finitely generated subgroup is finitely presented, J. Austral. Math. Soc. Ser. A 26 (1978), no. 1, 115-125.

\bibitem{Hansen} V. Hansen, The homotopy groups of a space of maps between oriented closed surfaces, Bull. London Math. Soc. 15 (1983), no. 4, 360-364.

\bibitem{Hartshorne} R. Hartshorne, Agebraic geometry, Grad. Texts in Math., No. 52, Springer-Verlag, New York-Heidelberg, 1977.

\bibitem{Hatcher} A. Hatcher, Concordance spaces, higher simple-homotopy theory, and applications, Algebraic and geometric topology (Proc. Sympos. Pure Math., Stanford Univ., Stanford, Calif., 1976), Part 1, pp. 3-21, American Mathematical Society, Providence, RI, 1978.

\bibitem{Hatcher83} A. Hatcher, A proof of the Smale conjecture, $\mathrm{Diff} (S^{3}) \simeq \mathrm{O}(4)$, Ann. of Math. (2) 117 (1983), no. 3, 553-607.

\bibitem{Hillman} J. Hillman, Four-manifolds, geometries and knots, Geom. Topol. Monogr., 5, Geometry \& Topology Publications, Coventry, 2002.

\bibitem{Hirsch} M. Hirsch, Smooth regular neighborhoods, Ann. of Math. (2) 76 (1962), 524-530.

\bibitem{Howie} J. Howie, On locally indicable groups, Math. Z. 180 (1982), no. 4, 445-461.

\bibitem{HTW90} C. Hughes, L. Taylor and E. Williams, Bundle theories for topological manifolds, Trans. Amer. Math. Soc. 319 (1990), no. 1, 1-65.

\bibitem{Johnson_Wall} F. Johnson and C.~T.~C. Wall, On groups satisfying Poincar\'{e} duality, Ann. of Math. (2) 96 (1972), 592-598.

\bibitem{Kirby_Siebenmann} R. Kirby and L. Siebenmann, Foundational essays on topological manifolds, smoothings, and triangulations, Ann. of Math. Stud., No. 88, Princeton University Press, Princeton, NJ; University of Tokyo Press, Tokyo, 1977.

\bibitem{Klein_Qin_Su} J. Klein, L. Qin and Y. Su, On the various notions of Poincar\'{e} duality pair, Trans. Amer. Math. Soc. 375 (2022), no. 6, 4251-4283.

\bibitem{Kneser} H. Kneser, Die Deformationss\"{a}tze der einfach zusammenh\"{a}ngenden Fl\"{a}chen, Math. Z. 25 (1926), no. 1, 362-372.

\bibitem{Loh} C.~L\"{o}h, Geometric group theory, An introduction, Universitext, Springer, Cham, 2017.

\bibitem{Luck} W. L\"{u}ck, $K$- and $L$-theory of group rings, Proceedings of the International Congress of Mathematicians, Hyderabad, India, 2010.

\bibitem{Luft_Sjerve} E. Luft and D. Sjerve, $3$-manifolds with subgroups $Z \oplus Z \oplus Z$ in their fundamental groups, Pacific J. Math. 114 (1984), no. 1, 191-205.

\bibitem{Maclagan_Sturmfels} D. Maclagan and B. Sturmfels, Introduction to tropical geometry, Grad. Stud. Math., 161, American Mathematical Society, Providence, RI, 2015.

\bibitem{Matsumura1} H. Matsumura, Commutative algebra, Second edition, Mathematics Lecture Note Series, 56. Benjamin/Cummings Publishing Co., Inc., Reading, Mass., 1980.

\bibitem{Matsumura2} H. Matsumura, Commutative ring theory, Second edition, Cambridge Studies in Advanced Mathematics, 8. Cambridge University Press, Cambridge, 1989.

\bibitem{MH} J. Milnor and D. Husemoller, Symmetric bilinear forms, Ergebnisse der Mathematik und ihrer Grenzgebiete, Band 73. Springer-Verlag, New York-Heidelberg, 1973.

\bibitem{Milnor_Stasheff} J. Milnor and J. Stasheff, Characteristic Classes, Princeton University Press, 1974.

\bibitem{Morgan} J. Morgan, A product formula for surgery obstructions, Mem. Amer. Math. Soc. 14 (1978), no. 201.

\bibitem{Mumford} D. Mumford, The red book of varieties and schemes, Second, expanded edition, Lecture Notes in Math., 1358, Springer-Verlag, Berlin, 1999.

\bibitem{Nekrashevych_Pete} V. Nekrashevych and G. Pete, Scale-invariant groups, Groups Geom. Dyn. 5 (2011), no. 1, 139-167.

\bibitem{Nishimura} J. Nishimura, Note on integral closures of a Noetherian integral domain, J. Math. Kyoto Univ. 16 (1976), no. 1, 117-122.

\bibitem{Novikov} S. Novikov, Topological invariance of rational classes of Pontrjagin, Dokl. Akad. Nauk SSSR 163 (1965), 298-300.

\bibitem{Olum} P. Olum, Mappings of manifolds and the notion of degree, Ann. of Math. (2) 58 (1953), 458-480.

\bibitem{Perron} B. Perron, Pseudo-isotopies et isotopies en dimension quatre dans la cat\'{e}gorie topologique, Topology 25 (1986), no. 4, 381-397.

\bibitem{Qin_Su_Wang} L. Qin, Y. Su and B. Wang, Self-Covering, finiteness, and fibering over a circle, Trans. Amer. Math. Soc. 377 (2024), no. 3, 1883-1914, arXiv: 2112.11750.

\bibitem{Quinn79} F. Quinn, Ends of maps, I., Ann. of Math. (2) 110 (1979), no. 2, 275-331.

\bibitem{Quinn86} F. Quinn, Isotopy of 4-manifolds, J. Differential Geom. 24 (1986), no. 3, 343-372.

\bibitem{Ranicki} A. Ranicki, Algebraic and geometric surgery, Oxford Mathematical Monographs, Oxford Science Publications, The Clarendon Press, Oxford University Press, Oxford, 2002.

\bibitem{Rosenlicht} M. Rosenlicht, Some rationality questions on algebraic groups, Ann. Mat. Pura Appl. (4) 43 (1957), 25-50.

\bibitem{Rourke_Sanderson68I} C. Rourke and B. Sanderson, Block bundles: I, Ann. of Math. (2) 87 (1968), 1-28.

\bibitem{Rourke_Sanderson68II} C. Rourke and B. Sanderson, Block bundles: II. Transversality, Ann. of Math. (2) 87 (1968), 256-278.

\bibitem{Rourke_Sanderson68III} C. Rourke and B. Sanderson, Block bundles: III. Homotopy theory, Ann. of Math. (2) 87 (1968), 431-483.

\bibitem{Rourke_Sanderson71} C. Rourke and B. Sanderson, $\Delta$-sets II: Block bundles and block fibrations, Quart. J. Math. Oxford Ser. (2) 22 (1971), 465-485.

\bibitem{Selick} P. Selick, Introduction to homotopy theory, Fields Inst. Monogr., 9, American Mathematical Society, Providence, RI, 1997.

\bibitem{Thomas} C. Thomas, Elliptic structures on $3$-manifolds, London Math. Soc. Lecture Note Ser., 104, Cambridge University Press, Cambridge, 1986.

\bibitem{VanLimbeek} W. Van Limbeek, Towers of regular self-covers and linear endomorphisms of tori, Geom. Topol. 22 (2018), no. 4, 2427-2464.

\bibitem{Wall65} C.~T.~C. Wall, Finiteness conditions for CW-complexes, Ann. of Math. (2) 81 (1965), 56-69.

\bibitem{Wall67} C.~T.~C. Wall, Poincar\'{e} Complexes: I, Annals of Mathematics, Vol. 86, No. 2 (1967), 213-245.

\bibitem{Wang_Wu} S. Wang and Y. Wu, Covering invariants and co-Hopficity of 3-manifold groups, Proc. London Math. Soc. (3) 68 (1994), no. 1, 203-224.

\bibitem{G.Whitehead} G. Whitehead, Elements of homotopy theory, Graduate Texts in Mathematics 61, Springer-Verlag, 1978.

\bibitem{Whitehead} J.~H.~C. Whitehead, Combinatorial homotopy. I, Bull. Amer. Math. Soc. 55 (1949), 213-245.

\bibitem{Yamanoshita} T. Yamanoshita, On the space of self-homotopy equivalences of the projective plane, J. Math. Soc. Japan 45 (1993), no. 3, 489-494.

\bibitem{Zeeman} E.~C. Zeeman, Seminar on combinatorial topology, Institut Des Hautes Etudes Scientifiques, 1963.

\end{thebibliography}
\end{document}